\documentclass[oneside,english,american]{amsart}
\usepackage[T1]{fontenc}
\usepackage[latin9]{inputenc}
\usepackage{xcolor}
\usepackage{float}
\usepackage{amstext}
\usepackage{amsthm}
\usepackage{amssymb}
\usepackage{graphicx}
\PassOptionsToPackage{normalem}{ulem}
\usepackage{ulem}

\makeatletter

\floatstyle{ruled}
\newfloat{algorithm}{tbp}{loa}
\providecommand{\algorithmname}{Algorithm}
\floatname{algorithm}{\protect\algorithmname}
\providecolor{lyxadded}{rgb}{0,0,1}
\providecolor{lyxdeleted}{rgb}{1,0,0}

\DeclareRobustCommand{\lyxdeleted}[3]{{\color{lyxdeleted}\sout{#3}}}

\numberwithin{equation}{section}
\numberwithin{figure}{section}
\theoremstyle{plain}
\newtheorem{thm}{\protect\theoremname}[section]
  \theoremstyle{plain}
  \newtheorem{cor}[thm]{\protect\corollaryname}
  \theoremstyle{remark}
  \newtheorem{rem}[thm]{\protect\remarkname}
  \theoremstyle{plain}
  \newtheorem{lem}[thm]{\protect\lemmaname}

\usepackage[T1]{fontenc}
\usepackage{ae,aecompl}

\makeatother

\usepackage{babel}
  \addto\captionsamerican{\renewcommand{\corollaryname}{Corollary}}
  \addto\captionsamerican{\renewcommand{\lemmaname}{Lemma}}
  \addto\captionsamerican{\renewcommand{\remarkname}{Remark}}
  \addto\captionsamerican{\renewcommand{\theoremname}{Theorem}}
  \addto\captionsenglish{\renewcommand{\corollaryname}{Corollary}}
  \addto\captionsenglish{\renewcommand{\lemmaname}{Lemma}}
  \addto\captionsenglish{\renewcommand{\remarkname}{Remark}}
  \addto\captionsenglish{\renewcommand{\theoremname}{Theorem}}
  \providecommand{\corollaryname}{Corollary}
  \providecommand{\lemmaname}{Lemma}
  \providecommand{\remarkname}{Remark}
\addto\captionsenglish{\renewcommand{\algorithmname}{Algorithm}}
\providecommand{\theoremname}{Theorem}

\begin{document}

\title[On computing distributions of products of random variables]{On computing distributions of products of random variables via Gaussian
multiresolution analysis  }

\author{Gregory Beylkin, Lucas Monz\'{o}n and Ignas Satkauskas }

\address{Department of Applied Mathematics \\
 University of Colorado at Boulder \\
 UCB 526 \\
 Boulder, CO 80309-0526 }
\begin{abstract}
We introduce a new approximate multiresolution analysis (MRA) using
a single Gaussian as the scaling function, which we call Gaussian
MRA (GMRA). As an initial application, we employ this new tool to
accurately and efficiently compute the probability density function
(PDF) of the product of independent random variables. In contrast
with Monte-Carlo (MC) type methods (the only other universal approach
known to address this problem), our method not only achieves accuracies
beyond the reach of MC but also produces a PDF expressed as a Gaussian
mixture, thus allowing for further efficient computations. We also
show that an exact MRA corresponding to our GMRA can be constructed
for a matching user-selected accuracy.
\end{abstract}


\thanks{This research was partially supported by NSF grant DMS-1320919.}

\maketitle

\section{Introduction}

While the probability density function (PDF) of the sum of two independent
random variables is easily described as the convolution of their PDFs,
the expression for the PDF of the product is significantly more complicated.
Given two independent random variables $X$ and $Y$ with PDFs $f$
and $g$, where $\int_{\mathbb{R}}f(x)dx=1$ and $\int_{\mathbb{R}}g(y)dy=1$,
the PDF $p$ of their product, $Z=XY$, can be succinctly written
as 
\begin{equation}
p\left(t\right)=\int_{\mathbb{R}}\int_{\mathbb{R}}f\left(x\right)g\left(y\right)\delta\left(xy-t\right)dxdy,\label{eq:integral-to-compute-1}
\end{equation}
where $\delta$ is the delta function defined as 
\[
\int_{-\infty}^{\infty}\delta\left(t-t'\right)u\left(t'\right)dt'=u\left(t\right)
\]
for a class of test functions $u$.When written in a more traditional
form,
\[
p\left(t\right)=\int_{\mathbb{R}}f\left(x\right)g\left(t/x\right)\frac{1}{\left|x\right|}dx,
\]
the difficulty of computing $p$ becomes apparent since the integral
in \eqref{eq:integral-to-compute-1} diverges at $t=0$ as long as
e.g. $g\left(0\right)\ne0$ and $f\left(0\right)\ne0$. Indeed, in
the simplest case when independent random variables are selected from
two normal distributions with zero means, the PDF $p$ has a logarithmic
singularity at $t=0$ (see \eqref{eq:K_0 distribution} below).

As far as we know, the only universal method currently available for
computing the PDF of the product of two independent random variables
relies on a Monte-Carlo approach, where one samples the individual
PDFs, computes the products, and collects enough samples to achieve
certain accuracy in the computation of $p$ in \eqref{eq:integral-to-compute-1}.
However, due to the slow convergence of such method (typically $1/\sqrt{N}$,
where $N$ is the number of samples) achieving high accuracy is not
feasible.

In this paper we show how to accurately and efficiently compute the
PDF $p$ in \eqref{eq:integral-to-compute-1} using an approximate
multiresolution analysis (MRA), where the scaling function is a Gaussian.
We demonstrate that, in contrast with the sampling method, our new
algorithm is both fast and accurate. Moreover, while a sampling method
only provides a histogram of $p$, we directly obtain the result in
a functional form that can be used in further computations. Thus,
we avoid the need for e.g. kernel density estimation of the result
of sampling the product PDF (see e.g. \cite{EFROMO:2008,SIMONO:2012}
and references therein). In particular, we do not have any obstacles
dealing with heavy-tailed distributions (see an example with the Cauchy
distribution in Section~\ref{sub:ExampleCauchy}). 

A representation of a function in Gaussian-based MRA (GMRA) can also
be viewed as a multiresolution Gaussian mixture. Approximations by
Gaussian mixtures have been used early on, see e.g.~\cite{SOR-ALS:1971}.
We prove that the integral (\ref{eq:integral-to-compute-1}) can be
efficiently evaluated as a multiresolution Gaussian mixture and provide
a tight estimate for the accuracy of the result. 

Our approach to constructing a GMRA is based on the observation that,
for any finite user-selected accuracy $\epsilon$, the function $\phi\left(x\right)=\sqrt{\frac{\alpha}{\pi}}e^{-\alpha x^{2}}$
satisfies the approximate two-scale relation 
\[
\left|\phi\left(x\right)-\sqrt{\frac{4\alpha}{3\pi}}\sum_{k\in\mathbb{Z}}e^{-\frac{\alpha}{3}k^{2}}\phi\left(2x-k\right)\right|\le\epsilon\phi\left(x\right)
\]
for an appropriately selected parameter $\alpha=\alpha\left(\epsilon\right)$.
As a consequence, the scaled and shifted functions 
\[
\left\{ \phi_{j,k}\left(x\right)=2^{j/2}\phi\left(2^{j}x-k\right)=2^{j/2}\sqrt{\frac{\alpha}{\pi}}e^{-\alpha\left(2^{j}x-k\right)^{2}}\right\} _{j,k\in\mathbb{Z}}
\]
form a Gaussian-based approximate MRA. We use this approximate basis
to represent PDFs of random variables which we assume to be smooth
functions except at a finite number of points where they may have
integrable singularities. Under this assumption these PDFs can be
represented by a reasonable number of basis functions. The approximate
nature of the GMRA does not limit its applicability since any finite
accuracy can be selected. For example, in our computations we roughly
match the double precision arithmetic of a typical computer by selecting
$\epsilon$ in the range $10^{-13}-10^{-16}$.

The integral (\ref{eq:integral-to-compute-1}) can, in principle,
be computed using the Mellin transform. It has been observed that,
for non-negative random variables, the Mellin transform of $p$ in
(\ref{eq:integral-to-compute-1}) is equal to the product of the Mellin
transforms of $f$ and $g$ (equivalently, $p$ is the so-called Mellin
convolution of $f$ and $g$). In e.g.~\cite{SPRING:1979}, this
approach has been extended to random variables taking both positive
and negative values. While analytically appealing, a numerical implementation
of Mellin convolution is problematic and has not resulted in a reliable
numerical method. 

Instead of using the Mellin transform to represent the PDFs of non-negative
random variables, we can use linear combinations of decaying exponentials
with complex exponents constructed via algorithms developed in \cite{BEY-MON:2005,HAU-BEY:2012}.
Such representation of PDFs of non-negative random variables leads
to a universal method for computing their product. We plan to address
such approach in a separate paper. However, for random variables that
can take both positive and negative values, it is more natural to
use multiresolution Gaussian mixtures as described in this paper.
Moreover, a generalization to higher dimensions becomes problematic
if one were to split the treatment of positive and negative values
of such random variables.

Computing the PDF of the product of independent random variables using
GMRA is just the first of several interesting applications of this
approximate basis. Although the idea of using Gaussians to approximate
arbitrary functions appeared earlier in \cite{LA-MA-SC:2004,MAZ-SCH:2007},
where the coefficients of the shifted Gaussians were associated with
function values it turns out that much more accurate approximations
can be achieved if the coefficients are computed differently, as described
later in this paper. Since many classical operators of mathematical
physics can be accurately approximated via a linear combination of
Gaussians (see e.g. \cite{H-F-Y-G-B:2004,BE-CH-PE:2008,BEY-MON:2010,B-F-H-K-M:2012}),
using GMRA to apply these operators can be reduced to the analytic
evaluation of integrals, and hence, suggests a new class of fast multiresolution
algorithms. However, constructing such algorithms is not straightforward
since integer translates of Gaussians are far from being orthogonal
(after all, they are positive functions so that no cancellations are
possible). Therefore, the use of GMRA differs from the usual MRA with
orthogonal or nearly orthogonal scaling functions and we briefly discuss
some of these issues in Appendix B.

We present our basis construction for functions of one variable and
note that the generalization to two and three variables is fairly
straightforward and we will describe it in appropriate contexts. Extending
our approach to bases for multivariate functions in higher dimensions
is more complicated and we plan to address it later. 

Finally, we note that the quotient of two independent random variables
can be evaluated in the same fashion as the product. Since the PDF
of the sums (or differences) of two real valued random variables is
obtained by the convolution of their PDFs, the representation provided
by GMRA yields a numerical calculus of PDFs. In particular, the expectation
\[
\mbox{E}\left[u\left(Z\right)\right]=\int u\left(t\right)p_{Z}\left(t\right)dt,
\]
where $u$ is a given function and $p_{Z}$ is the PDF of random variable
$Z$ represented in GMRA, can be easily evaluated (see e.g. Remarks~\ref{rem: moments remark}
and \ref{rem:expectation}).

We start by proving a key approximation result and considering its
consequences in Section~\ref{sec:Approximation-of-shifted}. In Section~\ref{sec:Distribution-of-the}
we prove a relative error estimate for the representation of the PDF
of the product via GMRA when one of the distributions is Gaussian.
We also present two approaches to approximating PDFs via a linear
combination of Gaussians as well as demonstrate that the product of
two random variables with a joint bivariate normal distribution can
also be handled by our approach. We then turn to numerical examples
in Section~\ref{sec:Numerical-Examples}. Conclusions and further
work is discussed in Section~\ref{sec:Conclusions-and-further}.
In Appendix~A we include a pseudo-code for adaptive integration used
in Section~\ref{sec:Distribution-of-the}. Finally, we present a
brief discussion of certain features of GMRA in Appendix B.

\section{\label{sec:Approximation-of-shifted}Approximation of shifted Gaussians
and approximate multiresolution analysis}

We start by proving a key estimate (Theorem~\ref{thm: projection of a Gaussian}
below) that allows us to replace any shifted Gaussian by an explicit
linear combination of Gaussians (with a larger, fixed exponent) shifted
at integer values. This estimate yields an approximate two-scale relation
for Gaussians that makes possible to construct an approximate multiresolution
basis with Gaussians as the scaling functions. Since the basis of
shifted Gaussians is far from being orthogonal, one would expect difficulties
in computing projections of functions on such basis. Remarkably, Theorem~\ref{thm: projection of a Gaussian}
and Theorem~\ref{thm:The-Gaussian-on arbitrary scale} allow us to
compute such projections accurately for any Gaussian with an arbitrary
exponent and shift. Since we are interested in applications where
all functions and the results of computing integrals can be approximated
via linear combinations of Gaussians, the following result is the
key to using GMRA as a numerical tool. 
\begin{thm}
\label{thm: projection of a Gaussian}If $0<\beta<\alpha$ then, for
$x,s\in\mathbb{R}$, we have

\begin{eqnarray}
\label{eq:key approximation}\\
\left|e^{-\beta\left(x-s\right)^{2}}-\frac{\alpha}{\sqrt{\pi\left(\alpha-\beta\right)}}\sum_{k\in\mathbb{Z}}e^{-\frac{\alpha\beta}{\alpha-\beta}\left(k-s\right)^{2}}e^{-\alpha\left(x-k\right)^{2}}\right| & \le & \epsilon\left(\alpha,\beta\right)e^{-\beta\left(x-s\right)^{2}},\nonumber 
\end{eqnarray}
where
\[
\epsilon\left(\alpha,\beta\right)=\vartheta_{3}\left(0,e^{-\frac{\pi^{2}}{\alpha}\left(1-\frac{\beta}{\alpha}\right)}\right)-1
\]
and 
\[
\vartheta_{3}\left(z,q\right)=\sum_{n\in\mathbb{Z}}q^{n^{2}}e^{2inz}
\]
is the Jacobi theta function. \end{thm}
\begin{proof}
We have
\[
e^{-\beta\left(x-s\right)^{2}}=\frac{\alpha}{\sqrt{\pi\left(\alpha-\beta\right)}}\int_{\mathbb{R}}e^{-\frac{\alpha\beta}{\alpha-\beta}\left(t-s\right)^{2}}e^{-\alpha\left(x-t\right)^{2}}dt,
\]
and need to estimate
\[
\int_{\mathbb{R}}F\left(t,s,x\right)dt-\sum_{k\in\mathbb{Z}}F\left(k,s,x\right),
\]
where
\[
F\left(t,s,x\right)=\frac{\alpha}{\sqrt{\pi\left(\alpha-\beta\right)}}e^{-\frac{\alpha\beta}{\alpha-\beta}\left(t-s\right)^{2}}e^{-\alpha\left(x-t\right)^{2}}.
\]
Using Poisson\textquoteright s summation formula, we have
\begin{equation}
\left|\int_{\mathbb{R}}F\left(t,s,x\right)dt-\sum_{k\in\mathbb{Z}}F\left(k,s,x\right)\right|\le\sum_{n=\pm1,\pm2,\dots}\left|\widehat{F}\left(n,s,x\right)\right|,\label{eq:estimate from revisited}
\end{equation}
where
\[
\widehat{F}\left(n,s,x\right)=\int_{\mathbb{R}}F\left(t,s,x\right)e^{-2\pi itn}dt=e^{-n^{2}\frac{\pi^{2}}{\alpha}\left(1-\frac{\beta}{\alpha}\right)}e^{-\beta\left(x-s\right)^{2}}e^{-2\pi in\frac{\beta}{\alpha}\left(s-x\right)-2\pi inx}.
\]
We obtain 
\begin{eqnarray*}
\sum_{n=\pm1,\pm2,\dots}\left|\widehat{F}\left(n,s,x\right)\right| & = & \left(\sum_{n=\pm1,\pm2,\dots}e^{-n^{2}\frac{\pi^{2}}{\alpha}\left(1-\frac{\beta}{\alpha}\right)}\right)e^{-\beta\left(x-s\right)^{2}}\\
 & \le & \left(\vartheta_{3}\left(0,e^{-\frac{\pi^{2}}{\alpha}\left(1-\frac{\beta}{\alpha}\right)}\right)-1\right)e^{-\beta\left(x-s\right)^{2}}
\end{eqnarray*}
yielding the estimate \eqref{eq:key approximation}. For similar type
of analysis see \cite[Section 2, eq.5]{BEY-MON:2010}, where we use
the step size to control the accuracy of approximating integrals by
a sum, whereas here we fix the step size but select the exponent of
the Gaussian $\phi$ to achieve the desired accuracy.\end{proof}
\begin{cor}
\label{cor:app_two_scale_eq}Defining $\phi\left(x\right)=\sqrt{\frac{\alpha}{\pi}}e^{-\alpha x^{2}}$
and assuming that $0<\beta\le\alpha/4$, we obtain from \eqref{eq:key approximation}

\begin{equation}
\left|e^{-\beta\left(x-s\right)^{2}}-\sqrt{\frac{\alpha}{\alpha-\beta}}\sum_{k\in\mathbb{Z}}e^{-\frac{\alpha\beta}{\alpha-\beta}\left(k-s\right)^{2}}\phi\left(x-k\right)\right|\le\epsilon e^{-\beta\left(x-s\right)^{2}},\label{eq:key approximation-1}
\end{equation}
where 
\begin{equation}
\epsilon=\vartheta_{3}\left(0,e^{-\frac{3\pi^{2}}{4\alpha}}\right)-1.\label{eq:parameter epsilon}
\end{equation}
Also, setting $s=0$, $\beta=\alpha/4$ and replacing $x$ by $2x$
in \eqref{eq:key approximation}, we obtain
\begin{equation}
\left|\phi\left(x\right)-\sqrt{\frac{4\alpha}{3\pi}}\sum_{k\in\mathbb{Z}}e^{-\frac{\alpha}{3}k^{2}}\phi\left(2x-k\right)\right|\le\epsilon\phi\left(x\right),\label{eq:approx_two_scale}
\end{equation}
an approximate two-scale relation for Gaussians. 
\end{cor}
Consequently, we can use the function $\phi$ as a scaling function
for an approximate multiresolution basis. By selecting an appropriate
$\alpha$ we can reduce the value of $\epsilon$ in \eqref{eq:parameter epsilon}
and achieve arbitrary finite precision in \eqref{eq:key approximation}.
In the examples presented later in the paper, we use $\alpha=1/4$
so that $\epsilon\approx2.77\cdot10^{-13}$ (see Figures~\ref{fig:example0101-2}
and \ref{fig:example010101-2}). We note that selecting $\alpha=1/5$
yields $\epsilon\approx2.22\cdot10^{-16}$ achieving double precision
accuracy. 
\begin{rem}
If $\beta=0$ in \eqref{eq:key approximation}, we recover the approximation
described in \cite{MAZ-SCH:2007}. However, if $\beta\ne0$, our approximation
\eqref{eq:key approximation} is substantially different and significantly
more accurate than that obtained in \cite{MAZ-SCH:2007}. The key
difference, of course, is in the approach to the selection of coefficients
of the shifted Gaussians.
\end{rem}
Both, in Theorem~\ref{thm: projection of a Gaussian} and Corollary~\ref{cor:app_two_scale_eq},
the error is described in terms of the quantity $\vartheta_{3}\left(0,e^{-\gamma}\right)-1$
for some appropriate $\gamma>0$. In the next lemma we show that,
as a function of $\gamma$, $\vartheta_{3}\left(0,e^{-\gamma}\right)-1$
decreases exponentially fast. Thus, slightly increasing the value
of $\gamma$ provides a much improved error. Also, using the lemma,
for a given $\alpha$ and target accuracy $\epsilon,$ we can accurately
estimate the maximum value of $\beta$ to use in Corollary~\ref{cor:app_two_scale_eq}.
\begin{lem}
\label{lem:Estimate_Theta_at_zero} 

For $\gamma\geq\gamma_{0}>0$, 
\[
\vartheta_{3}\left(0,e^{-\gamma}\right)-1\leq e^{\gamma_{0}}\left(\vartheta_{3}\left(0,e^{-\gamma_{0}}\right)-1\right)\thinspace e^{-\gamma}.
\]
In particular, for $\gamma_{0}=\pi$, we have
\begin{equation}
\vartheta_{3}\left(0,e^{-\gamma}\right)-1\leq c_{\pi}\thinspace e^{-\gamma},\ \mbox{for}\ \gamma\geq\pi,\label{ineq_v0-1}
\end{equation}
and
\begin{equation}
\sqrt{\frac{\pi}{\gamma}}<\vartheta_{3}\left(0,e^{-\gamma}\right)\leq\sqrt{\frac{\pi}{\gamma}}\left(1+c_{\pi}\thinspace e^{-\frac{\pi^{2}}{\gamma}}\right),\ \mbox{for}\ \gamma\leq\pi,\label{ineq_v0}
\end{equation}
where
\begin{equation}
c_{\pi}=e^{\pi}\left(\frac{\sqrt[4]{\pi}}{\Gamma\left(\frac{3}{4}\right)}-1\right)<2.002.\label{constant_c_pi}
\end{equation}
\end{lem}
\begin{proof}
Consider 
\[
f\left(\gamma\right)=e^{\gamma}\left(\vartheta_{3}\left(0,e^{-\gamma}\right)-1\right)=2\sum_{n\geq1}e^{-\gamma\left(n^{2}-1\right)}.
\]
Since $f'\left(\gamma\right)<0$, $f$ is a decreasing function which
implies the first claim. For \eqref{ineq_v0-1}, use \cite[page 325]{BERNDT:1998}
to write $f\left(\pi\right)=e^{\pi}\left(\vartheta_{3}\left(0,e^{-\pi}\right)-1\right)=e^{\pi}\left(\frac{\sqrt[4]{\pi}}{\Gamma\left(\frac{3}{4}\right)}-1\right)<2.002$.
For \eqref{ineq_v0}, note that 
\begin{equation}
\vartheta_{3}\left(0,e^{-\gamma}\right)=\sqrt{\frac{\pi}{\gamma}}\vartheta_{3}\left(0,e^{-\frac{\pi^{2}}{\gamma}}\right)=\sqrt{\frac{\pi}{\gamma}}\sum_{n\in\mathbb{Z}}e^{-\frac{\pi^{2}}{\gamma}n^{2}}>\sqrt{\frac{\pi}{\gamma}}.\label{theta_at_0_using_reciprocal}
\end{equation}
Since for $\gamma\leq\pi$ we have $\frac{\pi^{2}}{\gamma}\geq\pi$,
we obtain the result using \eqref{ineq_v0-1} with $\gamma$ replaced
by $\frac{\pi^{2}}{\gamma}$.
\end{proof}

\subsection{Multiresolution analysis with a Gaussian as the scaling function}

We consider the family of functions 
\begin{equation}
\left\{ \phi_{j,k}\left(x\right)=2^{j/2}\phi\left(2^{j}x-k\right)=2^{j/2}\sqrt{\frac{\alpha}{\pi}}e^{-\alpha\left(2^{j}x-k\right)^{2}}\right\} _{j,k\in\mathbb{Z}}\label{eq:Gaussian multiresolution basis}
\end{equation}
normalized so that the $L^{2}$-norm 
\[
\left\Vert \phi_{j,k}\right\Vert _{2}=\left(\frac{\alpha}{2\pi}\right)^{1/4}
\]
does not depend on the scale $j$ (\textit{n.b}., setting the norm
to $1$ would have resulted in bulkier formulas). In the next theorem,
we show that a Gaussian with an arbitrary exponent can be projected
(approximated) at an appropriate scale of GMRA. It is remarkable that
the projection coefficients are obtained explicitly without solving
any equations.
\begin{thm}
\label{thm:The-Gaussian-on arbitrary scale}The Gaussian $e^{-\beta\left(x-s\right)^{2}}$
with exponent $\beta$ and shift $s$\textup{ can be approximated
in the GMRA }\eqref{eq:Gaussian multiresolution basis}\textup{ as
}
\begin{equation}
\left|e^{-\beta\left(x-s\right)^{2}}-\sum_{k\in\mathbb{Z}}g_{k}^{j}\phi_{jk}\left(x\right)\right|\le\epsilon e^{-\beta\left(x-s\right)^{2}},\label{eq:expansion on scale j}
\end{equation}
where \textup{$4^{j-2}\alpha<\beta\le4^{j-1}\alpha$ and} 
\begin{equation}
g_{k}^{j}=2^{-j/2}\sqrt{\frac{\alpha}{\pi\left(\alpha-4^{-j}\beta\right)}}e^{-\frac{\alpha\beta}{\alpha-4^{-j}\beta}\left(2^{-j}k-s\right)^{2}}.\label{eq:coefficients on scale j}
\end{equation}
\end{thm}
\begin{proof}
Given $\beta$, we select the scale $j$ so that $4^{j-2}\alpha<\beta\le4^{j-1}\alpha$.
Once $j$ is selected, we rescale the exponent $\beta$ and the shift
$s$ as $\widetilde{\beta}=4^{-j}\beta$ and $\widetilde{s}=2^{j}s$
and compute coefficients using \eqref{eq:key approximation} for the
rescaled Gaussian $e^{-\widetilde{\beta}\left(x-\widetilde{s}\right)^{2}}$,
\[
\left|e^{-\widetilde{\beta}\left(x-\widetilde{s}\right)^{2}}-\frac{\alpha}{\sqrt{\pi\left(\alpha-\widetilde{\beta}\right)}}\sum_{k\in\mathbb{Z}}e^{-\frac{\alpha\widetilde{\beta}}{\alpha-\widetilde{\beta}}\left(k-\widetilde{s}\right)^{2}}e^{-\alpha\left(x-k\right)^{2}}\right|\le\epsilon e^{-\widetilde{\beta}\left(x-\widetilde{s}\right)^{2}}.
\]
Replacing the variable $x$ by $2^{j}x$ in the expression above,
we arrive at \eqref{eq:expansion on scale j}.
\end{proof}

\section{\label{sec:Distribution-of-the}Distribution of the product of independent
random variables}

To motivate our approach, we first consider two independent Gaussian
variables $X\thicksim N\left(\mu_{x},\sigma_{x}^{2}\right)$, $Y\thicksim N\left(\mu_{y},\sigma_{y}^{2}\right)$,
with $\mu_{x}=\mu_{y}=0$. The distribution of their product variable,
$Z=XY,$ is given by

\begin{equation}
p_{00}\left(t\right)=\frac{1}{\pi\sigma_{x}\sigma_{y}}K_{0}\left(\frac{\left|t\right|}{\sigma_{x}\sigma_{y}}\right),\label{eq:K_0 distribution}
\end{equation}
where $K_{0}$ is modified Bessel function of the second kind. Hence,
the PDF $p_{00}$ has an integrable logarithmic singularity at zero
even though the normal distributions do not have any singularities.
In general, the contribution of such singularity to the cumulative
distribution function (CDF) may or may not be significant but it is
always present as it follows from the representation of $p$ in \eqref{eq:integral-to-compute-1}.
Specifically, for $p_{00}\left(t\right)$ in \eqref{eq:K_0 distribution},
as $t$ approaches $0$ we have 
\[
p_{00}\left(t\right)=\frac{1}{\pi\sigma_{x}\sigma_{y}}\left(-\gamma+\log\left(2\sigma_{x}\sigma_{y}\right)-\log\left|t\right|\right)+\mathcal{O}\left(t^{2}\log\left|t\right|\right),
\]
where $\gamma$ is the Euler constant. For three or more Gaussian
PDFs with zero means, their product can be expressed via Meijer G-functions
(see \cite[Definition 9.301]{GRA-RYZ:2007}). However, when the Gaussian
PDFs have non-zero means, we are not aware of any explicit expression
for the PDF of their product in terms of special functions. 

We start by computing the PDF of the product of two independent random
variables, $X$ with PDF $f\left(x\right)$, $\int_{\mathbb{R}}f\left(x\right)dx=1$,
and the Gaussian variable $Y\thicksim N\left(\mu_{y},\sigma_{y}^{2}\right)$
with PDF 
\[
g\left(y\right)=\frac{1}{\sqrt{2\pi}\sigma_{y}}e^{-\frac{(y-\mu_{y})^{2}}{2\sigma_{y}^{2}}}.
\]
The PDF of their product, $Z=XY$, is expressed via the integral 
\begin{equation}
p\left(t\right)=\frac{1}{\sqrt{2\pi}\sigma_{y}}\int_{-\infty}^{\infty}{\displaystyle \int_{-\infty}^{\infty}f\left(x\right)e^{-\frac{(y-\mu_{y})^{2}}{2\sigma_{y}^{2}}}\delta\left(t-xy\right)dxdy}.\label{eq:Definition of the PDF}
\end{equation}
We assume that the PDF $f$ is a smooth function except at zero and
a finite number of points where it may have integrable singularities.
Specifically, $f$ may have an integrable logarithmic singularity
at $x=0$, $f\left(x\right)\sim\left(\log\left|x\right|\right)^{m}$,
for some integer $m\ge1$. We also assume a mild decay of $f$ as
$x\to\infty$, 
\[
\left|f\left(x\right)\right|\le\frac{C_{0}}{\left|x\right|^{2}},\,\,\,\mbox{for}\,\,\,\left|x\right|>C_{1},
\]
where $C_{0}$ and $C_{1}$ are positive constants. In particular,
$f$ can be the Cauchy distribution.

In the next theorem we obtain a multiresolution approximation of $p\left(t\right)$,
where the coefficients at all scales are given explicitly by integrals
of well-behaved functions over a unit interval. These integrals are
evaluated numerically via adaptive integration, see Section~\ref{sub:Adaptive-method-for}
below.
\begin{thm}
\label{thm:The-PDF-of the product}The PDF of the product of two independent
random variables $X$ and the Gaussian variable $Y\thicksim N\left(\mu_{y},\sigma_{y}^{2}\right)$
in \eqref{eq:Definition of the PDF} can be approximated as

\begin{equation}
\left|p\left(t\right)-\sum_{j\in\mathbb{Z}}\sum_{k\in\mathbb{Z}}w_{k}^{j}\phi_{jk}\left(t\right)\right|\le\epsilon p\left(t\right)\label{eq:multiresolution expansion of PDF}
\end{equation}
where $\epsilon$ is given in \eqref{eq:parameter epsilon},
\begin{equation}
w_{k}^{j}=2^{-j/2}\log2\int_{0}^{1}\frac{w_{+}\left(\tau\right)+w_{-}\left(\tau\right)}{\sqrt{1-4^{\tau-2}}}d\tau,\label{eq:Coefficients of MR expansion of PDF}
\end{equation}
 
\[
w_{+}\left(\tau\right)=\frac{1}{\sqrt{2\pi}\sigma_{y}}f\left(\frac{2^{2-j}}{\sqrt{2\alpha}\sigma_{y}}2^{-\tau}\right)e^{-\frac{\alpha4^{\tau-2}}{1-4^{\tau-2}}\left(k-\frac{\mu_{y}}{\sqrt{2\alpha}\sigma_{y}}2^{-\tau+2}\right)^{2}},
\]
and 
\[
w_{-}\left(\tau\right)=\frac{1}{\sqrt{2\pi}\sigma_{y}}f\left(-\frac{2^{2-j}}{\sqrt{2\alpha}\sigma_{y}}2^{-\tau}\right)e^{-\frac{\alpha4^{\tau-2}}{1-4^{\tau-2}}\left(k+\frac{\mu_{y}}{\sqrt{2\alpha}\sigma_{y}}2^{-\tau+2}\right)^{2}}.
\]
\end{thm}
\begin{proof}
Splitting the integration over $x$ in \eqref{eq:Definition of the PDF},
we have 

\begin{eqnarray*}
p(t) & = & \frac{1}{\sqrt{2\pi}\sigma_{y}}\left(\int_{0}^{\infty}dx\int_{-\infty}^{\infty}f\left(x\right)e^{-\frac{(y-\mu_{y})^{2}}{2\sigma_{y}^{2}}}\delta(xy-t)dy\right.\\
 & + & \left.\int_{-\infty}^{0}dx{\displaystyle \int_{-\infty}^{\infty}f\left(x\right)e^{-\frac{(y-\mu_{y})^{2}}{2\sigma_{y}^{2}}}\delta(xy-t)dy}\right)
\end{eqnarray*}
 and changing $x\rightarrow-x$ in the second integral, obtain
\begin{eqnarray*}
p(t) & = & \frac{1}{\sqrt{2\pi}\sigma_{y}}\left(\int_{0}^{\infty}dx\int_{-\infty}^{\infty}f\left(x\right)e^{-\frac{(y-\mu_{y})^{2}}{2\sigma_{y}^{2}}}\delta(xy-t)dy\right.\\
 & + & \left.\int_{0}^{\infty}dx{\displaystyle \int_{-\infty}^{\infty}f\left(-x\right)e^{-\frac{(y-\mu_{y})^{2}}{2\sigma_{y}^{2}}}\delta(xy+t)dy}\right).
\end{eqnarray*}
Since $\delta(xy)=\frac{1}{x}\delta(y)$ for $x>0$, we have
\[
p(t)=\frac{1}{\sqrt{2\pi}\sigma_{y}}\int_{0}^{\infty}\left(f\left(x\right)e^{-\frac{(\frac{t}{x}-\mu_{y})^{2}}{2\sigma_{y}^{2}}}+f\left(-x\right)e^{-\frac{(\frac{t}{x}+\mu_{y})^{2}}{2\sigma_{y}^{2}}}\right)\frac{1}{x}dx.
\]
Changing variables $x=\frac{1}{\sqrt{2\alpha}\sigma_{y}}2^{-s}$,
we obtain
\begin{eqnarray}
p(t) & = & c\int_{-\infty}^{\infty}\left(f\left(\frac{2^{-s}}{\sqrt{2\alpha}\sigma_{y}}\right)e^{-\alpha4^{s}\left(t-\frac{\mu_{y}}{\sqrt{2\alpha}\sigma_{y}}2^{-s}\right)^{2}}\right.\nonumber \\
 & + & \left.f\left(-\frac{2^{-s}}{\sqrt{2\alpha}\sigma_{y}}\right)e^{-\alpha4^{s}\left(t+\frac{\mu_{y}}{\sqrt{2\alpha}\sigma_{y}}2^{-s}\right)^{2}}\right)ds,\label{eq:before splitting into unit intervals}
\end{eqnarray}
where $c=\log\left(2\right)/\left(\sqrt{2\pi}\sigma_{y}\right)$.
Splitting the integration in \eqref{eq:before splitting into unit intervals}
into integrals over unit intervals, we rewrite
\begin{eqnarray*}
p(t) & = & c\sum_{j\in\mathbb{Z}}\int_{j-2}^{j-1}\left(f\left(\frac{2^{-s}}{\sqrt{2\alpha}\sigma_{y}}\right)e^{-\alpha4^{s}\left(t-\frac{\mu_{y}}{\sqrt{2\alpha}\sigma_{y}}2^{-s}\right)^{2}}\right.\\
 & + & \left.f\left(-\frac{2^{-s}}{\sqrt{2\alpha}\sigma_{y}}\right)e^{-\alpha4^{s}\left(t+\frac{\mu_{y}}{\sqrt{2\alpha}\sigma_{y}}2^{-s}\right)^{2}}\right)ds
\end{eqnarray*}
and changing the variable of integration, $s=\tau+j-2$, we arrive
at 
\begin{eqnarray}
p(t) & = & c\sum_{j\in\mathbb{Z}}\int_{0}^{1}\left[f\left(\frac{2^{-\tau-j+2}}{\sqrt{2\alpha}\sigma_{y}}\right)e^{-\alpha4^{\tau+j-2}\left(t-\frac{\mu_{y}}{\sqrt{2\alpha}\sigma_{y}}2^{-\tau-j+2}\right)^{2}}\right.\nonumber \\
 & + & \left.f\left(-\frac{2^{-\tau-j+2}}{\sqrt{2\alpha}\sigma_{y}}\right)e^{-\alpha4^{\tau+j-2}\left(t+\frac{\mu_{y}}{\sqrt{2\alpha}\sigma_{y}}2^{-\tau-j+2}\right)^{2}}\right]d\tau.\label{eq:intermediate formula}
\end{eqnarray}
For $\tau\in\left[0,1\right]$ we have $1\le4^{\tau}\le4$, so that
$\beta=\alpha4^{\tau+j-2}$ satisfies $4^{j-2}\alpha\le\beta\le4^{j-1}\alpha$.
Consequently, using \eqref{eq:expansion on scale j} and \eqref{eq:coefficients on scale j},
we obtain
\[
\left|e^{-\alpha4^{\tau+j-2}\left(t\pm\frac{\mu_{y}}{\sqrt{2\alpha}\sigma_{y}}2^{-\tau-j+2}\right)^{2}}-\sum_{k\in\mathbb{Z}}g_{k}^{j,\pm}\phi_{jk}\left(t\right)\right|\le\epsilon e^{-\alpha4^{\tau+j-2}\left(t\pm\frac{\mu_{y}}{\sqrt{2\alpha}\sigma_{y}}2^{-\tau-j+2}\right)^{2}},
\]
where
\[
g_{k}^{j,\pm}=2^{-j/2}\sqrt{\frac{1}{1-4^{\tau-2}}}e^{-\frac{\alpha4^{\tau-2}}{1-4^{\tau-2}}\left(k\pm\frac{\mu_{y}}{\sqrt{2\alpha}\sigma_{y}}2^{-\tau+2}\right)^{2}},
\]
and use these expressions to replace the corresponding Gaussians in
\eqref{eq:intermediate formula}. Setting $w_{k}^{j}$ as in \eqref{eq:Coefficients of MR expansion of PDF},
we now estimate $d\left(t\right)=p\left(t\right)-\sum_{j\in\mathbb{Z}}\sum_{k\in\mathbb{Z}}w_{k}^{j}\phi_{jk}\left(t\right)$.
We have 
\begin{eqnarray}
\label{eq:intermediate formula-1}\\
d\left(t\right) & = & c\sum_{j\in\mathbb{Z}}\int_{0}^{1}\left[f\left(\frac{2^{-\tau-j+2}}{\sqrt{2\alpha}\sigma_{y}}\right)\left(e^{-\alpha4^{\tau+j-2}\left(t-\frac{\mu_{y}}{\sqrt{2\alpha}\sigma_{y}}2^{-\tau-j+2}\right)^{2}}-\sum_{k\in\mathbb{Z}}g_{k}^{j,-}\phi_{jk}\left(x\right)\right)\right.\nonumber \\
 & + & \left.f\left(-\frac{2^{-\tau-j+2}}{\sqrt{2\alpha}\sigma_{y}}\right)\left(e^{-\alpha4^{\tau+j-2}\left(t+\frac{\mu_{y}}{\sqrt{2\alpha}\sigma_{y}}2^{-\tau-j+2}\right)^{2}}-\sum_{k\in\mathbb{Z}}g_{k}^{j,+}\phi_{jk}\left(x\right)\right)\right]d\tau.\nonumber 
\end{eqnarray}
Since $f$ is non-negative, we obtain
\begin{eqnarray}
\left|d\left(t\right)\right| & \le & \epsilon c\sum_{j\in\mathbb{Z}}\int_{0}^{1}\left[f\left(\frac{2^{-\tau-j+2}}{\sqrt{2\alpha}\sigma_{y}}\right)e^{-\alpha4^{\tau+j-2}\left(t-\frac{\mu_{y}}{\sqrt{2\alpha}\sigma_{y}}2^{-\tau-j+2}\right)^{2}}\right.\nonumber \\
 & + & \left.f\left(-\frac{2^{-\tau-j+2}}{\sqrt{2\alpha}\sigma_{y}}\right)e^{-\alpha4^{\tau+j-2}\left(t+\frac{\mu_{y}}{\sqrt{2\alpha}\sigma_{y}}2^{-\tau-j+2}\right)^{2}}\right]d\tau.\label{eq:intermediate formula-2}
\end{eqnarray}
Using \eqref{eq:intermediate formula} and \eqref{eq:intermediate formula-2},
we have 
\[
\left|d\left(t\right)\right|\le\epsilon p\left(t\right)
\]
yielding the estimate in \eqref{eq:multiresolution expansion of PDF}. 
\end{proof}

\subsection{\label{sub:Adaptive-method-for}Adaptive method for computing integrals }

We compute the integrals 

\[
w_{j,k}^{\pm}=\int_{0}^{1}f\left(\mp\frac{2^{2-j}}{\sqrt{2\alpha}\sigma_{y}}2^{-\tau}\right)e^{-\frac{\alpha4^{\tau-2}}{1-4^{\tau-2}}\left(k\pm\frac{\mu_{y}}{\sqrt{2\alpha}\sigma_{y}}2^{-\tau+2}\right)^{2}}\frac{1}{\sqrt{1-4^{\tau-2}}}d\tau
\]
using adaptive integration. The integrand may have only integrable
singularities in the interior of the interval $\left[0,1\right]$
due to our assumption on $f$. Since $f$ may have a logarithmic singularity
at $x=0$, i.e. $f\left(\left|x\right|\right)\sim\left(\log\left|x\right|\right)^{m}$
near $x=0$, we have 
\[
f\left(\frac{2^{2-j}}{\sqrt{2\alpha}\sigma_{y}}2^{-\tau}\right)\sim\left(\log_{2}\left(\frac{2^{2-j}}{\sqrt{2\alpha}\sigma_{y}}\right)-\tau\right)^{m},
\]
so that, if $j$ is large, $f$ behaves like a polynomial as a function
of $\tau$ . Such behavior of the integrand explains our choice of
a simple and efficient method for adaptive integration. Specifically,
we subdivide the integral $\left[0,1\right]$ adaptively in a hierarchical
fashion while using a fixed number of Legendre nodes to compute the
integral on each subinterval. The subdivision terminates when the
integral computed on two subintervals matches that computed on the
parent interval within an appropriately chosen accuracy. In our examples
we used the Gauss-Legendre quadrature with 10 nodes and set the accuracy
parameter of the adaptive integrator to $10^{-14}$; the local maximum
depth of subdivision remained small in all cases. We present a pseudo-code
for this algorithm in Appendix~A.

\subsection{\label{sub:Representing-PDFs-in-GMRA}Representing PDFs in GMRA}

In order to reduce the computation of the PDF of the product in the
general case to the assumptions of Theorem~\ref{thm:The-PDF-of the product},
at least one of the PDFs has to be represented as a Gaussian mixture.
We consider two approaches for this purpose. The first approach is
applicable to distributions that can be expressed via integrals involving
Gaussians. We illustrate this case using the Laplace distribution.
The second approach relies on a rapidly decaying Fourier transform
of a PDF so that it can be approximated by a bandlimited function.
We illustrate this case using the Gumbel distribution. We also refer
to \cite{LA-MA-SC:2004,MAZ-SCH:2007} for a way to approximate a smooth
PDF via a linear combination of Gaussians provided that the accuracy
requirements are not too stringent. We note that these approximations
can be constructed even for smooth PDFs with compact support since,
outside the support, the approximation can be made arbitrarily small.

\subsubsection{Discretization of integrals to approximate PDFs via Gaussian mixtures}

Although the PDF of the Laplace distribution has two parameters, it
is sufficient to approximate the function in \eqref{eq:Laplace distribution with parameters}
for the particular case $b=1$ and $\mu=0$, 

\[
f\left(x\right)=\frac{1}{2}e^{-\left|x\right|}.
\]
Following \cite[Eq. 37]{BEY-MON:2010}, we have 
\[
\frac{1}{2}e^{-\left|x\right|}=\frac{1}{4\sqrt{\pi}}\int_{-\infty}^{\infty}e^{-e^{t}/4-x^{2}e^{-t}+\frac{1}{2}t}dt.
\]
A linear combination of Gaussians is then obtained by discretizing
this integral. It is sufficient to use the trapezoidal rule due to
the rapid decay of the integrand as $t\to\pm\infty$. We limit integration
to the interval $t\in\left[-40,10\right]$ and use $h=5/12$ as the
step size. Setting $t_{n}=-40+5\cdot n/12$, and using $N=120$ terms,
we arrive at the Gaussian mixture approximation,\lyxdeleted{beylkin}{Tue Jun  6 03:59:17 2017}{
} 
\[
\frac{1}{2}e^{-\left|x\right|}\approx\frac{h}{4\sqrt{\pi}}\sum_{n=0}^{N-1}e^{-e^{-t_{n}}x^{2}}e^{-e^{t_{n}}/4+\frac{1}{2}t_{n}}.
\]
We use this approximation in Example~\ref{sub:The-Laplace-and-Gumbel}
to compute the PDF of the product of two independent random variables,
one with Laplace distribution and the other with Gumbel distribution.
Approximation of the latter via a Gaussian mixture is described next.

\subsubsection{Using an interpolating Gaussian-based scaling function to approximate
PDFs via Gaussian mixtures}

An interpolating Gaussian-based scaling function spans the same subspace
as the function $\phi$, since it is constructed as
\[
L\left(x\right)=\sum_{k\in\mathbb{Z}}c_{k}\phi\left(x-k\right).
\]
Using the interpolating property $L\left(n\right)=\delta_{n0}$, $n\in\mathbb{Z}$,
we also have
\[
\delta_{n0}=\sum_{k\in\mathbb{Z}}c_{k}\phi\left(n-k\right).
\]
Multiplying this identity by $e^{-2\pi inp}$, and adding over $n$,
we obtain 
\begin{equation}
1=\left(\sum_{k\in\mathbb{Z}}c_{k}e^{-2\pi ikp}\right)\left(\sum_{n\in\mathbb{Z}}\phi\left(n\right)e^{-2\pi inp}\right).\label{eq:Interpolating function property}
\end{equation}
By evaluating the Fourier transform, 
\[
\widehat{L}\left(p\right)=\left(\sum_{k\in\mathbb{Z}}c_{k}e^{-2\pi ikp}\right)\widehat{\phi}\left(p\right),
\]
and using \eqref{eq:Interpolating function property}, we arrive at
\[
\widehat{L}\left(p\right)=\frac{\widehat{\phi}\left(p\right)}{\sum_{n\in\mathbb{Z}}\phi\left(n\right)e^{-2\pi inp}}.
\]
By Poisson summation formula,
\[
\sum_{n\in\mathbb{Z}}\phi\left(n\right)e^{-2\pi inp}=\sqrt{\frac{\alpha}{\pi}}\sum_{n\in\mathbb{Z}}e^{-\alpha n^{2}}e^{-2\pi inp}=\sqrt{\frac{\alpha}{\pi}}\vartheta_{3}\left(\pi p,e^{-\alpha}\right),
\]
we finally obtain 
\begin{equation}
\widehat{L}\left(p\right)=\frac{\widehat{\phi}\left(p\right)}{\sqrt{\frac{\alpha}{\pi}}\vartheta_{3}\left(\pi p,e^{-\alpha}\right)}.\label{eq:relation between interpolating and Gaussian scaling fncs}
\end{equation}
To illustrate the use of the interpolating scaling functions in the
construction of Gaussian mixtures, we consider the Gumbel distribution
in \eqref{eq:Gumbel Distribution with parameters} with parameters
$\sigma=1$ and $\mu=0$, 
\[
g\left(x\right)=e^{-x-e^{-x}}.
\]
Gaussian mixtures for general parameters are then obtained from this
case by rescaling. First we select an interval to approximate $g$.
For example, outside the interval $\left[-6,50\right]$ the function
$g$ is less than $\approx2\cdot10^{-22}$. We treat $g$ on this
interval as approximately periodic and, changing variables, $x=-6+56s$,
rescale it to the interval $\left[0,1\right]$, 
\begin{equation}
\tilde{g}\left(s\right)=g\left(-6+56s\right).\label{eq:rescaled_gumbel}
\end{equation}
Sampling $\tilde{g}\left(s\right)$ at $N$ points, where $N$ is
sufficiently large, we write its approximation using the interpolating
scaling function $L$ as 
\[
\tilde{g}\left(s\right)=\sum_{k=0}^{N-1}\tilde{g}\left(\frac{k}{N}\right)N^{1/2}L\left(Ns-k\right).
\]
Using the scaling function $\phi\left(s\right)=\sqrt{\frac{\alpha}{\pi}}e^{-\alpha s^{2}}$,
we seek the coefficients $g_{k}$ such that 
\begin{equation}
\tilde{g}\left(s\right)=\sum_{k=0}^{N-1}g_{k}N^{1/2}\phi\left(Ns-k\right).\label{eq:final Gaussian mixture}
\end{equation}
Computing the Fourier transform, we obtain 
\[
\widehat{\tilde{g}}\left(p\right)=N^{-1/2}\left(\sum_{k=0}^{N-1}\tilde{g}\left(\frac{k}{N}\right)e^{-2\pi ikp/N}\right)\widehat{L}\left(\frac{p}{N}\right),
\]
and 
\[
\widehat{\tilde{g}}\left(p\right)=N^{-1/2}\left(\sum_{k=0}^{N-1}g_{k}e^{-2\pi ikp/N}\right)\widehat{\phi}\left(\frac{p}{N}\right).
\]
Sampling at $N$ points and using \eqref{eq:relation between interpolating and Gaussian scaling fncs},
we obtain 
\[
\sum_{k=0}^{N-1}g_{k}e^{-2\pi ikn/N}=\frac{\sum_{k=0}^{N-1}\tilde{g}\left(\frac{k}{N}\right)e^{-2\pi ikn/N}}{\sqrt{\frac{\alpha}{\pi}}\vartheta_{3}\left(\pi\frac{n}{N},e^{-\alpha}\right)},\,\,\,n=0,\dots N-1,
\]
and compute the coefficients $g_{k}$ via the Fast Fourier Transform
(FFT). In this construction $N$ is chosen to have small prime factors
since we use FFT. We illustrate the result in Figure~\ref{fig:Error-of-approximating Gumbel},
where we set $N=300$. We use this Gaussian mixture as an approximation
of the Gumbel distribution in Example~\ref{sub:The-Laplace-and-Gumbel}.
\begin{rem}
We plan to develop a more efficient algorithm to yield multiresolution
representation of PDFs (i.e., with fewer terms) by applying the approach
presented in this section in an iterative fashion. 
\end{rem}
\begin{figure}
\includegraphics[scale=0.4]{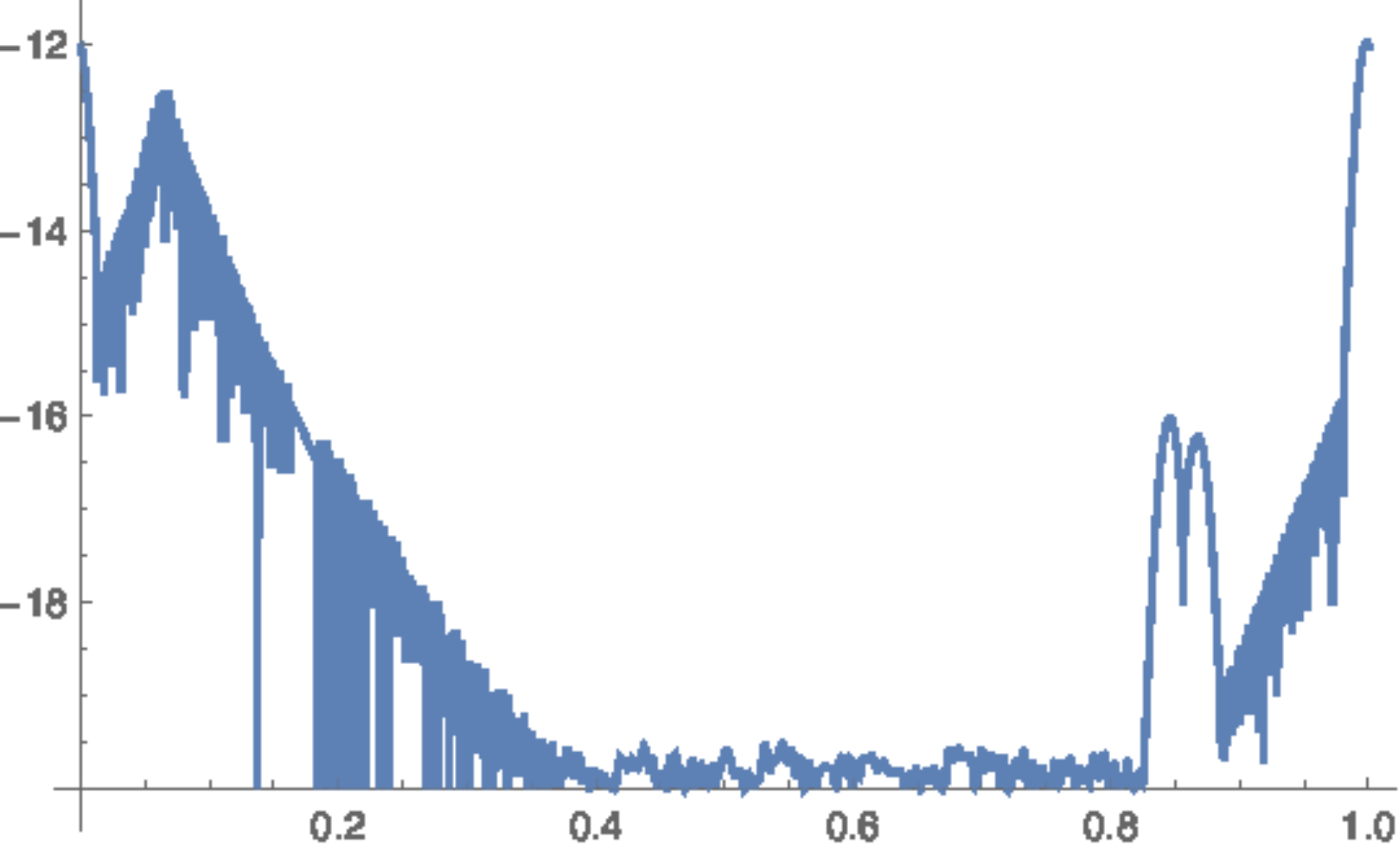}\caption{Error\label{fig:Error-of-approximating Gumbel} of approximating the
rescaled Gumbel PDF $\tilde{g}\left(s\right)$ in \eqref{eq:rescaled_gumbel}
as a Gaussian mixture with $N=300$ terms.}
\end{figure}

\subsection{Computing the PDFs in GMRA representation}

For computing the PDF of the product of independent random variables
where the original PDFs are in GMRA representation, we repeat the
derivation in Theorem~\ref{thm:The-PDF-of the product} using two
representative Gaussians from our approximate basis \eqref{eq:Gaussian multiresolution basis},
namely 
\begin{equation}
f\left(x\right)=2^{n/2}e^{-\alpha\left(2^{n}x-m\right)}\:\:\mbox{and}\:\:g\left(y\right)=2^{l/2}e^{-\alpha\left(2^{l}y-m'\right)}.\label{f_and_g}
\end{equation}
We have 
\begin{equation}
p(t)=2^{(n+l)/2}\int_{-\infty}^{\infty}{\displaystyle \int_{-\infty}^{\infty}e^{-\alpha\left(2^{n}x-m\right)^{2}}e^{-\alpha\left(2^{l}y-m'\right)^{2}}\delta\left(xy-t\right)dxdy}\label{eq:integral between basis functions-2}
\end{equation}
and employing the same changes of variables as in the proof of Theorem~\ref{thm:The-PDF-of the product},
arrive at 
\[
p\left(t\right)=2^{(n+l)/2}\int_{0}^{\infty}\left(e^{-\alpha\left(2^{n}x-m\right)^{2}}e^{-\alpha\left(2^{l}\frac{t}{x}-m'\right)^{2}}+e^{-\alpha\left(2^{n}x+m\right)^{2}}e^{-\alpha\left(2^{l}\frac{t}{x}+m'\right)^{2}}\right)\frac{1}{x}dx.
\]
Changing variables $y=2^{n}x$, 
\[
p\left(t\right)=2^{(n+l)/2}\int_{0}^{\infty}\left(e^{-\alpha\left(y-m\right)^{2}}e^{-\alpha\left(2^{n+l}\frac{t}{y}-m'\right)^{2}}+e^{-\alpha\left(y+m\right)^{2}}e^{-\alpha\left(2^{n+l}\frac{t}{y}+m'\right)^{2}}\right)\frac{1}{y}dy,
\]
we define 
\begin{eqnarray*}
u_{mm'}\left(t\right) & = & p\left(2^{-n-l}t\right)2^{-\left(n+l\right)/2}\\
 & = & \int_{0}^{\infty}\left(e^{-\alpha\left(y-m\right)^{2}}e^{-\alpha\left(\frac{t}{y}-m'\right)^{2}}+e^{-\alpha\left(y+m\right)^{2}}e^{-\alpha\left(\frac{t}{y}+m'\right)^{2}}\right)\frac{1}{y}dy.
\end{eqnarray*}
Changing variables $y=2^{-s}$, $s\in\mathbb{R}$, and splitting the
real axis into unit subintervals (see proof of Theorem~\ref{thm:The-PDF-of the product}),
we obtain
\[
u_{mm'}\left(t\right)=\log\left(2\right)\sum_{j\in\mathbb{Z}}\int_{j-2}^{j-1}\left(e^{-\alpha\left(2^{-s}-m\right)^{2}}e^{-\alpha\left(2^{s}t-m'\right)^{2}}+e^{-\alpha\left(2^{-s}+m\right)^{2}}e^{-\alpha\left(2^{s}t+m'\right)^{2}}\right)ds.
\]
Changing variables one more time, $s=\tau+j-2$, we arrive at 
\begin{eqnarray*}
u_{mm'}\left(t\right) & = & \log\left(2\right)\sum_{j\in\mathbb{Z}}\int_{0}^{1}\left(e^{-\alpha\left(2^{2-j-\tau}-m\right)^{2}}e^{-\alpha\left(2^{\tau+j-2}t-m'\right)^{2}}\right.\\
 & + & \left.e^{-\alpha\left(2^{2-j-\tau}+m\right)^{2}}e^{-\alpha\left(2^{\tau+j-2}t+m'\right)^{2}}\right)d\tau.
\end{eqnarray*}
Finally, we use Theorem~\ref{thm:The-Gaussian-on arbitrary scale}
to obtain

\[
u_{mm'}\left(t\right)=\log\left(2\right)\sum_{j\in\mathbb{Z}}\sum_{k\in\mathbb{Z}}\phi_{jk}\left(t\right)2^{-j/2}u_{kmm'}^{j},
\]
where
\begin{eqnarray*}
u_{kmm'}^{j} & = & \int_{0}^{1}\frac{1}{\sqrt{1-4^{\tau-2}}}\left(e^{-\alpha\left(2^{2-j-\tau}-m\right)^{2}}e^{-\frac{\alpha}{1-4^{\tau-2}}\left(2^{\tau-2}k-m'\right)^{2}}\right.\\
 & + & \left.e^{-\alpha\left(2^{2-j-\tau}+m\right)^{2}}e^{-\frac{\alpha}{1-4^{\tau-2}}\left(2^{\tau-2}k+m'\right)^{2}}\right)d\tau.
\end{eqnarray*}
Since the integrand is comprised of well-behaved functions, we can
discretize the integral using the Gauss-Legendre quadrature with $M$
nodes directly (rather than using an adaptive integrator). It follows
that
\[
u_{kmm'}^{j}=\sum_{i=1}^{M}\eta_{i}\left(U_{j,m}^{-,i}V_{km'}^{-,i}+U_{j,m}^{+,i}V_{km'}^{+,i}\right),
\]
where $\eta_{i}=w_{i}\left(1-4^{\tau_{i}-2}\right)^{-1/2}$, $w_{i}$
are the weights of the Gauss-Legendre quadrature,
\[
U_{j,m}^{\pm,i}=e^{-\alpha\left(2^{2-j-\tau_{i}}\pm m\right)^{2}},
\]
and 
\[
V_{km'}^{\pm,i}=e^{-\frac{\alpha}{1-4^{\tau_{i}-2}}\left(2^{\tau_{i}-2}k+m'\right)^{2}}.
\]
We note that matrices $U_{j,m}^{\pm,i}$ and $V_{km'}^{\pm,i}$ can
be precomputed and that they are sparse (banded) since small entries
can be removed. Also, these matrices do not depend on the scales $n$
and $l$ of the original Gaussians $f$ and $g$ in \eqref{f_and_g}.

\subsection{Computing the PDF of a product of two random variables with a bivariate
normal distribution }

We show how to compute the PDF $p$,

\begin{equation}
p\left(t\right)=\int_{\mathbb{R}^{2}}g\left(x,y\right)\delta\left(xy-t\right)dxdy,\label{eq:product_correlated}
\end{equation}
where $g$ is the joint PDF
\[
g\left(x,y\right)=\frac{1}{2\pi\sigma_{x}\sigma_{y}\sqrt{1-\rho^{2}}}e^{-\frac{1}{2}\left(\begin{array}{c}
x-\mu_{x}\\
y-\mu_{y}
\end{array}\right)^{T}\Sigma^{-1}\left(\begin{array}{c}
x-\mu_{x}\\
y-\mu_{y}
\end{array}\right)},
\]
with
\[
\Sigma^{-1}=\frac{1}{1-\rho^{2}}\left(\begin{array}{cc}
\frac{1}{\sigma_{x}^{2}} & -\frac{\rho}{\sigma_{x}\sigma_{y}}\\
-\frac{\rho}{\sigma_{x}\sigma_{y}} & \frac{1}{\sigma_{y}^{2}}
\end{array}\right).
\]
Here $\rho$ is the correlation coefficient and $\rho=0$ corresponds
to the two variables being uncorrelated and independent. Using the
next lemma we reduce the computation of $p$ to the results on independent
normal variables in Section~\ref{sec:Distribution-of-the}.
\begin{lem}
The PDF p in \eqref{eq:product_correlated} can be described as
\begin{eqnarray*}
p\left(t\right) & = & \frac{e^{\frac{\rho}{\left(1-\rho^{2}\right)}\left(\frac{t-\mu_{x}\mu_{y}}{\sigma_{x}\sigma_{y}}\right)+\frac{\rho^{2}}{2\left(1-\rho^{2}\right)}\left[\left(\frac{\mu_{x}}{\sigma_{x}}\right)^{2}+\left(\frac{\mu_{y}}{\sigma_{y}}\right)^{2}\right]}}{2\pi\sigma_{x}\sigma_{y}\sqrt{1-\rho^{2}}}\\
 &  & \int_{\mathbb{R}^{2}}e^{-\frac{1}{2\left(1-\rho^{2}\right)}\left[\left(\frac{x-\mu_{x}}{\sigma_{x}}+\rho\frac{\mu_{y}}{\sigma_{y}}\right)^{2}+\left(\frac{y-\mu_{y}}{\sigma_{y}}+\rho\frac{\mu_{x}}{\sigma_{x}}\right)^{2}\right]}\delta\left(xy-t\right)dxdy.
\end{eqnarray*}
\end{lem}
\begin{proof}
Denote by $Q$ the quadratic form, 
\begin{eqnarray*}
Q\left(x,y\right) & = & -\frac{1}{2}\left(\begin{array}{c}
x-\mu_{x}\\
y-\mu_{y}
\end{array}\right)^{T}\Sigma^{-1}\left(\begin{array}{c}
x-\mu_{x}\\
y-\mu_{y}
\end{array}\right)\\
 & = & -\frac{1}{2\left(1-\rho^{2}\right)}\left[\left(\frac{x-\mu_{x}}{\sigma_{x}}\right)^{2}+\left(\frac{y-\mu_{y}}{\sigma_{y}}\right)^{2}-2\rho\left(\frac{x-\mu_{x}}{\sigma_{x}}\right)\left(\frac{y-\mu_{y}}{\sigma_{y}}\right)\right]\\
 & = & -\frac{\tilde{x}^{2}+\tilde{y}^{2}-2\rho\tilde{x}\tilde{y}}{2\left(1-\rho^{2}\right)},
\end{eqnarray*}
where $\tilde{x}=\left(x-\mu_{x}\right)/\sigma_{x}$ and $\tilde{y}=\left(y-\mu_{y}\right)/\sigma_{y}$.
Since 
\begin{align*}
\tilde{x}\tilde{y} & =\frac{xy}{\sigma_{x}\sigma_{y}}+\frac{\mu_{x}\mu_{y}}{\sigma_{x}\sigma_{y}}-\frac{\mu_{y}}{\sigma_{x}\sigma_{y}}\left(\sigma_{x}\tilde{x}+\mu_{x}\right)-\frac{\mu_{x}}{\sigma_{x}\sigma_{y}}\left(\sigma_{y}\tilde{y}+\mu_{y}\right)\\
 & =\frac{xy-\mu_{x}\mu_{y}}{\sigma_{x}\sigma_{y}}-\frac{\mu_{y}}{\sigma_{y}}\tilde{x}-\frac{\mu_{x}}{\sigma_{x}}\tilde{y},
\end{align*}
with $q_{x}=\mu_{x}/\sigma_{x}$ and $q_{y}=\mu_{y}/\sigma_{y}$,
we rewrite the numerator of $Q\left(x,y\right)$ as
\begin{align*}
\tilde{x}^{2}+\tilde{y}^{2}-2\rho\tilde{x}\tilde{y} & =\tilde{x}^{2}+2\rho q_{y}\tilde{x}+\tilde{y}^{2}+2\rho q_{y}\tilde{y}-2\rho\left(\frac{xy-\mu_{x}\mu_{y}}{\sigma_{x}\sigma_{y}}\right)\\
 & =\left(\tilde{x}+\rho q_{y}\right)^{2}+\left(\tilde{y}+\rho q_{x}\right)^{2}-2\rho\left(\frac{xy-\mu_{x}\mu_{y}}{\sigma_{x}\sigma_{y}}\right)-\rho^{2}\left(q_{x}^{2}+q_{y}^{2}\right).
\end{align*}
Therefore, using that $xy$ can be replaced by $t$ in the integrand
of \eqref{eq:product_correlated}, we have
\begin{eqnarray*}
\int_{\mathbb{R}^{2}}e^{Q\left(x,y\right)}\delta\left(xy-t\right)dxdy & = & e^{\frac{\rho}{\left(1-\rho^{2}\right)}\frac{t-\mu_{x}\mu_{y}}{\sigma_{x}\sigma_{y}}+\frac{\rho^{2}}{2\left(1-\rho^{2}\right)}q_{x}^{2}+q_{y}^{2}}\\
 &  & \int_{\mathbb{R}^{2}}e^{-\frac{1}{2\left(1-\rho^{2}\right)}\left[\left(\tilde{x}+\rho q_{y}\right)^{2}+\left(\tilde{y}+\rho q_{x}\right)^{2}\right]}\delta\left(xy-t\right)dxdy
\end{eqnarray*}
which proves the result. \end{proof}
\begin{rem}
\label{Ordering remark}The integral describing the PDF of the product
\eqref{eq:integral-to-compute-1} does not depend on the order of
the factors $f$ and $g$. However, the computation of \eqref{eq:integral-to-compute-1}
relies on the evaluation of the coefficients in \eqref{eq:Coefficients of MR expansion of PDF}
and these coefficients do depend on the order of the factors. While
the resulting PDF does not depend on the chosen order, one of the
two possible representations may require significantly more coefficients.
To explain the reason for such behavior, let us assume that both random
variables in the product have normal distributions but different means
and deviations, $\mu_{x}$, $\sigma_{x}$ and $\mu_{y}$, $\sigma_{y}$.
According to Theorem~\ref{thm:The-PDF-of the product}, we need to
evaluate 

\[
\frac{1}{\sqrt{2\pi}\sigma_{y}}f\left(\mp\frac{2^{2-j}}{\sqrt{2\alpha}\sigma_{y}}2^{-\tau}\right)=\frac{1}{2\pi\sigma_{x}\sigma_{y}}e^{-\frac{(\mp2^{2-j-\tau}-\sqrt{2\alpha}\sigma_{y}\mu_{x})^{2}}{4\alpha\sigma_{y}^{2}\sigma_{x}^{2}}}.
\]
As $j$ becomes large, we have 
\[
\lim_{j\to\infty}e^{-\frac{(\mp2^{2-j-\tau}-\sqrt{2\alpha}\sigma_{y}\mu_{x})^{2}}{4\alpha\sigma_{y}^{2}\sigma_{x}^{2}}}=e^{-\frac{\mu_{x}^{2}}{2\sigma_{x}^{2}}}.
\]
Thus, in order to reduce the number of scales, given the ratios, 
\begin{equation}
\frac{\mu_{x}^{2}}{2\sigma_{x}^{2}}\:\:\mbox{and}\:\:\frac{\mu_{y}^{2}}{2\sigma_{y}^{2}}\label{eq:ratios}
\end{equation}
we pick $f$ in \eqref{eq:Definition of the PDF} as the distribution
with the largest ratio. The impact of ordering the factors in this
way is illustrated in Example~\ref{sub:example6121}. 
\end{rem}
This remark implies that the ordering of the Gaussians $f$ and $g$
in \eqref{f_and_g}-\eqref{eq:integral between basis functions-2}
should be such that $m\ge m'$. We also note that in this particular
case the ratios in \eqref{eq:ratios} do not depend on the scales
$n$ and $l$. 
\begin{rem}
\label{rem: moments remark}Once a PDF $f$ is represented via GMRA,
\[
f\left(x\right)=\sum_{j,k}c_{j,k}\phi_{j,k}\left(x\right),
\]
where $j$ and $k$ only take a finite number of integer values, it
is immediate to compute its moments since the moments of the basis
functions $\phi_{j,k}$, are readily available. To be precise, for
any nonnegative integer $n$, we compute
\[
\mu_{j,k}^{\left(n\right)}=\int_{\mathbb{R}}\phi_{j,k}\left(x\right)x^{n}dx=2^{j/2}\sqrt{\frac{\alpha}{\pi}}\int_{\mathbb{R}}e^{-\alpha\left(2^{j}x-k\right)^{2}}x^{n}dx=2^{-j\left(n+\frac{1}{2}\right)}\mu_{k}^{\left(n\right)},
\]
where
\begin{equation}
\mu_{k}^{\left(n\right)}=\sqrt{\frac{\alpha}{\pi}}\int_{\mathbb{R}}e^{-\alpha\left(x-k\right)^{2}}x^{n}dx.\label{gaussian_moments}
\end{equation}
Thus we have $\mu_{k}^{\left(0\right)}=1$, $\mu_{k}^{\left(1\right)}=k$,
and higher moments of the basis functions $\phi_{j,k}$ can be evaluated
using the recurrence,
\[
\mu_{k}^{\left(n\right)}=k\mu_{k}^{\left(n-1\right)}+\frac{n-1}{2\alpha}\mu_{k}^{\left(n-2\right)},\ n\geq2,
\]
obtained by rewriting \eqref{gaussian_moments} as 
\begin{eqnarray*}
\mu_{k}^{\left(n\right)} & = & \sqrt{\frac{\alpha}{\pi}}\int_{\mathbb{R}}e^{-\alpha\left(x-k\right)^{2}}\left(k+x-k\right)x^{n-1}dx\\
 & = & k\mu_{k}^{\left(n-1\right)}+\sqrt{\frac{\alpha}{\pi}}\int_{\mathbb{R}}\frac{d}{dx}\left(\frac{e^{-\alpha\left(x-k\right)^{2}}}{-2\alpha}\right)x^{n-1}dx
\end{eqnarray*}
and integrating by parts. For the moments of $f$ we arrive at the
recurrence 
\[
\mathcal{M}^{\left(n\right)}=\int_{\mathbb{R}}f\left(x\right)x^{n}dx=\sum_{j}2^{-j\left(n+\frac{1}{2}\right)}\left(\sum_{k}c_{j,k}k\mu_{k}^{\left(n-1\right)}\right)+\frac{n-1}{\alpha}2^{-2j}\mathcal{M}^{\left(n-2\right)},\ n\geq2,
\]
where $\mathcal{M}^{\left(0\right)}=\sum_{j}2^{-j/2}\sum_{k}c_{j,k}$
and $\mathcal{M}^{\left(1\right)}=\sum_{j}2^{-3j/2}\sum_{k}c_{j,k}k.$
In order to compute the variance, we have
\[
E\left[\left(X-\mathcal{M}^{\left(1\right)}\right)^{2}\right]=\mathcal{M}^{\left(2\right)}-\left(\mathcal{M}^{\left(1\right)}\right)^{2},
\]
with similar formulas for skewness and kurtosis.
\end{rem}
~
\begin{rem}
\label{rem:expectation}More generally, an important use of the PDF
$p_{Z}$ represented in GMRA is to compute, for a function $u$, the
expectation of the variable $u\left(Z\right)$,
\begin{equation}
\mbox{E}\left[u\left(Z\right)\right]=\int_{-\infty}^{\infty}u\left(t\right)p_{Z}\left(t\right)dt.\label{eq:expectation integral}
\end{equation}
If the function $u$ is given analytically (i.e., we can evaluate
it at any point), the fact that we have a functional representation
of $p_{Z}$ allows us to use an appropriate quadrature to evaluate
this integral to a desired accuracy. If only samples of the function
$u$ are provided, then we can treat $p_{Z}$ as a weight and construct
an appropriate quadrature with nodes at locations where the values
of $u$ are available. On the other hand, if the function $u$ is
also representated in GMRA the integral in \eqref{eq:expectation integral}
can be evaluated explicitly. 

For example, if $u$ is a non-negative function and \foreignlanguage{english}{$\tilde{p}_{Z}$}
is the GMRA approximation of $p_{Z}$ such that 
\[
\left|p_{Z}\left(t\right)-\tilde{p}_{Z}\left(t\right)\right|\le\epsilon p_{Z}\left(t\right),
\]
(e.g. as in Theorem~\ref{thm:The-PDF-of the product}), then we have
the relative error estimate 
\[
\left|\mbox{E}\left[u\left(Z\right)\right]-\tilde{\mbox{E}}\left[u\left(Z\right)\right]\right|\le\epsilon\mbox{E}\left[u\left(Z\right)\right],
\]
where
\[
\tilde{\mbox{E}}\left[u\left(Z\right)\right]=\int_{-\infty}^{\infty}u\left(t\right)\tilde{p}_{Z}\left(t\right)dt.
\]

\end{rem}

\selectlanguage{american}%

\section{\label{sec:Numerical-Examples}Examples of computing PDFs of products
of independent random variables}

Unless noted otherwise, in all our numerical examples we set the parameter
$\alpha=1/4$, so that $\epsilon=2.77\cdot10^{-13}$ in \eqref{eq:parameter epsilon}.
For representing the PDF of the product of random variables, we use
a GMRA with a very large number of scales ($100$ scales) to demonstrate
that we can accurately resolve the singularity at zero. Obviously,
there is no need to use so many scales in practical applications;
we use them to illustrate that the accuracy of our approach is well
beyond the computational feasibility of any Monte-Carlo method. The
test code was written in Fortran 90 and run on a laptop with $2.20$
GHz Intel processor. While no speed comparison with a Monte-Carlo
method is meaningful, we note the actual time required for our algorithm.
Computing the coefficients of the product of two normal distributions
and using $100$ scales takes $\approx0.12$ seconds and evaluation
of the resulting function at $5,000$ points takes $\approx0.1$ second.
While the computation of the coefficients $w_{k}^{j}$ in \eqref{eq:multiresolution expansion of PDF}
is trivially parallelizable, we did not take advantage of it in our
code.

In all our examples, we display significant coefficients of the GMRA
representation, $w_{k}^{j}$ in \eqref{eq:multiresolution expansion of PDF},
as a matrix where the horizontal direction corresponds to the shift
$k$ and the vertical direction to the scale $j$. In interpreting
these images (see Figures~\ref{fig:example0101-1}, \ref{fig:example010101-2},
\ref{fig:Example2111-2}, \ref{fig:Example6121-3}, \ref{fig: Cauchy-2},
\ref{fig: Laplace2}, \ref{fig: Gumbel2} and \ref{fig:The-Laplace-and-the-Gumbel-figure-2}
below), it is important to remember that the shift $k$ of the basis
functions $2^{j/2}\phi\left(2^{j}t-k\right)$ indicates the distance
to the singularity at $t=0$, which is measured in scale dependent
units, i.e. $2^{-j}$ at the scale $j$. 

Just as an extra check, in the examples where explicit answers are
not available, we verify the computed distributions by evaluating
zero, first and second moments. In all cases these moments were within
the accuracy of computation. The moments of PDFs in GMRA representation
were computed as described in Remark~\ref{rem: moments remark}.

\subsection{Accuracy tests}

The PDFs of the product of two and three normal random variables with
zero means are available analytically and we use them to ascertain
accuracy of our algorithm.

\subsubsection{\label{sub:example0101}Two normal PDFs with zero means}

The exact PDF of the product of two random variables $X\thicksim N\left(\mu_{x},\sigma_{x}^{2}\right)$,
$Y\thicksim N\left(\mu_{y},\sigma_{y}^{2}\right)$, with $\mu_{x}=\mu_{y}=0$,
is given in \eqref{eq:K_0 distribution}. Setting $\sigma_{x}=\sigma_{y}=1$,
we compute the PDF of the product $Z=XY$ via our algorithm and compare
it with the exact PDF. The computed PDF $p_{Z}$ and its GMRA coefficients
are displayed in Figure~\ref{fig:example0101-1} and the relative
errors
\begin{equation}
\epsilon\left(x\right)=\log_{10}\left(\frac{\left|p_{Z}\left(10^{x}\right)-\frac{1}{\pi}K_{0}\left(10^{x}\right)\right|}{\frac{1}{\pi}K_{0}\left(10^{x}\right)}\right),\,\,\,-30\le x\le0,\label{eq:error(x)}
\end{equation}
for different values of the parameter $\alpha$ of the GMRA, are displayed
in Figure~\ref{fig:example0101-2}. The exact PDF in \eqref{eq:K_0 distribution}
has a logarithmic singularity at the origin and we use $100$ scales
to recover the PDF up to the interval $\left[-10^{-27},10^{-27}\right]$
with accuracy $\epsilon\approx10^{-13}$. The limit of resolution
of this singularity is illustrated in Figure~\ref{fig:example0101-2}
(again, well beyond what may be needed in applications). 

\begin{figure}[H]
\begin{centering}
\includegraphics[scale=0.3]{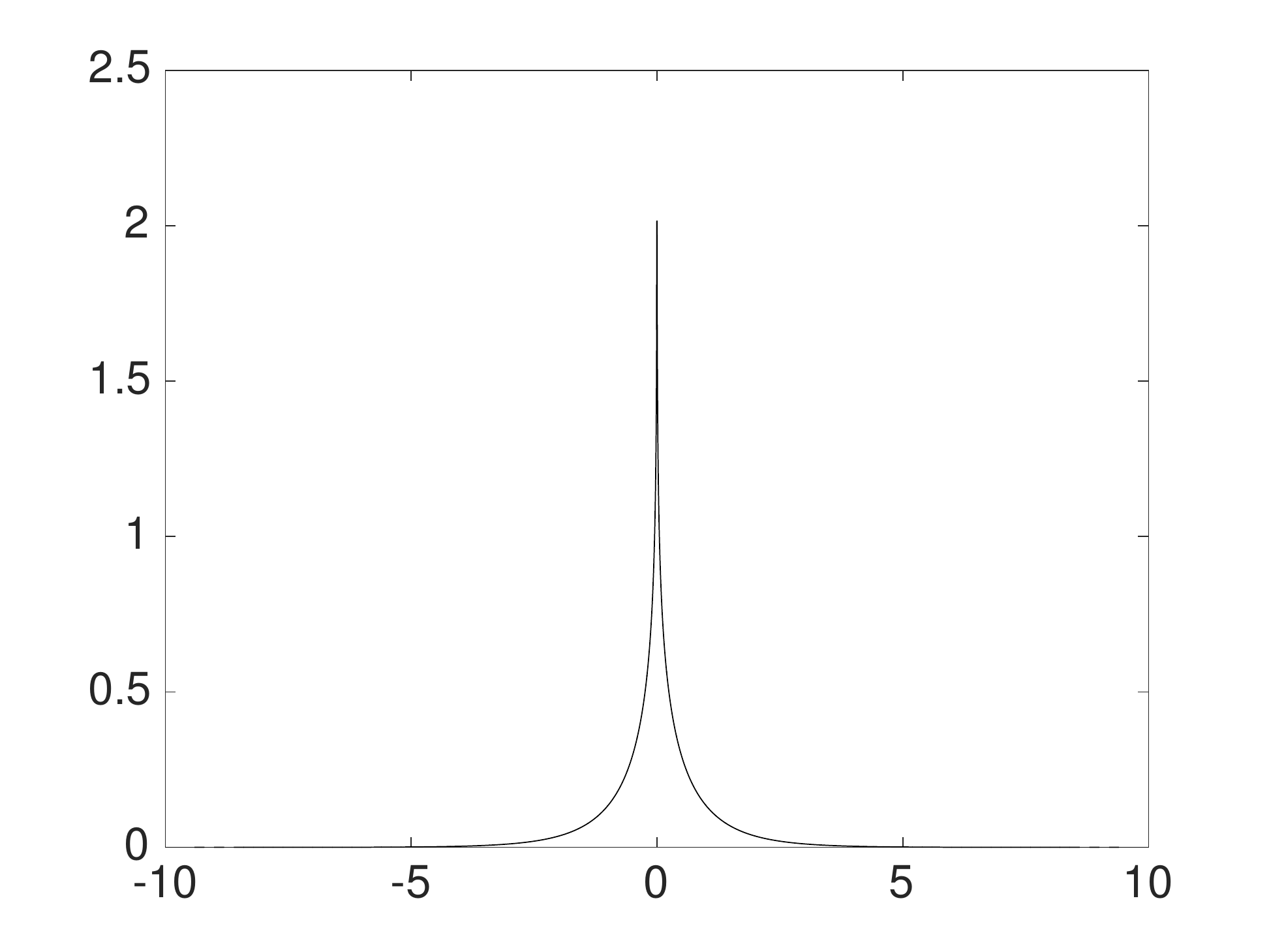}
\includegraphics[scale=0.3]{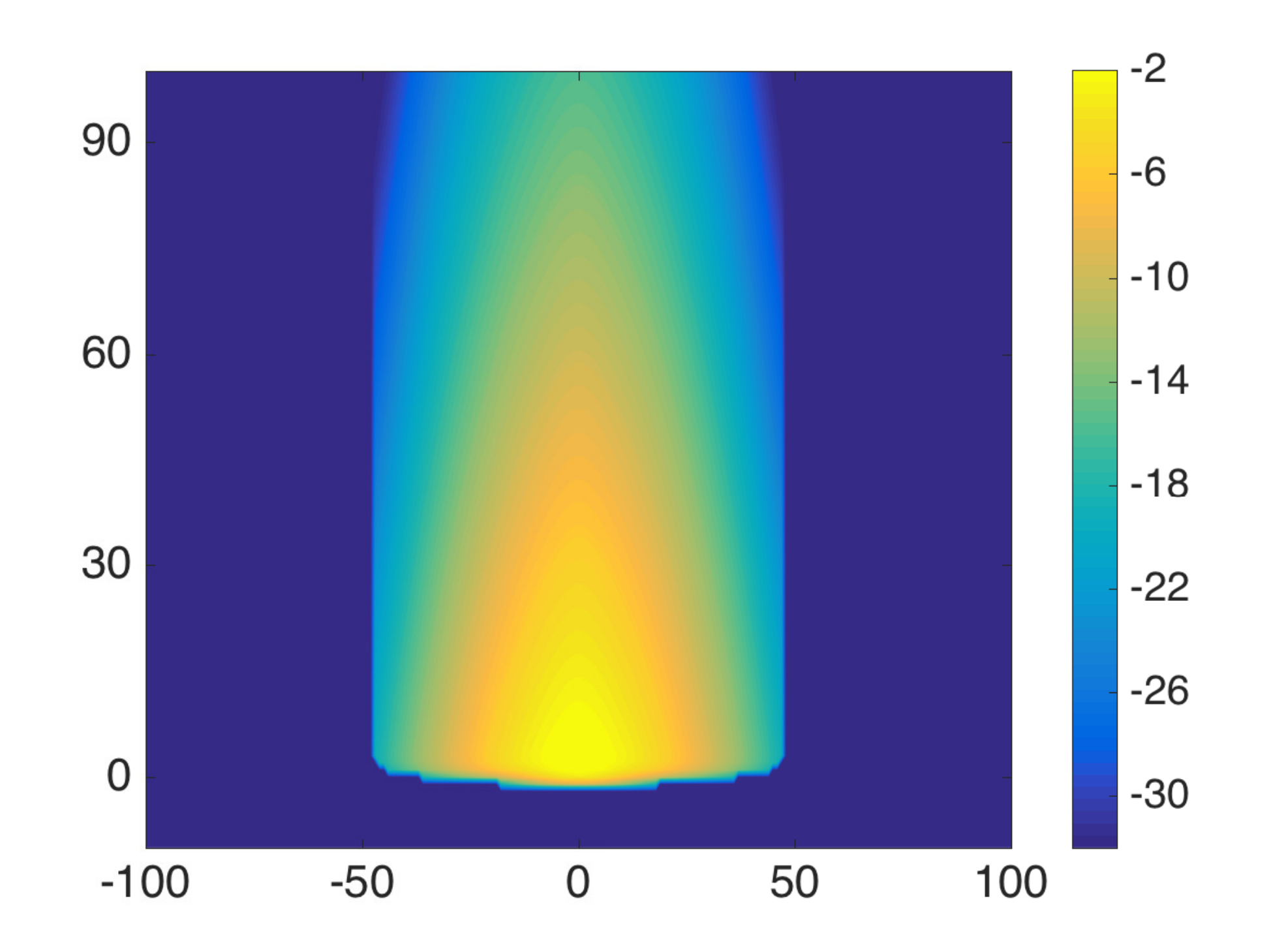}
\par\end{centering}

\centering{}\caption{\label{fig:example0101-1} PDF $p_{Z}$ of the product in Example~\ref{sub:example0101}
(left) and logarithm (base 10) of GMRA coefficients of $p_{Z}$ (right).}
\end{figure}

\begin{figure}[H]
\begin{centering}
\includegraphics[scale=0.3]{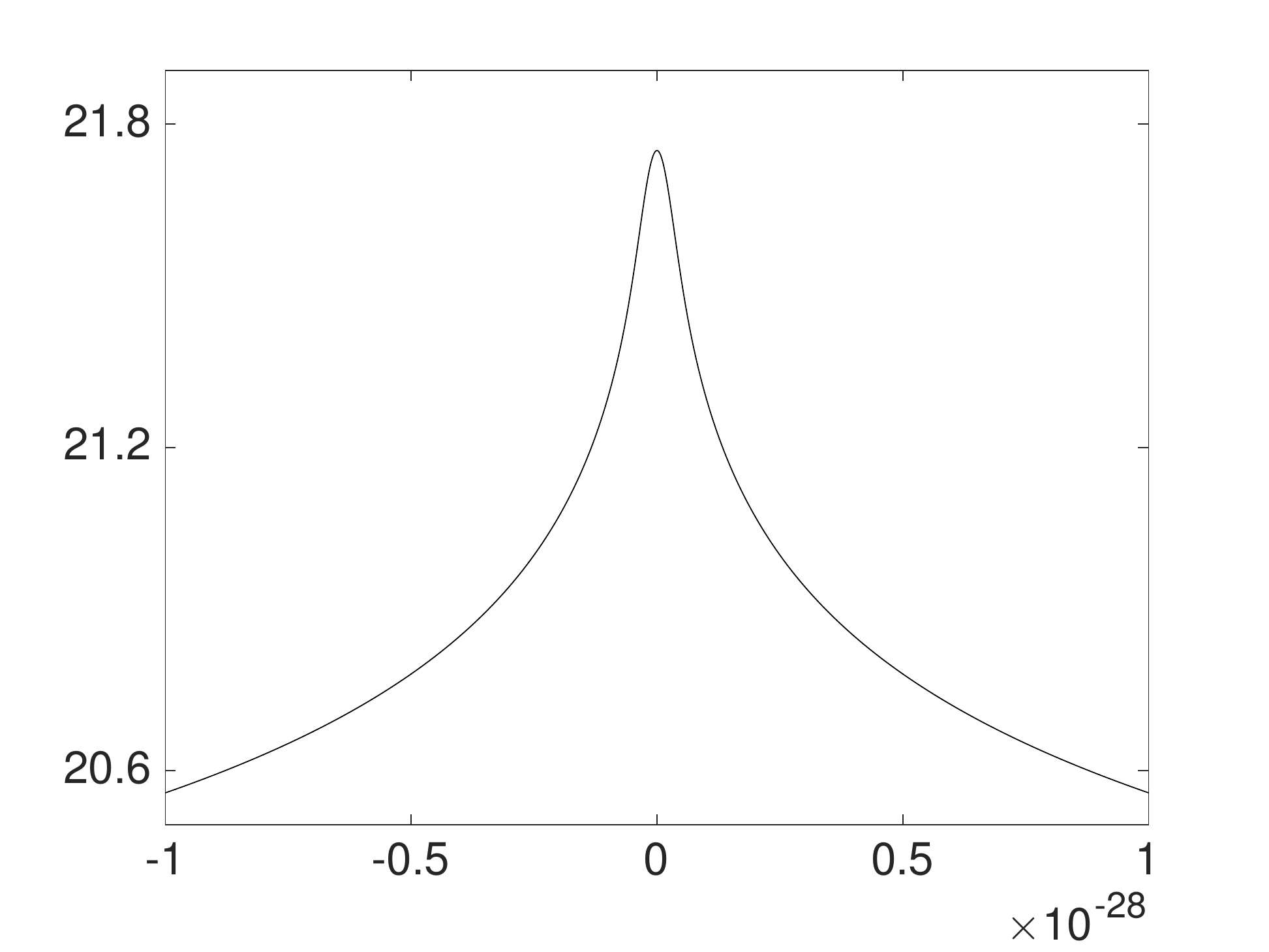}~ \includegraphics[scale=0.3]{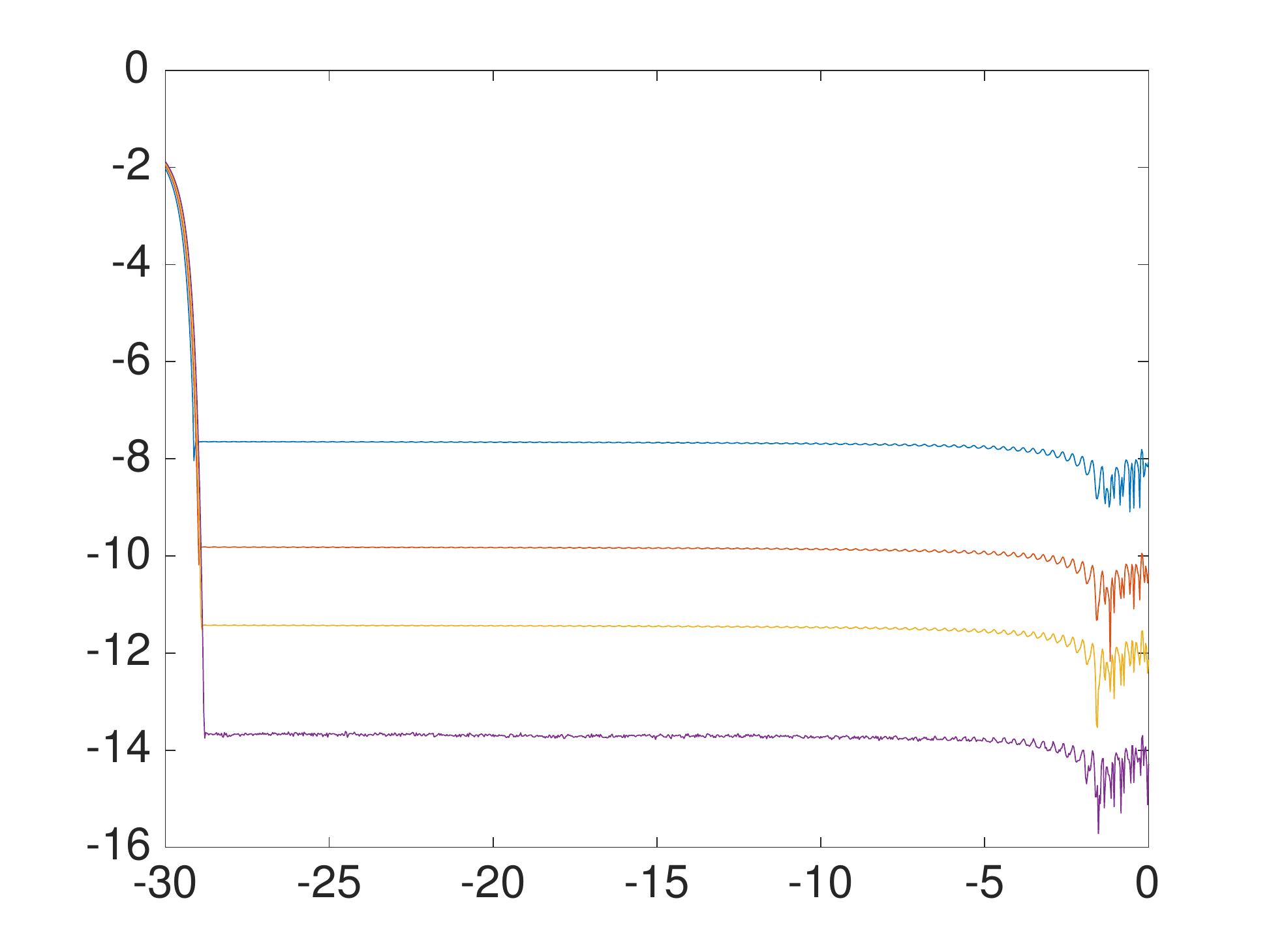}
\par\end{centering}

\centering{}\caption{\label{fig:example0101-2} The limit of resolution of singularity
of the PDF $p_{Z}$ in Example \ref{sub:example0101} using GMRA with
$100$ scales and $\alpha=0.25$ (left) and the relative error curves
\eqref{eq:error(x)} shown over the interval $\left[10^{-30},0\right]$
for different parameters $\alpha$ in the GMRA \eqref{eq:Gaussian multiresolution basis}.
The bottom curve corresponds to $\alpha=0.25$, followed by $\alpha=0.3$,
$\alpha=0.35$ and $\alpha=0.4$, as we go up in the display (right).}
\end{figure}

\subsubsection{\label{sub:example010101}Three normal PDFs with zero means}

We compute the PDF of the product of three zero-mean Gaussian random
variables, starting with the GMRA representation of the product of
two of them computed in the previous example. Given three zero-mean
Gaussian random variables, $X_{1}\thicksim N\left(0,\sigma_{1}^{2}\right)$,
$X_{2}\thicksim N\left(0,\sigma_{2}^{2}\right)$, and $X_{3}\thicksim N\left(0,\sigma_{3}^{2}\right)$,
the PDF of their product, $W=X_{1}X_{2}X_{3}$, is available analytically
as 
\[
p(t)=\frac{G_{0,3}^{3,0}\left(0,0,0;\frac{t^{2}}{8\sigma_{1}^{2}\sigma_{2}^{2}\sigma_{3}^{2}}\right)}{\left(2\pi\right)^{3/2}\sigma_{1}\sigma_{2}\sigma_{3}},
\]
where $G_{0,3}^{3,0}$ is a special case of Meijer G-function, see
\cite[Definition 9.301]{GRA-RYZ:2007}. The computed product PDF $p_{W}$
is obtained by multiplying $X_{3}\sim N\left(0,1\right)$ by the GMRA
representation of the PDF of $Z=X_{1}X_{2}$, as computed in Example
\ref{sub:example0101}. The PDFs of the random variables $Z$, $X_{3}$,
and $W$ are illustrated in Figure~\ref{fig:example010101-1}. As
in Example~\ref{sub:example0101}, we use $100$ scales and $\alpha=0.25$
to recover the PDF $p$ with accuracy $\epsilon\approx10^{-13}$.
In Figure~\ref{fig:example010101-2} we display the GMRA coefficients
of $p_{W}$ and its logarithmic singularity resolved to the interval
$\left[-10^{-28},10^{-28}\right]$. We compute the relative error, 

\begin{equation}
\epsilon\left(x\right)=\log_{10}\left(\frac{\left|p_{W}\left(10^{x}\right)-p\left(10^{x}\right)\right|}{p\left(10^{x}\right)}\right),\,\,\,-30\le x\le0,\label{eq:error(x)-1}
\end{equation}
for different values of the parameter $\alpha$ in the GMRA and display
them in Figure~\ref{fig:example010101-2}. 

\begin{figure}[H]
\begin{centering}
\includegraphics[scale=0.3]{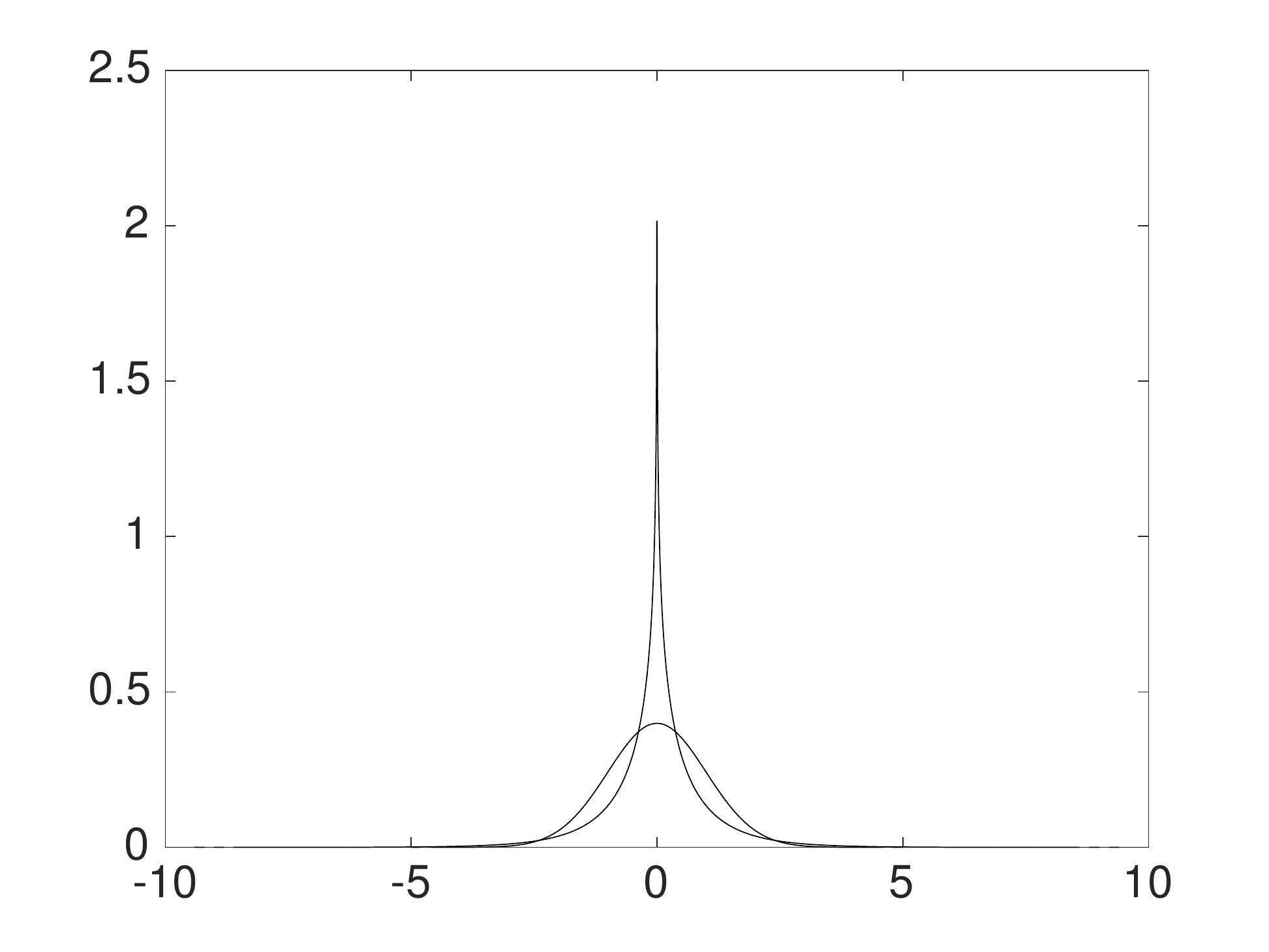}\includegraphics[scale=0.3]{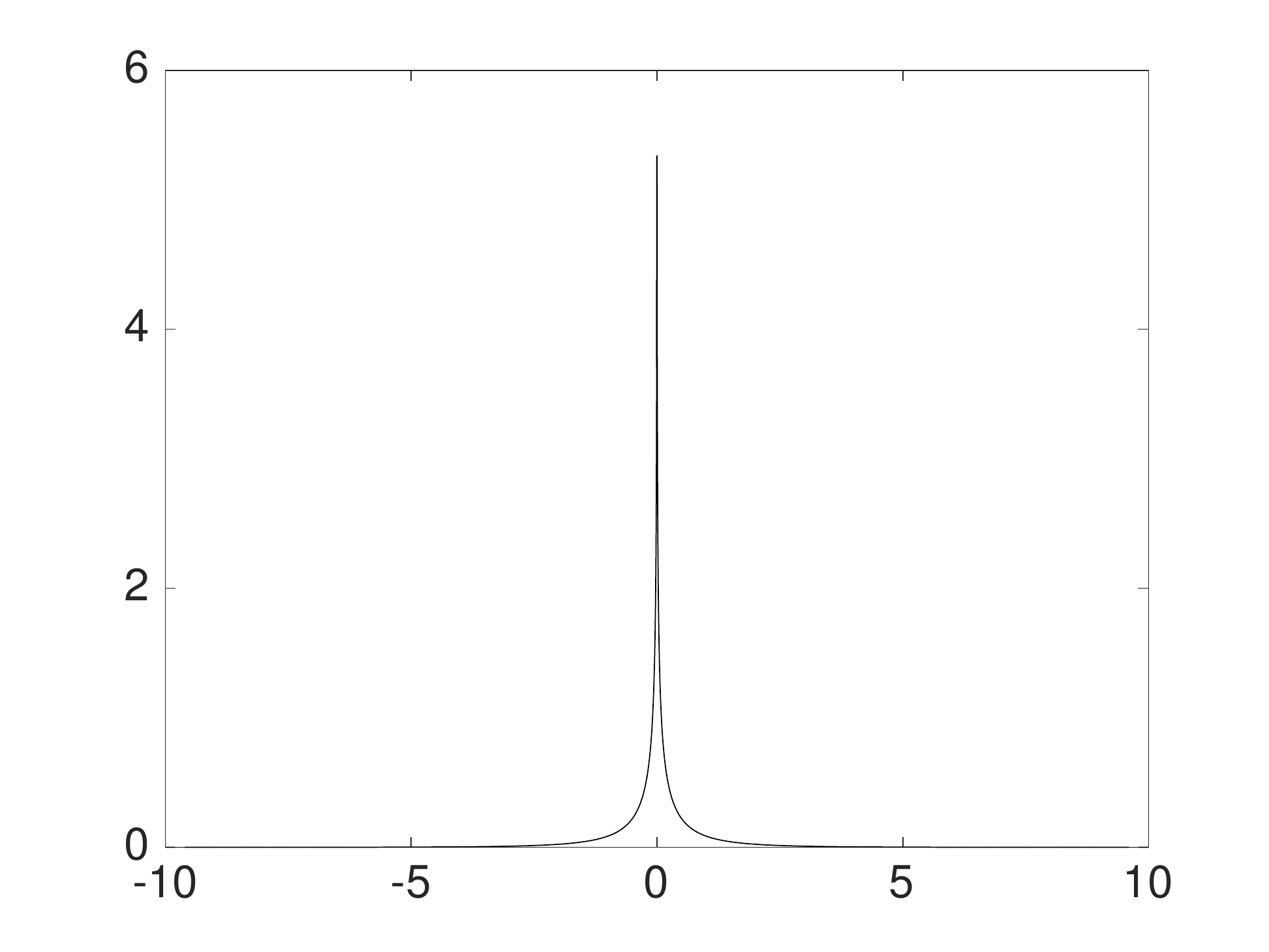}
\par\end{centering}

\centering{}\caption{\label{fig:example010101-1} PDFs of the random variables $Z=X_{1}X_{2}$
and $X_{3}$ in Example~\ref{sub:example010101} (left) and PDF of
the product random variable $W=ZX_{3}$ in Example~\ref{sub:example010101}
(right).}
\end{figure}

\begin{figure}[H]
\begin{centering}
\includegraphics[scale=0.3]{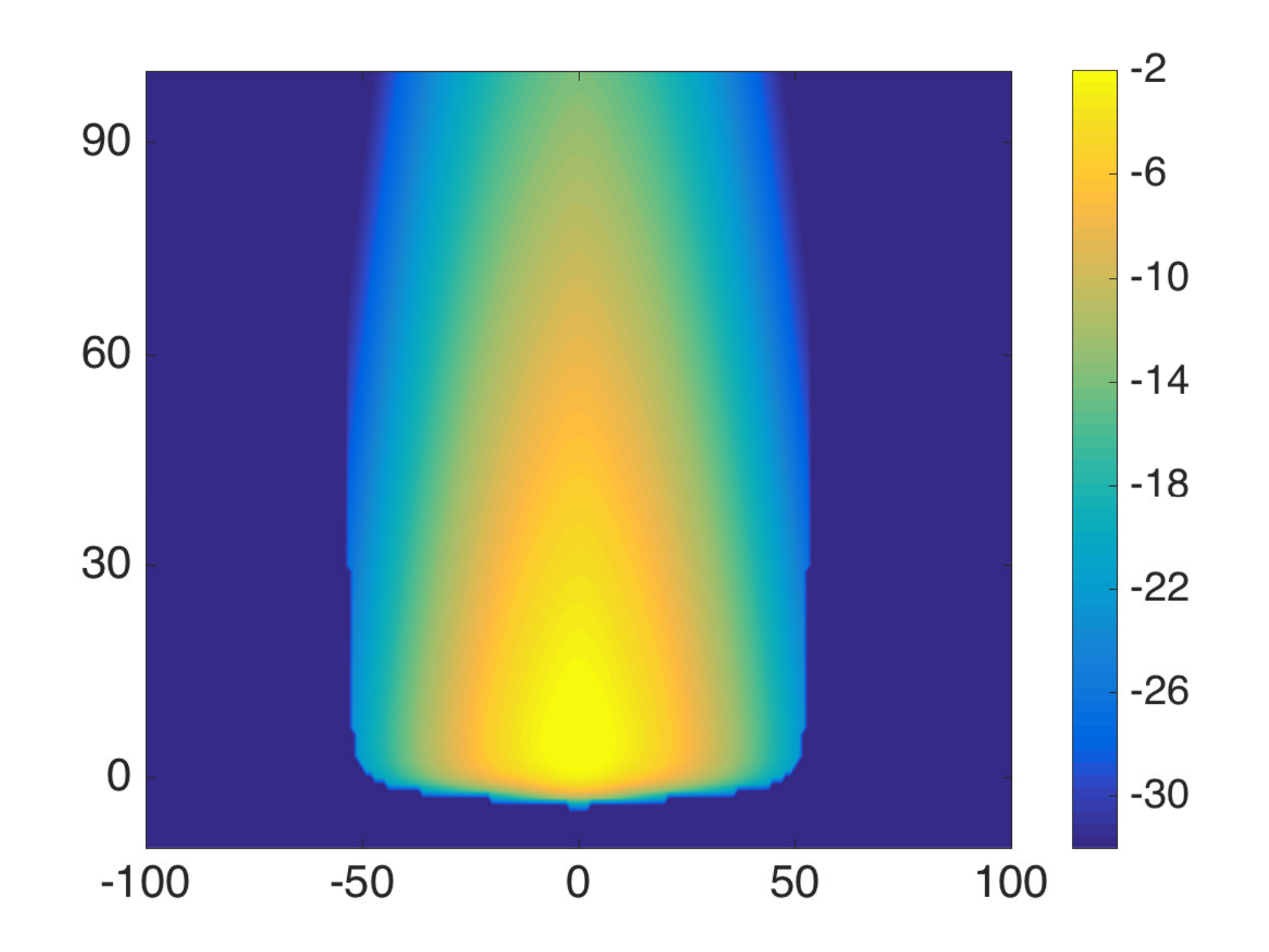}\includegraphics[scale=0.3]{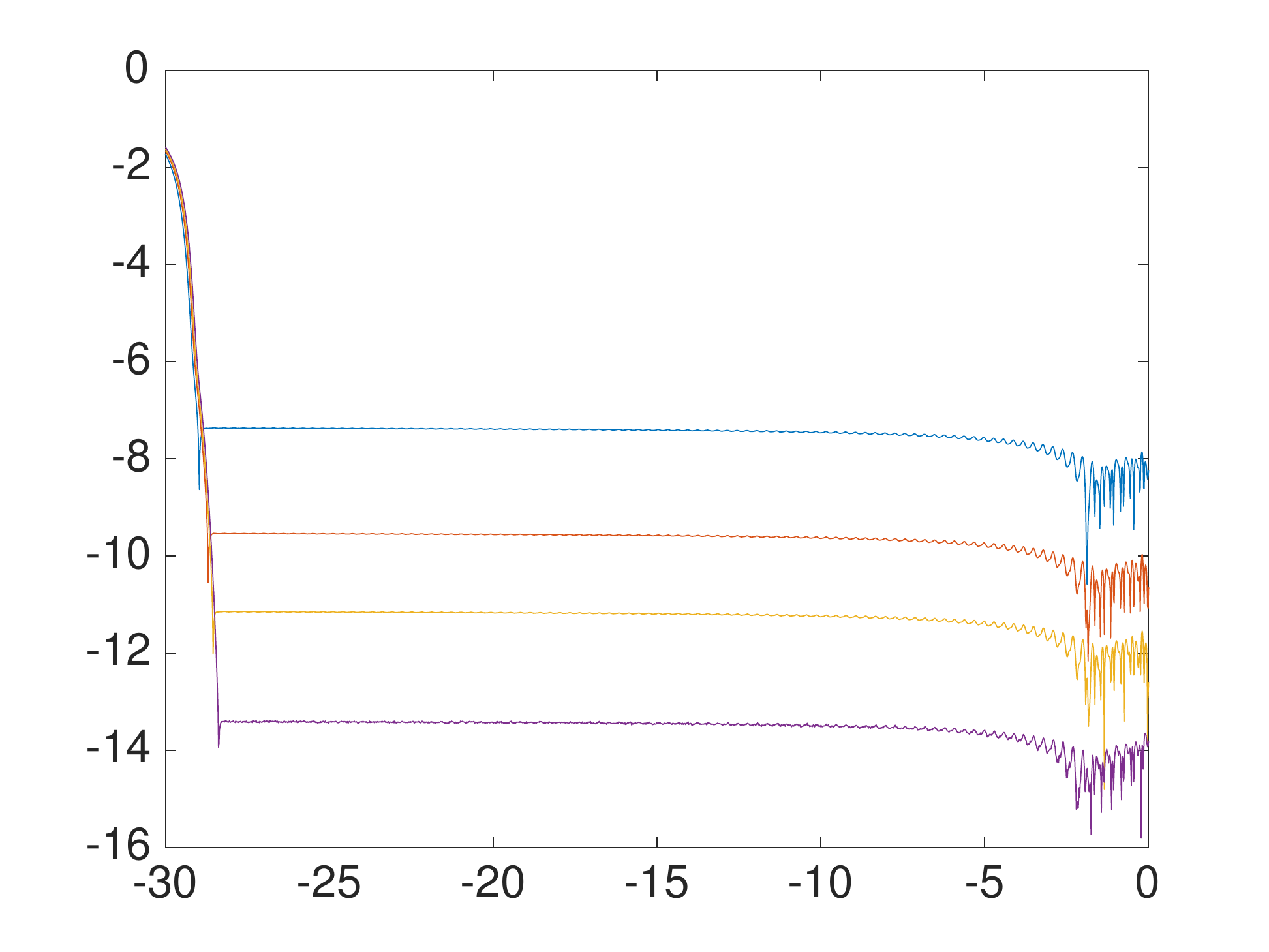}
\par\end{centering}

\centering{}\caption{\label{fig:example010101-2} Logarithm (base 10) of GMRA coefficients
of the PDF $p_{W}$ in Example~\ref{sub:example010101} (left) and
the relative error curves \eqref{eq:error(x)-1} shown over the interval
$\left[10^{-30},0\right]$ for different parameters $\alpha$ in the
GMRA \eqref{eq:Gaussian multiresolution basis}. The bottom curve
corresponds to $\alpha=0.25$, followed by $\alpha=0.3$, $\alpha=0.35$
and $\alpha=0.4$, as we go up in the display, demonstrating resolution
up to the interval $\left[-10^{-28},10^{-28}\right]$.}
\end{figure}

\subsection{Examples of computing the PDFs of products of random variables with
various distributions}

In the next examples we use our algorithm to compute the PDFs of the
product of two independent random variables with different distributions.
In Section~\ref{sub:example2111} we compute the product of two independent
normal random variables with non-zero means. We comment on the impact
of the order of integrands in Section~\ref{sub:example6121}. In
Sections~\ref{sub:ExampleCauchy}-\ref{sub:ExampleGumbel} we compute
the product of a random variable with either Cauchy, Laplace or Gumbel
distributions with a second normally distributed random variable.
Finally, in Section~\ref{sub:The-Laplace-and-Gumbel} we compute
the product of two independent random variables with Laplace and Gumbel
distributions where both distributions are approximated via Gaussian
mixtures as described in Section~\ref{sub:Representing-PDFs-in-GMRA}.

\subsubsection{\label{sub:example2111}Normal PDFs with non-zero means}

For the random variables $X\thicksim N(\mu_{x},\sigma_{x}^{2})$,
$Y\thicksim N(\mu_{y},\sigma_{y}^{2})$ with $\mu_{x}=2$, $\sigma_{x}=1$
and $\mu_{y}=1$, $\sigma_{y}=1$, we compute the PDF $p_{Z}$ of
their product, $Z=XY$. Figure~\ref{fig:Example2111-1} shows the
PDFs of the random variables $X$, $Y$ and $Z$. As far as we know,
there is no analytic expression for the result which we compute with
accuracy $\approx10^{-13}$. Using $100$ scales, we resolve the logarithmic
singularity at zero within an interval as small as $\left[-10^{-26},10^{-26}\right]$
(see Figure~\ref{fig:Example2111-2}). We observe that to compute
the PDF of $Z$ to such accuracy and resolution via a Monte-Carlo
method is not feasible. Moreover, we obtain the PDF of $Z$ in a functional
form which can be used for further computations. In contrast, a Monte-Carlo
method using $10^{8}$ samples and $10^{3}$ bins of approximate size
$0.04$ yields just the histogram shown in Figure~\ref{fig:Example2111-3}. 

\begin{figure}[h]
\begin{centering}
\includegraphics[scale=0.3]{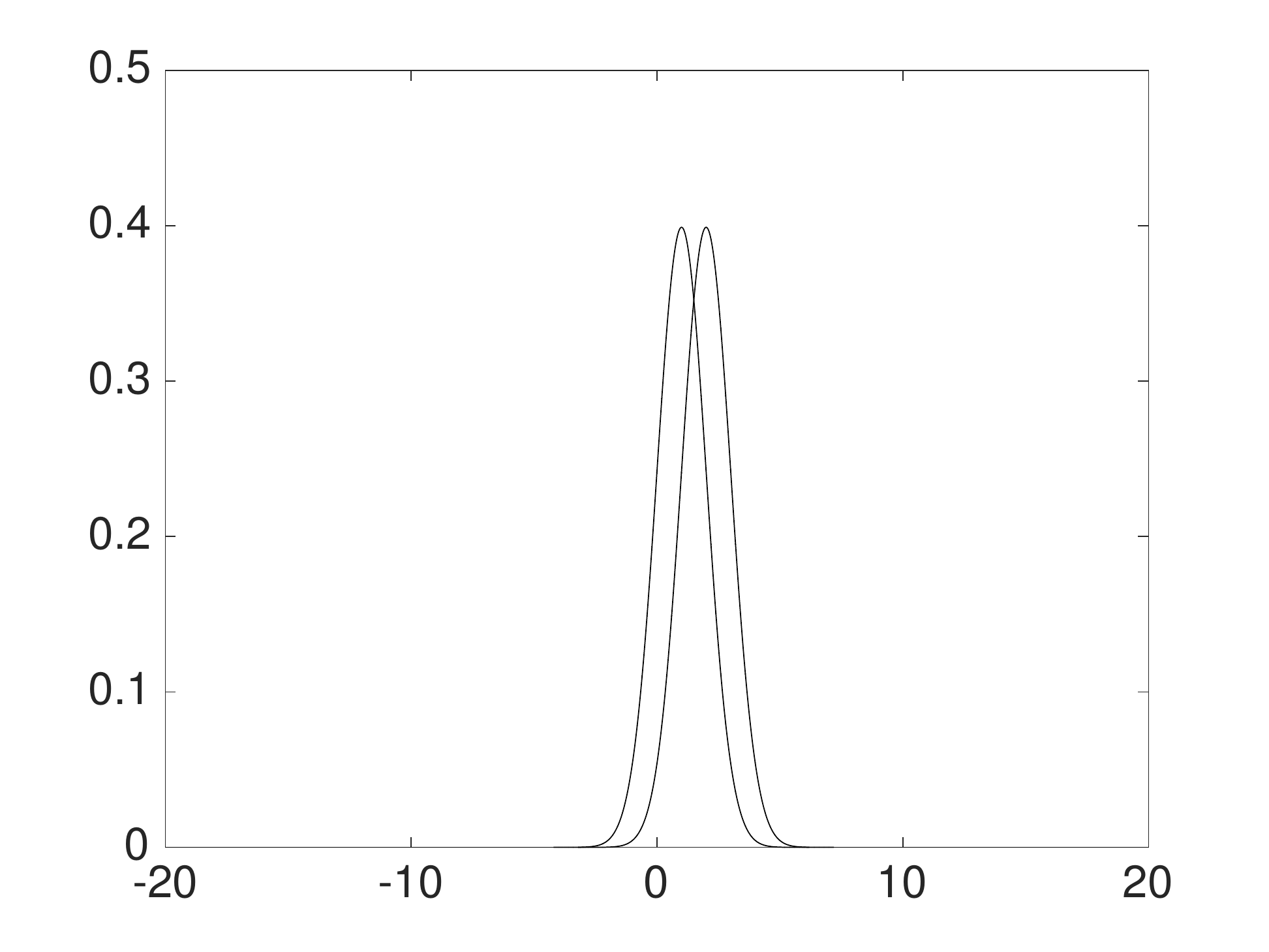}\includegraphics[scale=0.3]{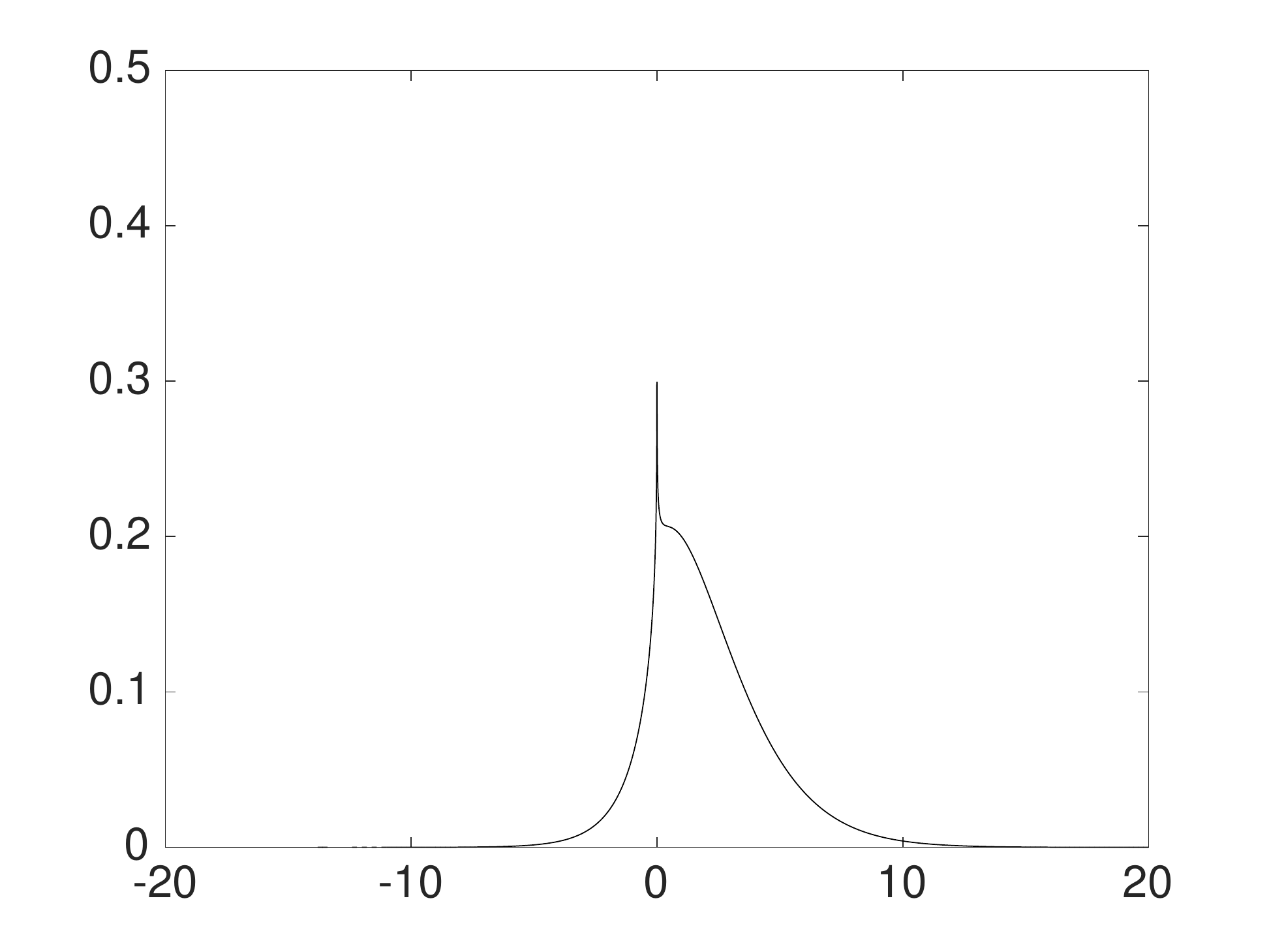}
\par\end{centering}

\centering{}\caption{\label{fig:Example2111-1} PDFs of the random variables $X$ and $Y$
of Example~\ref{sub:example2111} (left) and the computed PDF $p_{Z}$
of Example~\ref{sub:example2111} (right).}
\end{figure}

\begin{figure}[h]
\begin{centering}
\includegraphics[scale=0.3]{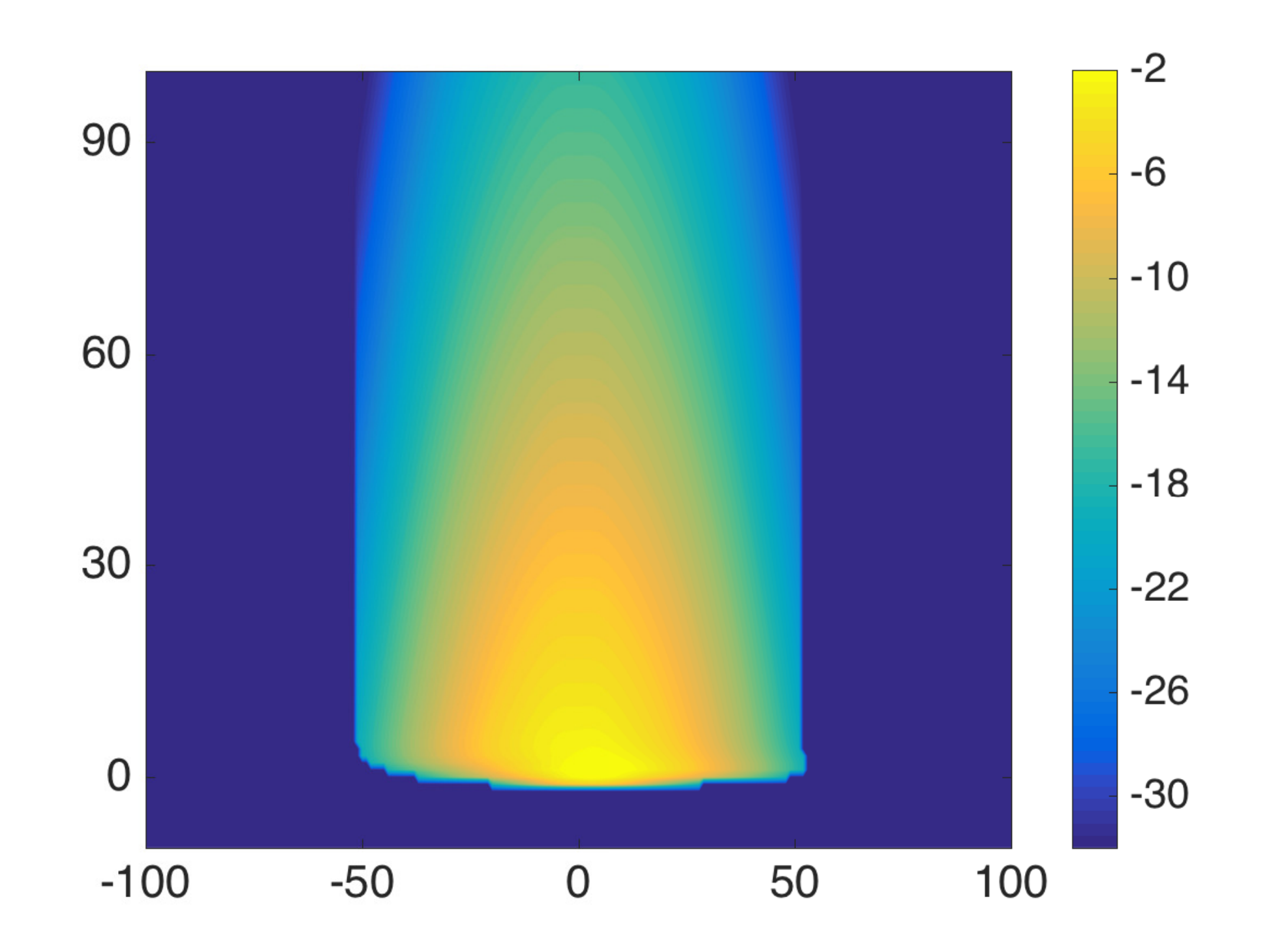}\includegraphics[scale=0.3]{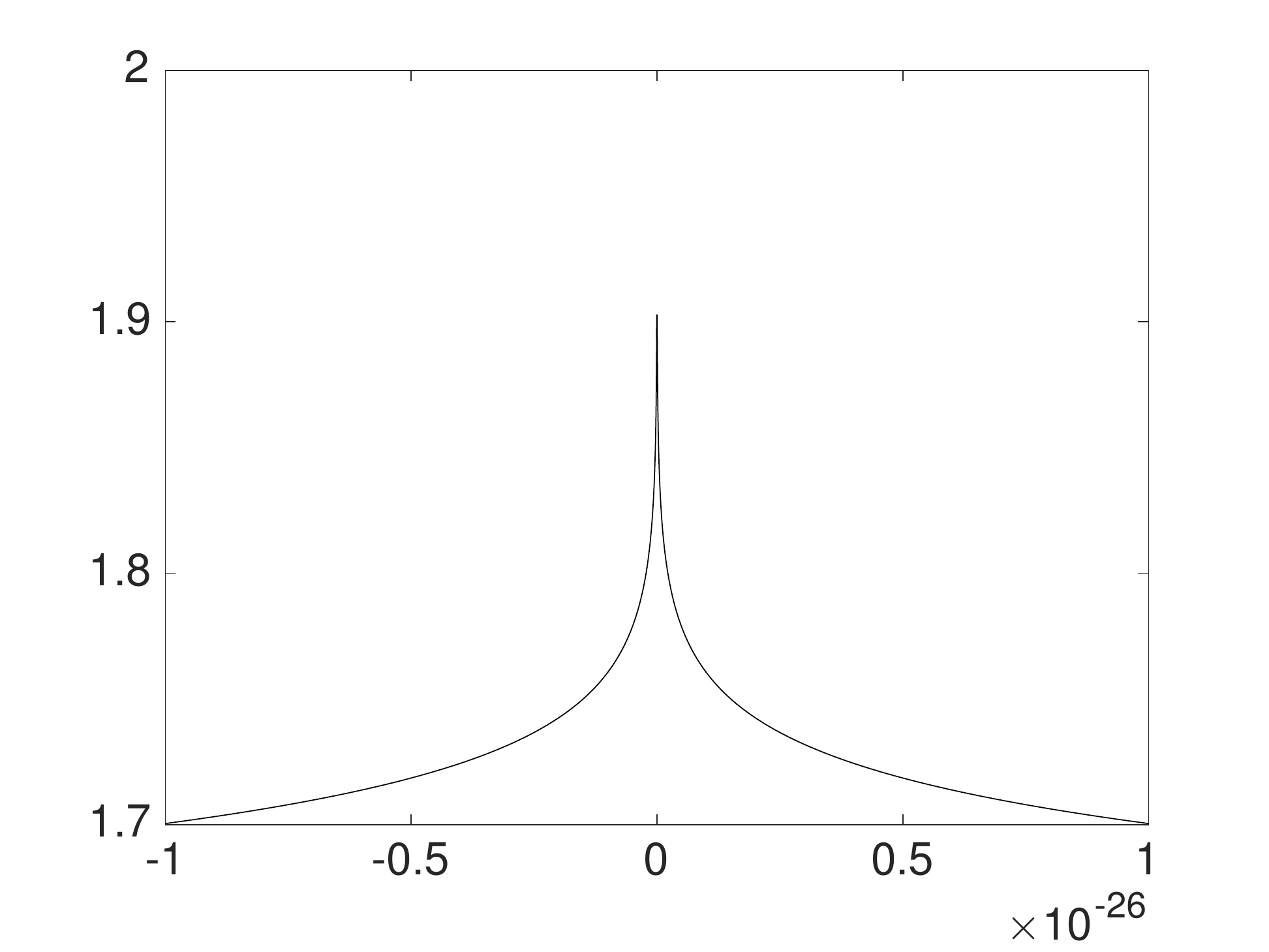}
\par\end{centering}

\centering{}\caption{\label{fig:Example2111-2} Logarithm (base 10) of GMRA coefficients
of the PDF $p_{Z}$ in Example~\ref{sub:example2111} (left) and
logarithmic singularity of $p_{Z}$ resolved to the interval $\left[-2\cdot10^{-26},2\cdot10^{-26}\right]$.}
\end{figure}

\begin{figure}[H]
\includegraphics[scale=0.4]{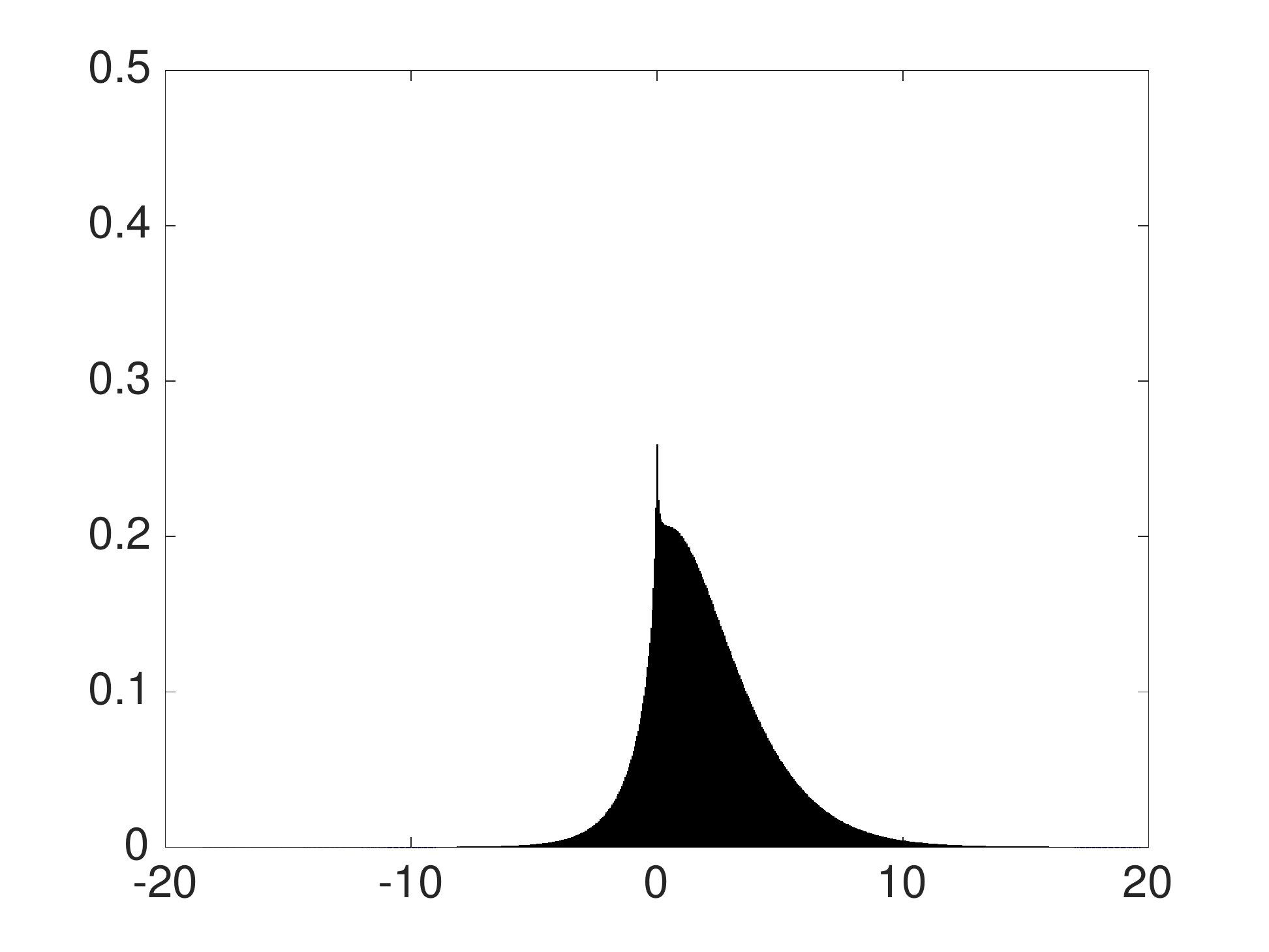}\caption{\label{fig:Example2111-3}A histogram of the distribution of Example~\ref{sub:example2111}
computed with $10^{8}$ samples using $10^{3}$ bins of approximate
size $0.04$. The logarithmic singularity at the origin is poorly
resolved (cf. the resolution of our method illustrated in Figure~\ref{fig:Example2111-2}).}
\end{figure}

\subsubsection{\label{sub:example6121}Order of the integrands when computing with
normal PDFs with non-zero means}

We demonstrate the effect of the order of factors on the coefficients
of the representation of the product as described in Remark~\ref{Ordering remark}.
We compute the PDFs of $Z_{1}=XY$ and $Z_{2}=YX$, where $X\sim N(6,1)$
and $Y\sim N(2,1)$, and display the results in Figures~\ref{fig:Example6121-1}
and \ref{fig:Example6121-3}. While the values of the PDFs $p_{Z_{1}}$
and $p_{Z_{2}}$ are the same within the accuracy of computation,
the coefficients of $p_{Z_{1}}$ and $p_{Z_{2}}$ do depend on the
order of the factors as illustrated in Figure~\ref{fig:Example6121-3}.

\begin{figure}[h]
\begin{centering}
\includegraphics[scale=0.3]{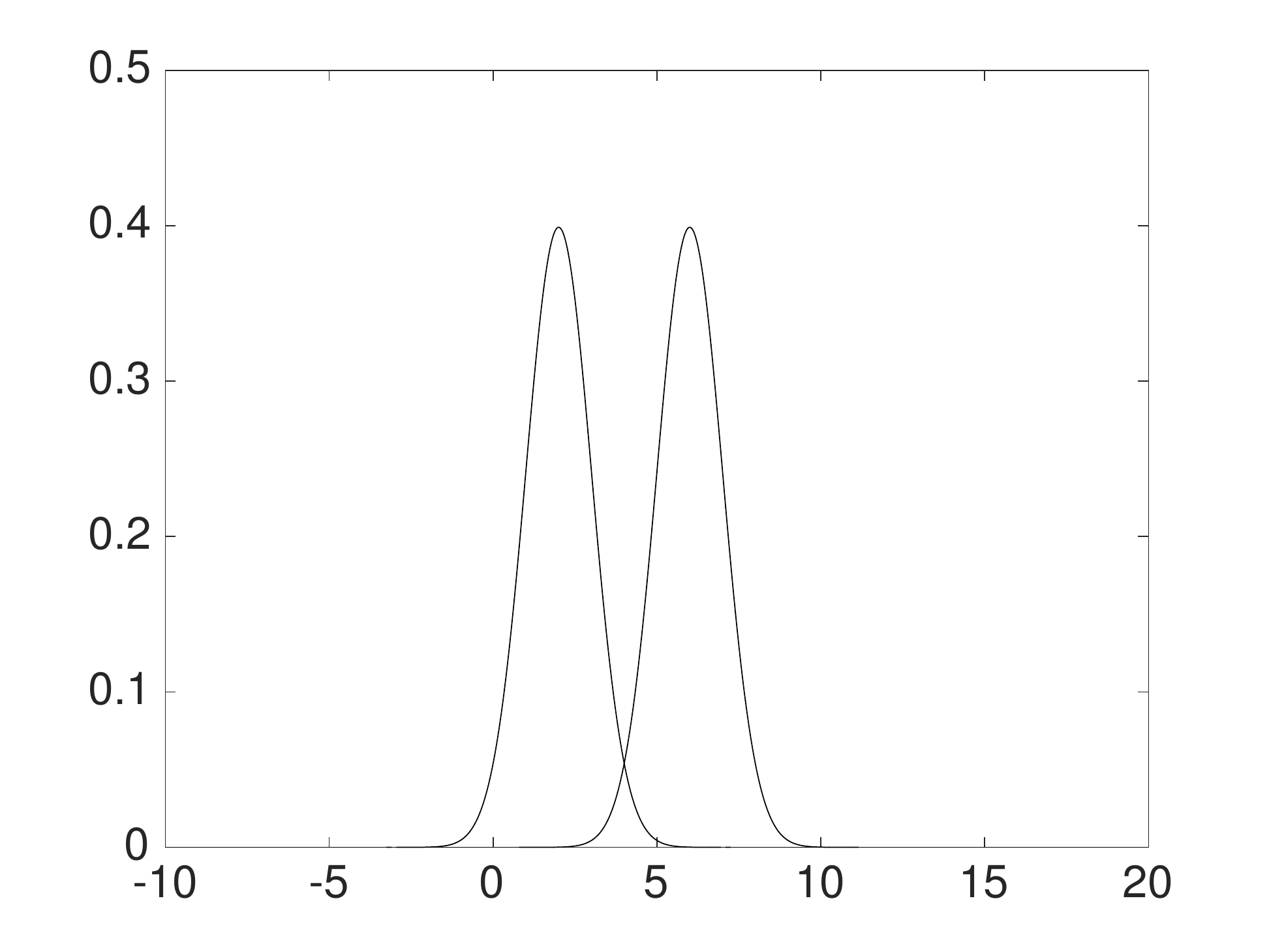}\includegraphics[scale=0.3]{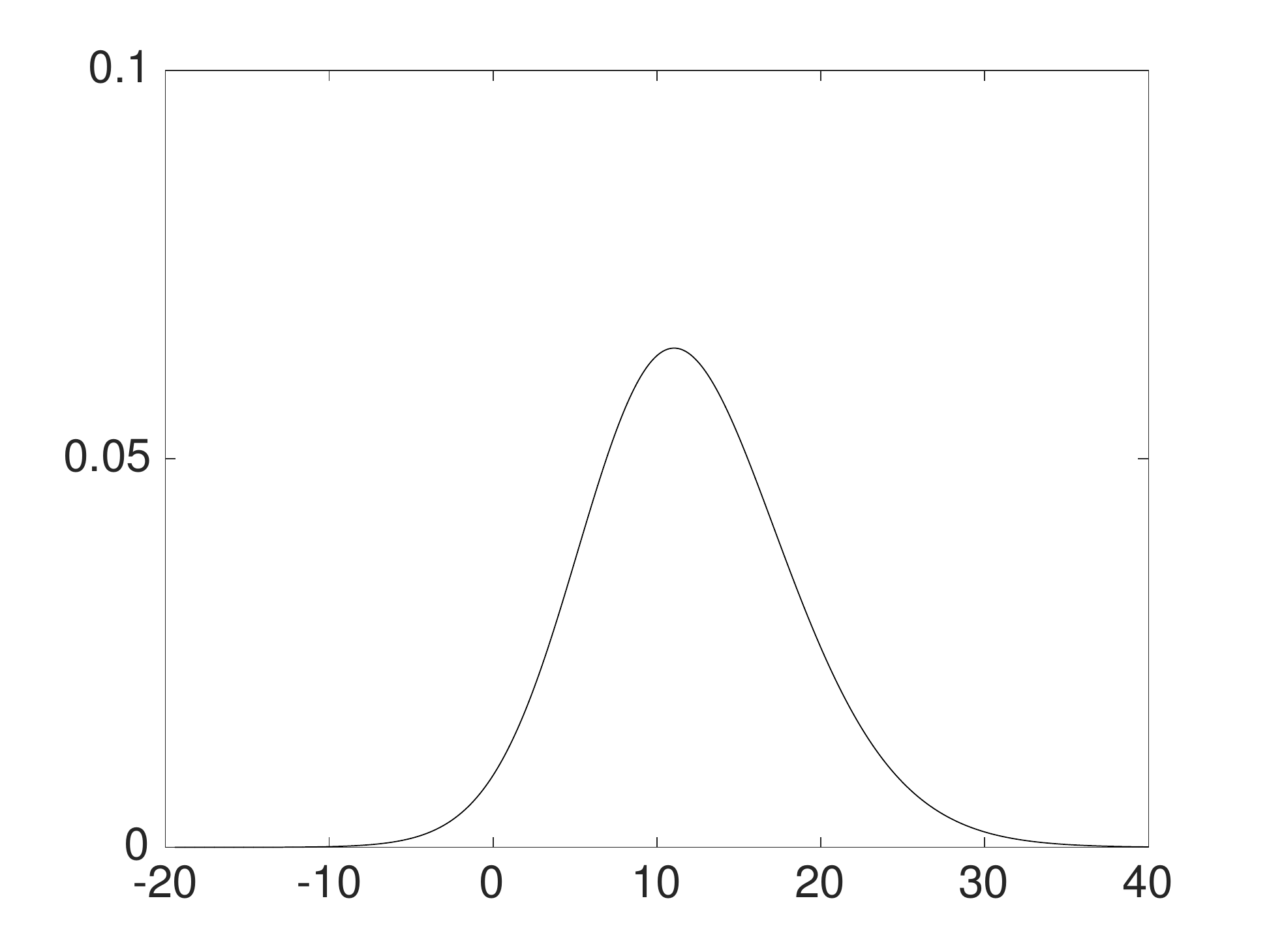}
\par\end{centering}

\centering{}\caption{\label{fig:Example6121-1} PDFs of random variables $X$ and $Y$
in Example~\ref{sub:example6121} (left) and computed product PDF
$p_{Z_{1}}$ (right).}
\end{figure}

\begin{figure}[H]
\begin{centering}
\includegraphics[scale=0.3]{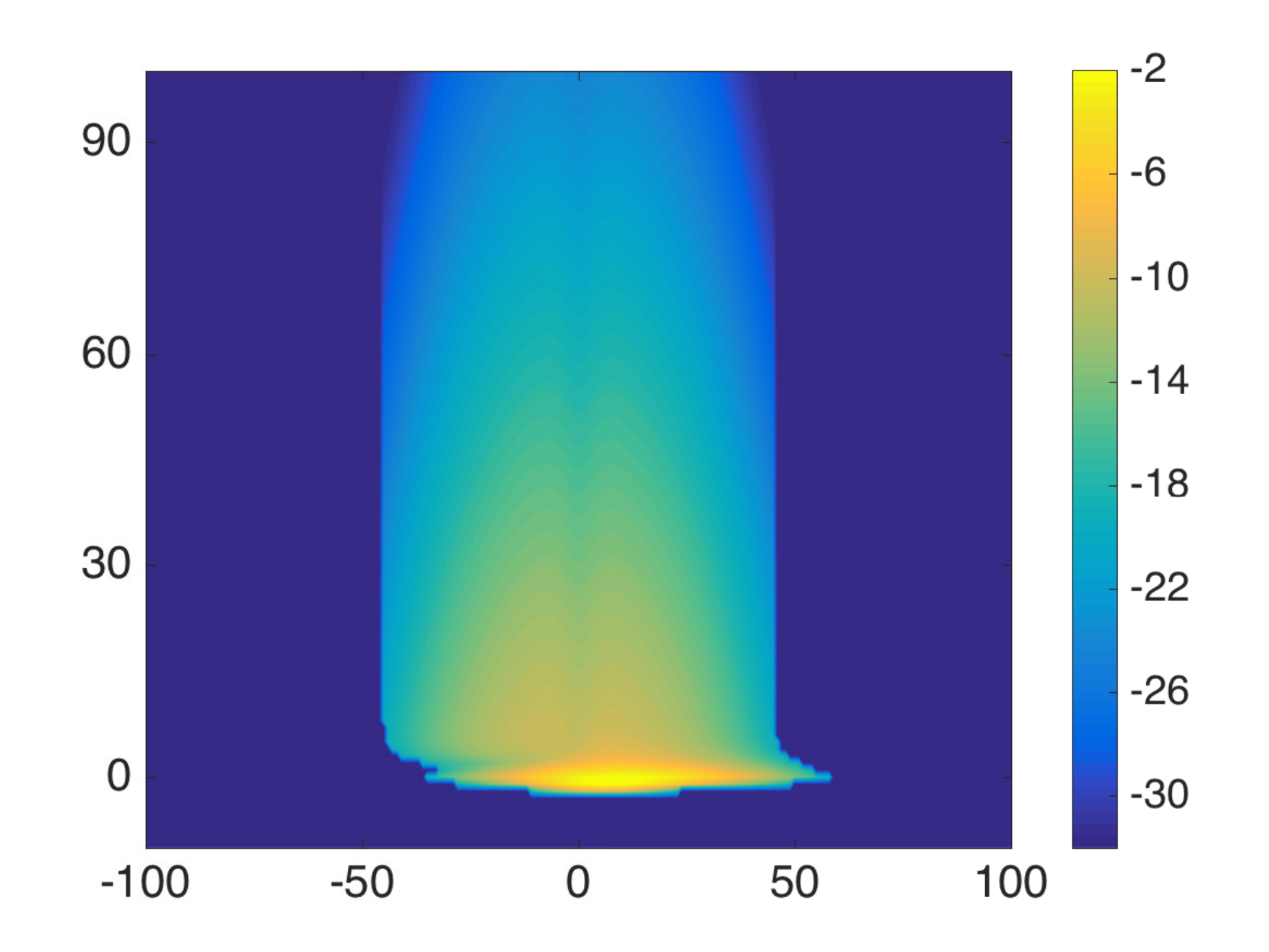}\includegraphics[scale=0.3]{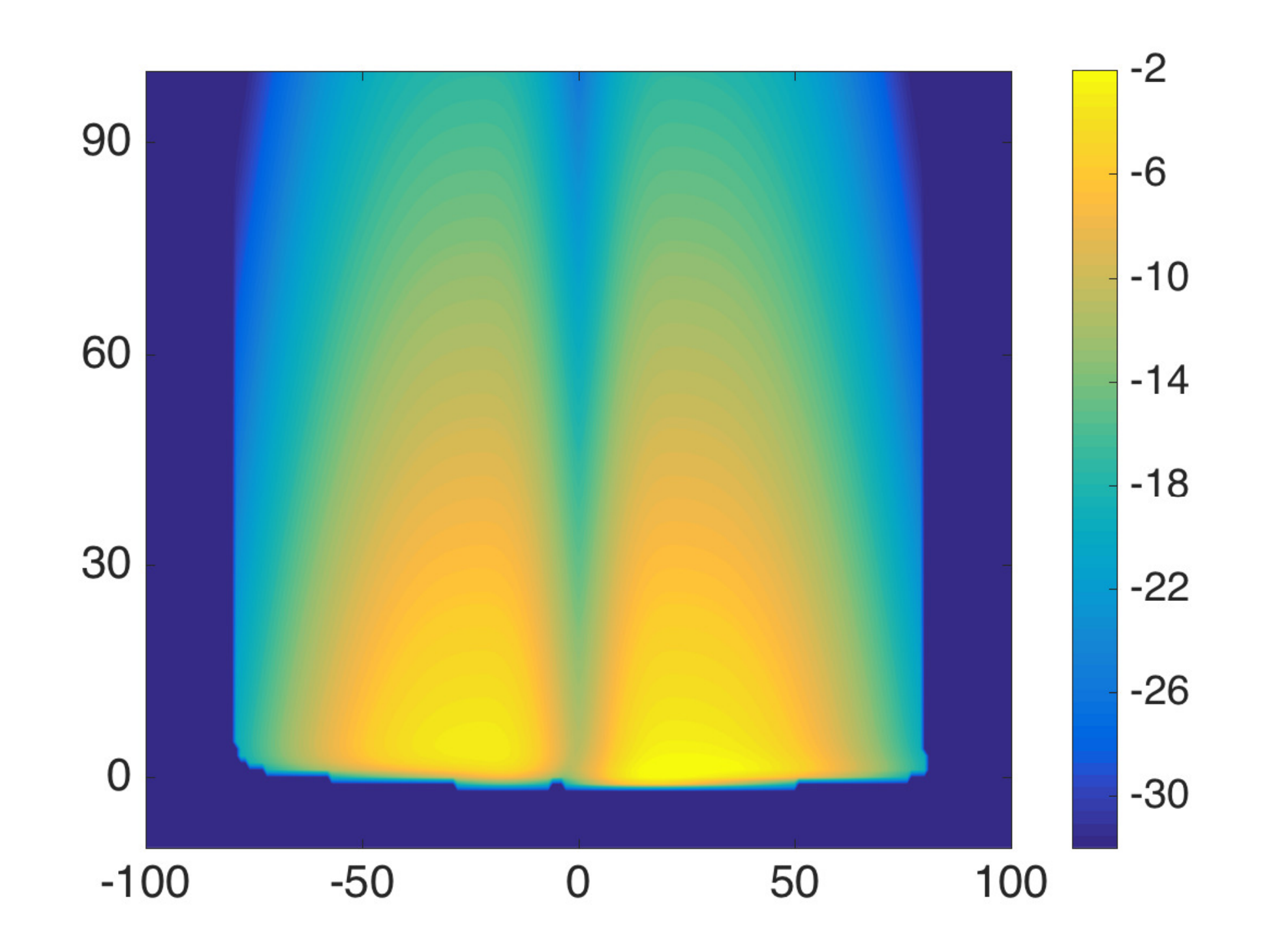}
\par\end{centering}

\centering{}\caption{\label{fig:Example6121-3}Logarithm (base 10) of GMRA coefficients
of the PDF $p_{Z_{1}}$ (left) and logarithm (base 10) of GMRA coefficients
of the PDF $p_{Z_{2}}$ (right).}
\end{figure}

\subsubsection{\label{sub:ExampleCauchy}Normal distribution and Cauchy distribution}

We now turn to computing the PDF of the product of random variables
where one of them has a non-Gaussian distribution. First, we consider
the PDF of the Cauchy distribution

\[
f\left(x;x_{0},\gamma\right)=\frac{1}{\pi\gamma}\left(\frac{\gamma^{2}}{\left(x-x_{0}\right)^{2}+\gamma^{2}}\right),
\]
where $x_{0}$ is the location parameter and $\gamma$ is the scale
parameter. The Cauchy distribution is defined on the real line and
is an example of a heavy-tailed distribution for which the moments
are infinite. In this example, we compute the PDF $p_{Z}$ of $Z=XY$,
where $X\sim f\left(x;-2,1\right)$ and $Y\sim N\left(1.5,1\right)$.
We used $40$ ``negative scales'', i.e. $j=-40,\dots100$ in order
to capture the heavy-tail behavior of the product distribution (which
also has no finite moments). By using so many coarse scales, the integral
of the resulting PDF over the real line differs from $1$ by approximately
$\approx10^{-14}$. PDFs of the factors $X$ and $Y$ and the product
$Z$ are illustrated in Figure~\ref{fig: Cauchy-1 }. In Figure~\ref{fig: Cauchy-2},
we display the GMRA coefficients of $p_{Z}$ and resolve the singularity
at zero of $p_{Z}$ within the interval $\left[-10^{-28},10^{-28}\right]$. 

\begin{figure}[H]
\begin{centering}
\includegraphics[scale=0.3]{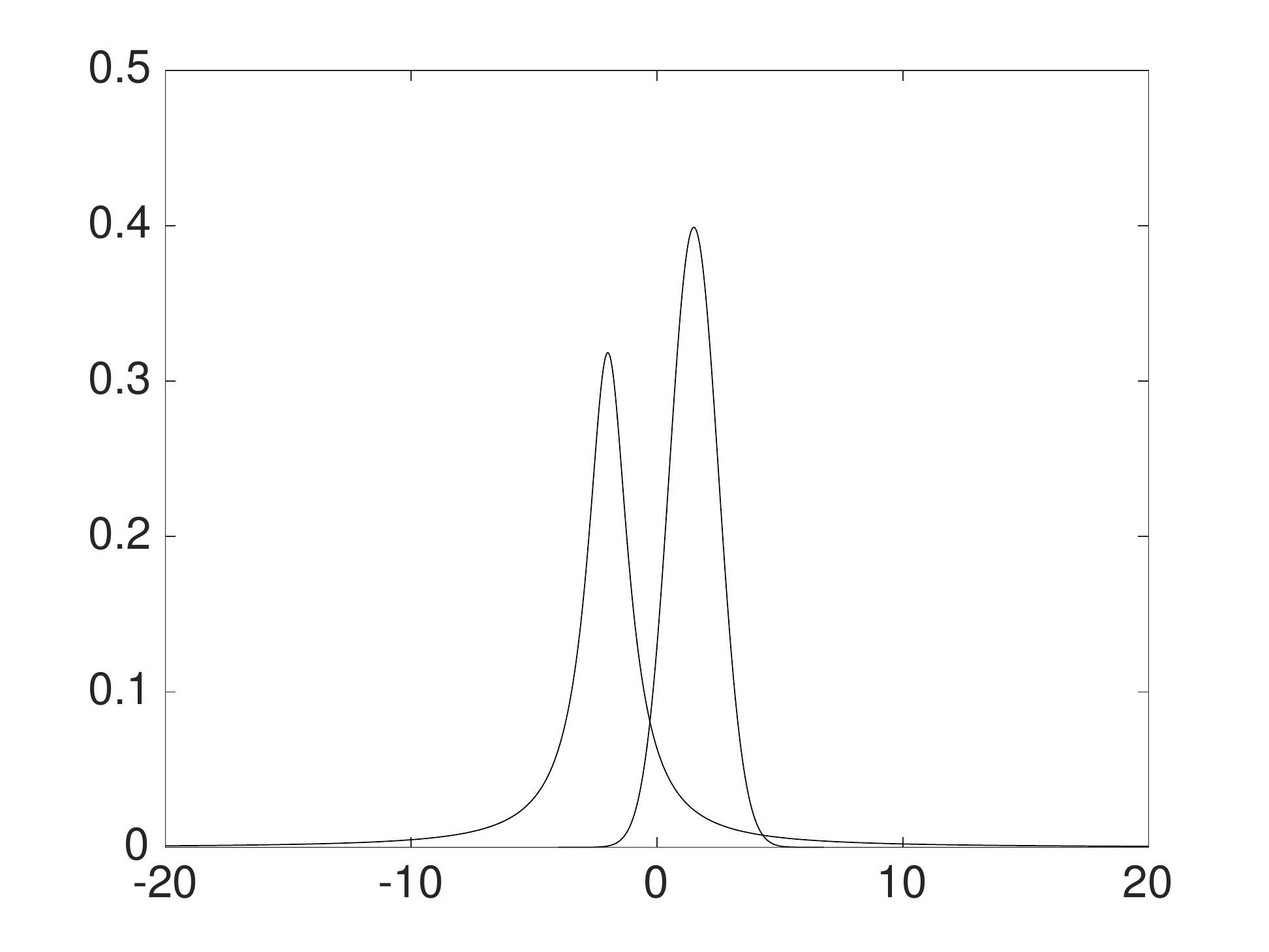}\includegraphics[scale=0.3]{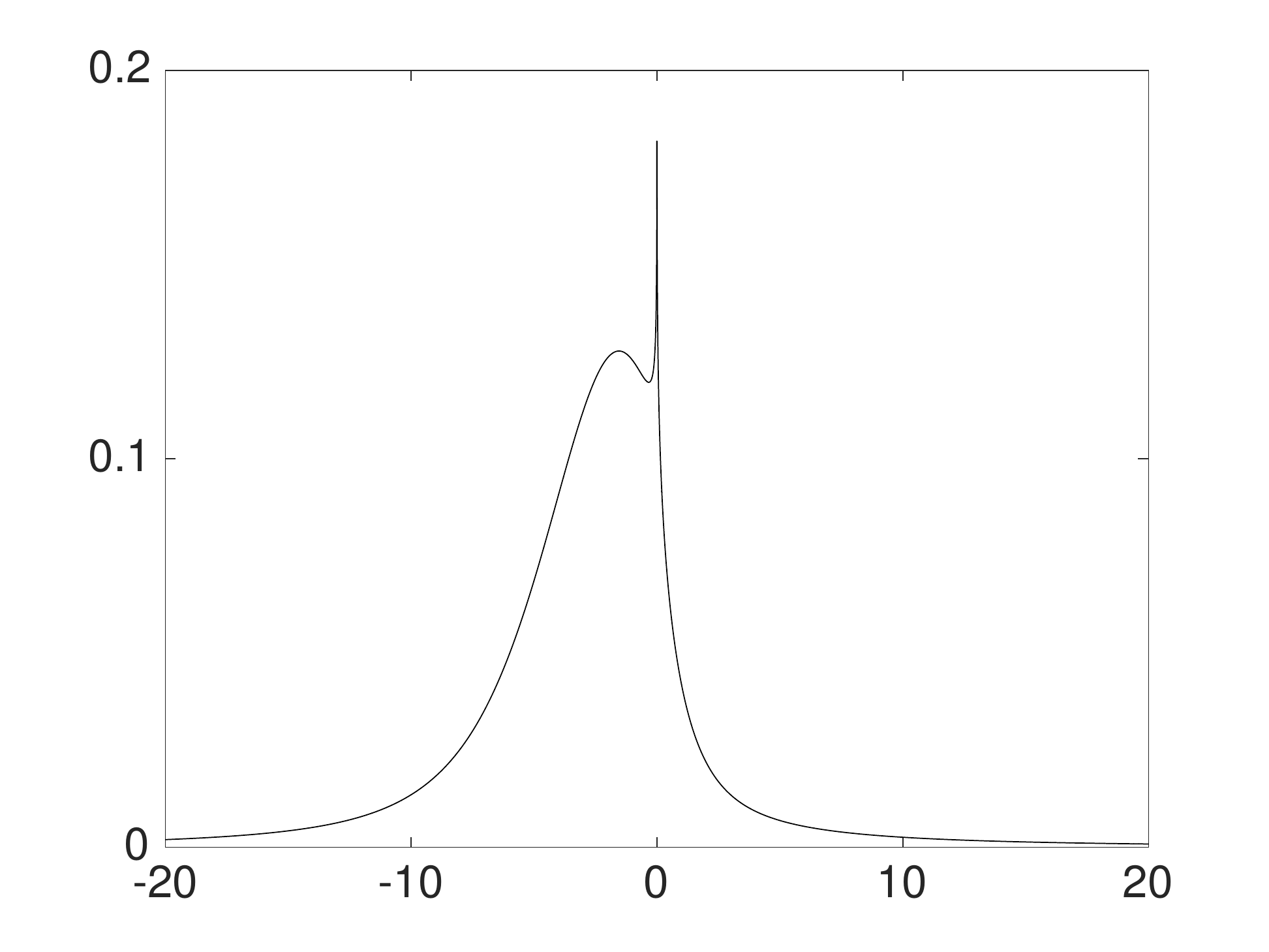}
\par\end{centering}

\centering{}\caption{\label{fig: Cauchy-1 } The PDFs of the random variables $X$ and
$Y$ in Example~\ref{sub:ExampleCauchy} (left) and computed product
PDF $p_{Z}$ (right).}
\end{figure}

\begin{figure}[H]
\begin{centering}
\includegraphics[scale=0.3]{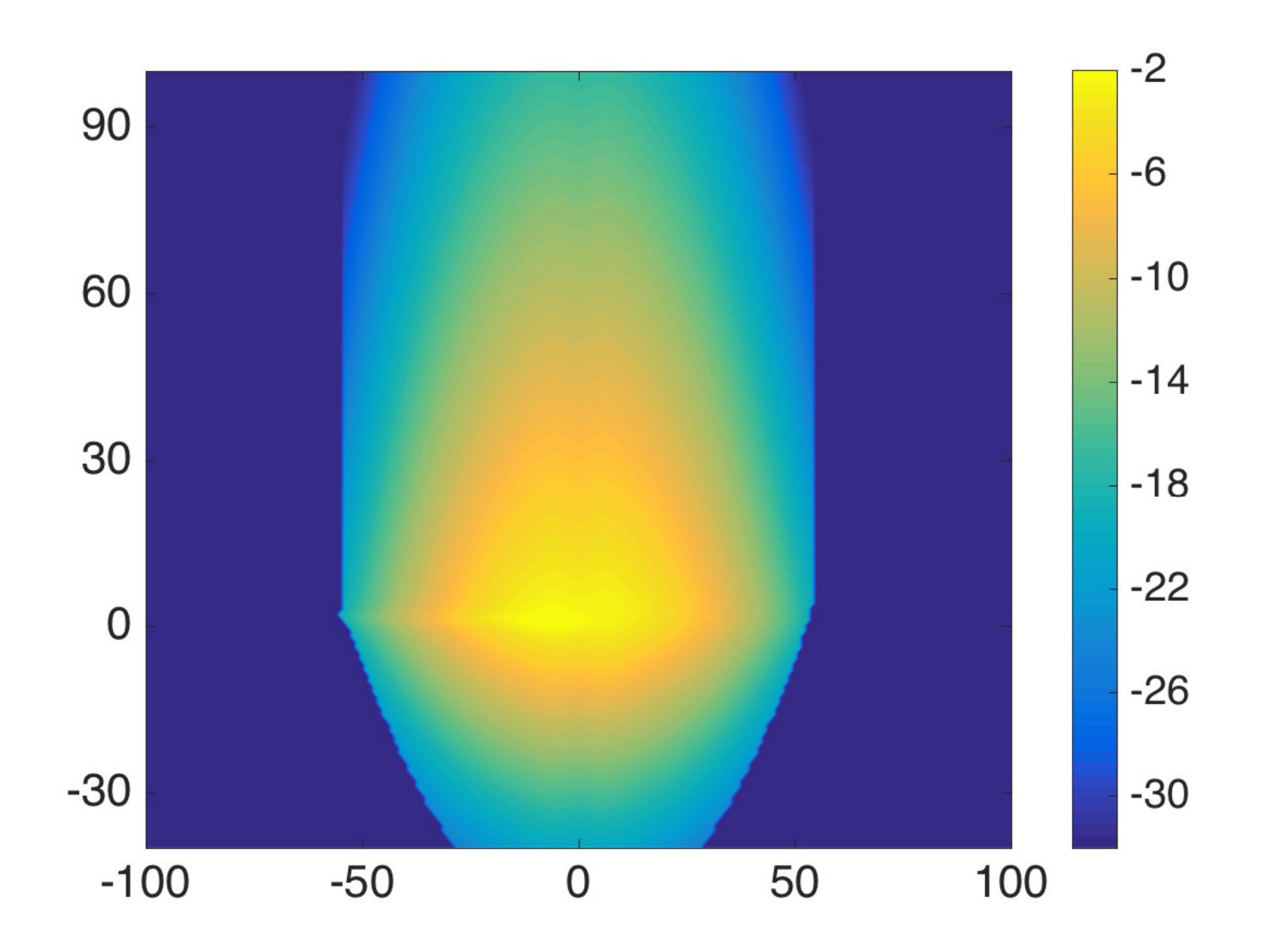}\includegraphics[scale=0.3]{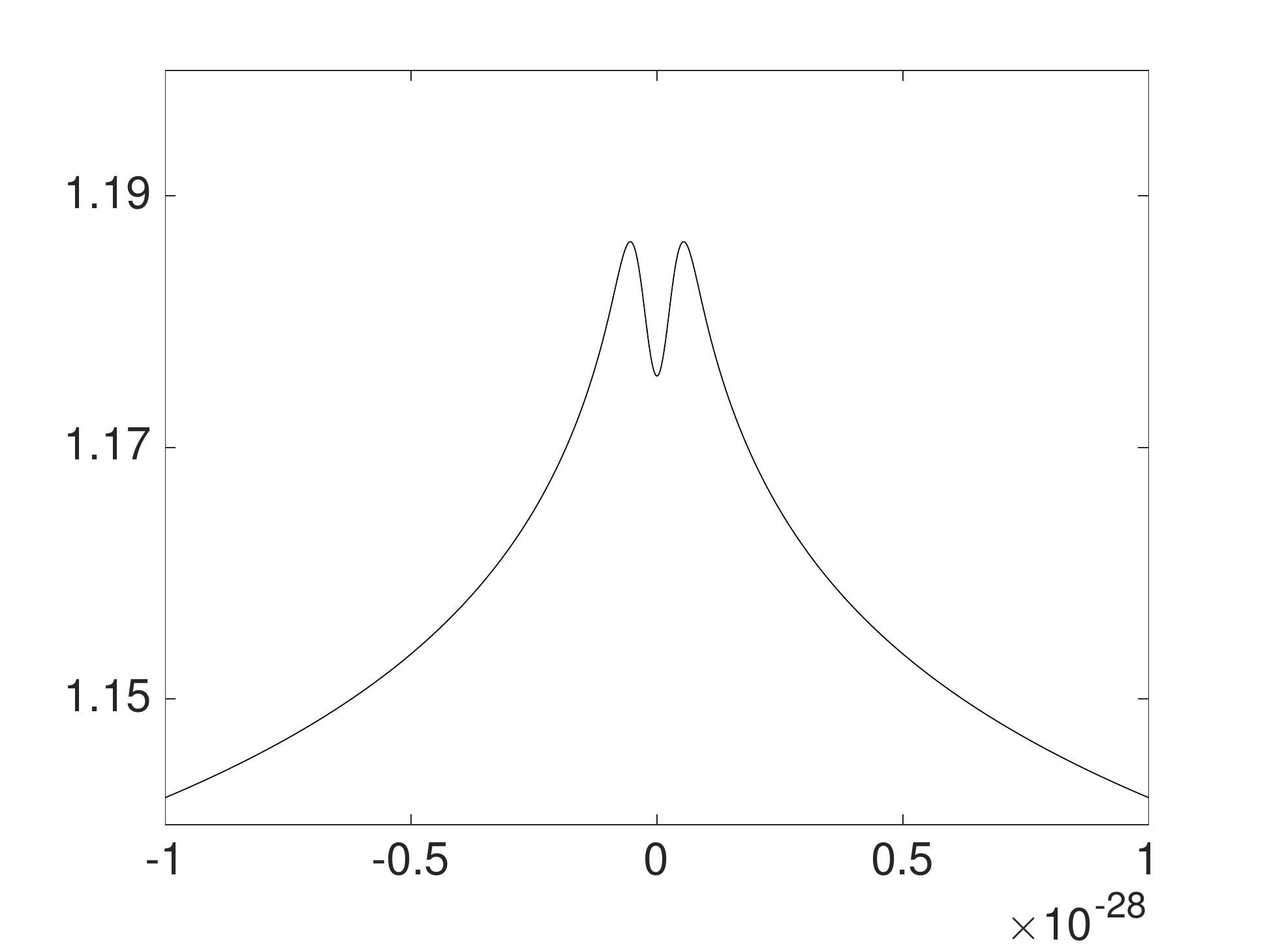}
\par\end{centering}

\centering{}\caption{\label{fig: Cauchy-2} Logarithm (base 10) of GMRA coefficients of
the PDF $p_{Z}$ in Example~\ref{sub:ExampleCauchy} (left) and logarithmic
singularity of $p_{Z}$ resolved to the interval $\left[-10^{-28},10^{-28}\right]$
(right). }
\end{figure}

\subsubsection{\label{sub:ExampleLaplace}Normal distribution and Laplace distribution}

We consider the PDF of the Laplace distribution

\begin{equation}
f\left(x;\mu,b\right)=\frac{1}{2b}e^{-\frac{\left|x-\mu\right|}{b}},\label{eq:Laplace distribution with parameters}
\end{equation}
where $\mu$ is the location parameter and $b>0$ is the scale parameter.
The density function, $f\left(x;\mu,b\right)$, has a cusp at $x=\mu$.
In this example, we compute the PDF $p_{Z}$ of the product $Z=XY$,
where $X\sim f(x;2,1)$ and $Y\sim N(1,1)$ and display these PDFs
in Figure \ref{fig: Laplace1}. In Figure~\ref{fig: Laplace2} we
display the GMRA coefficients of $p_{Z}$ and resolve, within the
interval $\left[-10^{-28},10^{-28}\right]$, the singularity of $p_{Z}$
at zero. Note that the resulting distribution does not have a singularity
at $x=\mu$ but has a singularity at zero as in the other examples. 

\begin{figure}[H]
\begin{centering}
\includegraphics[scale=0.3]{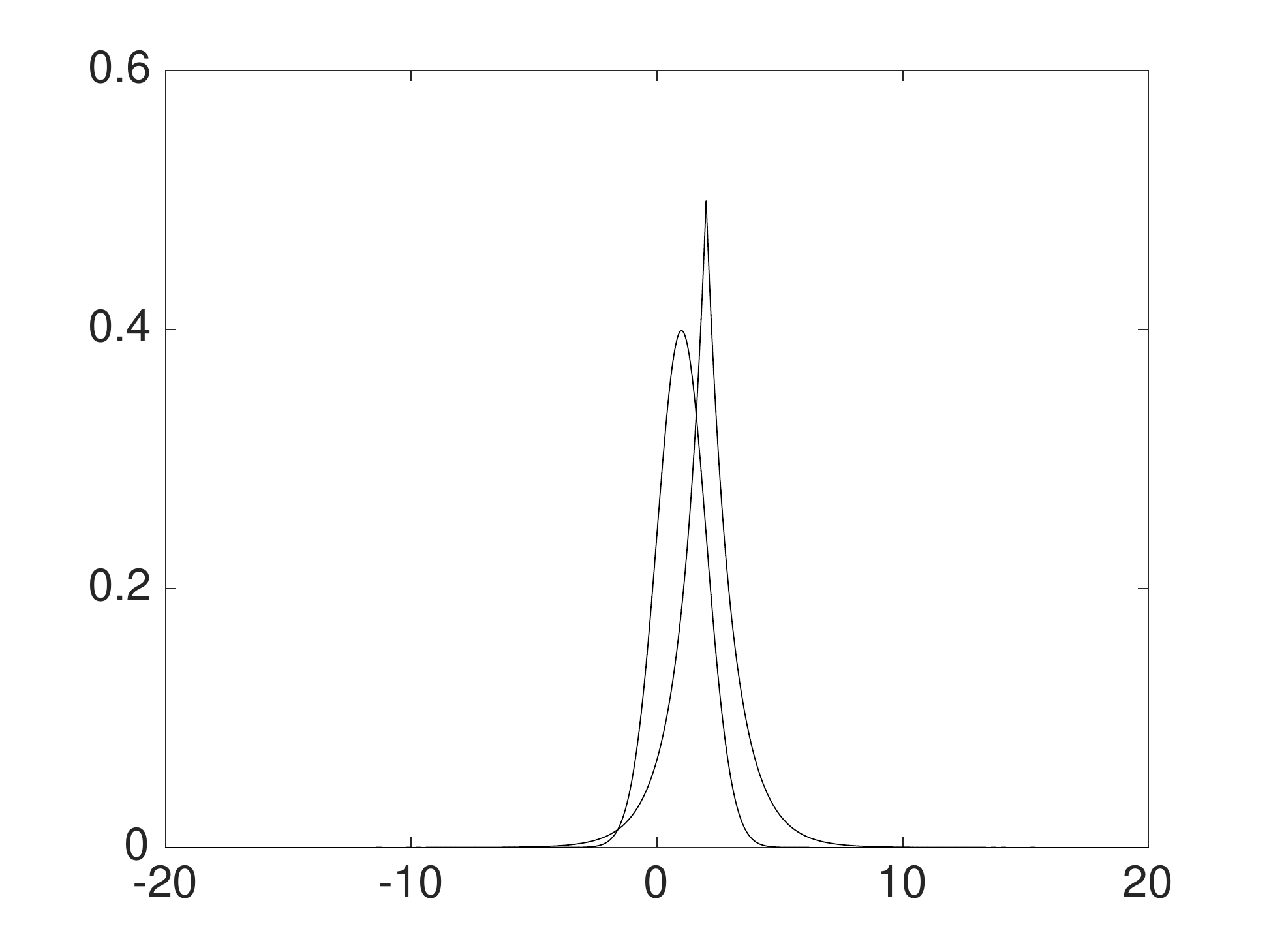}\includegraphics[scale=0.3]{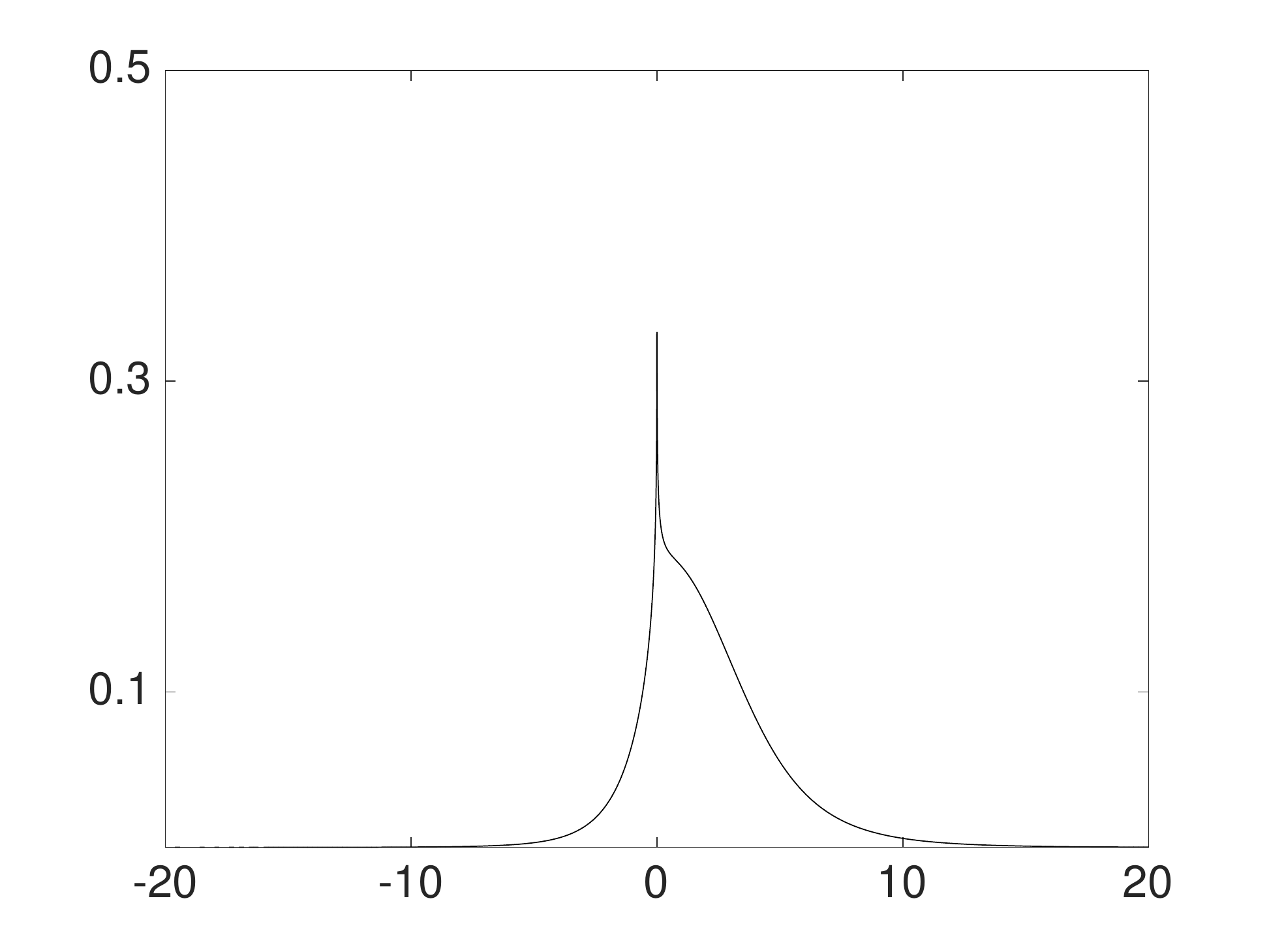}
\par\end{centering}

\centering{}\caption{\label{fig: Laplace1} The PDFs of the random variables $X$ and $Y$
in Example~\ref{sub:ExampleLaplace} (left) and computed product
PDF $p_{Z}$ (right). }
\end{figure}

\begin{figure}[H]
\begin{centering}
\includegraphics[scale=0.3]{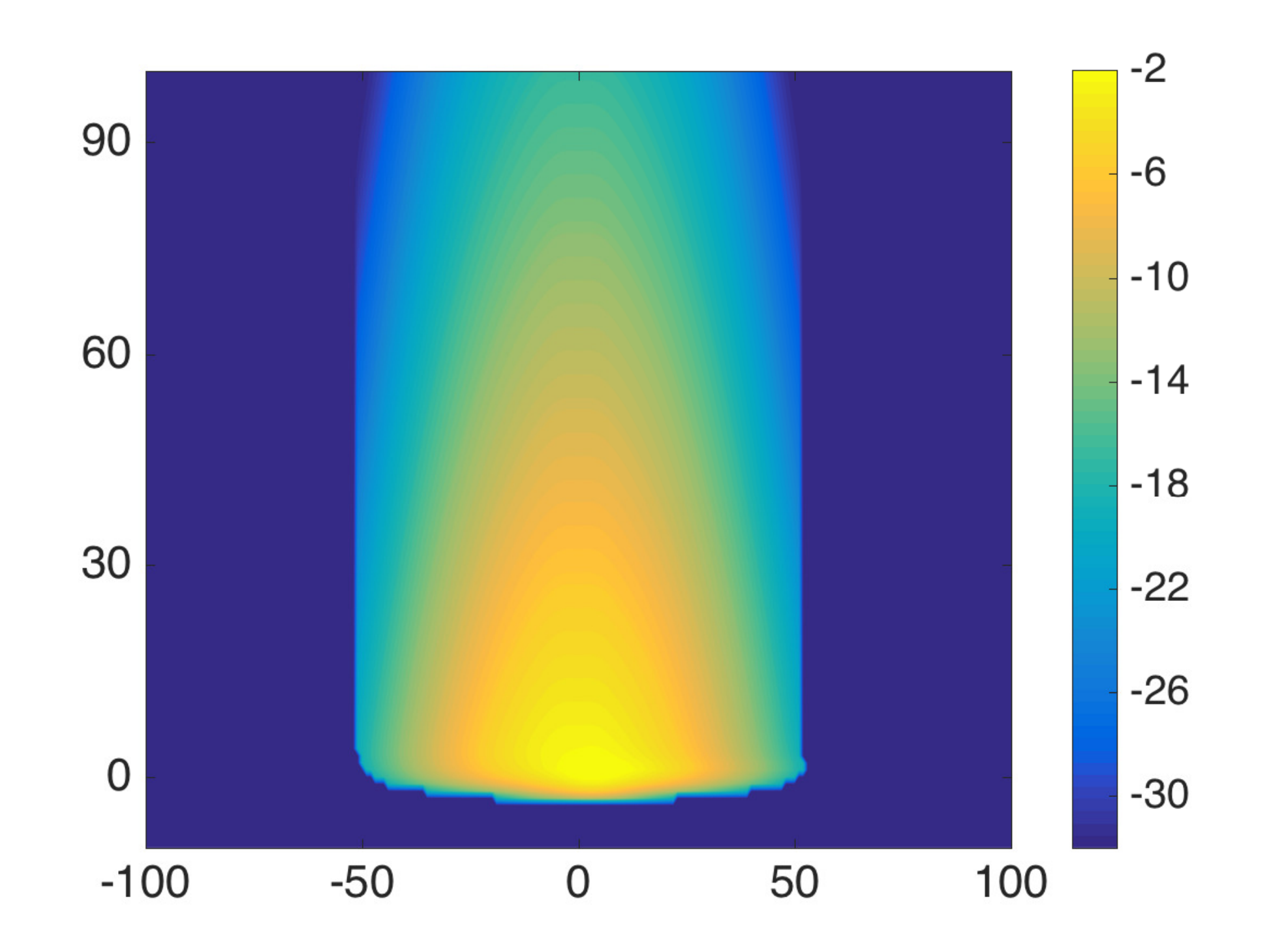}\includegraphics[scale=0.3]{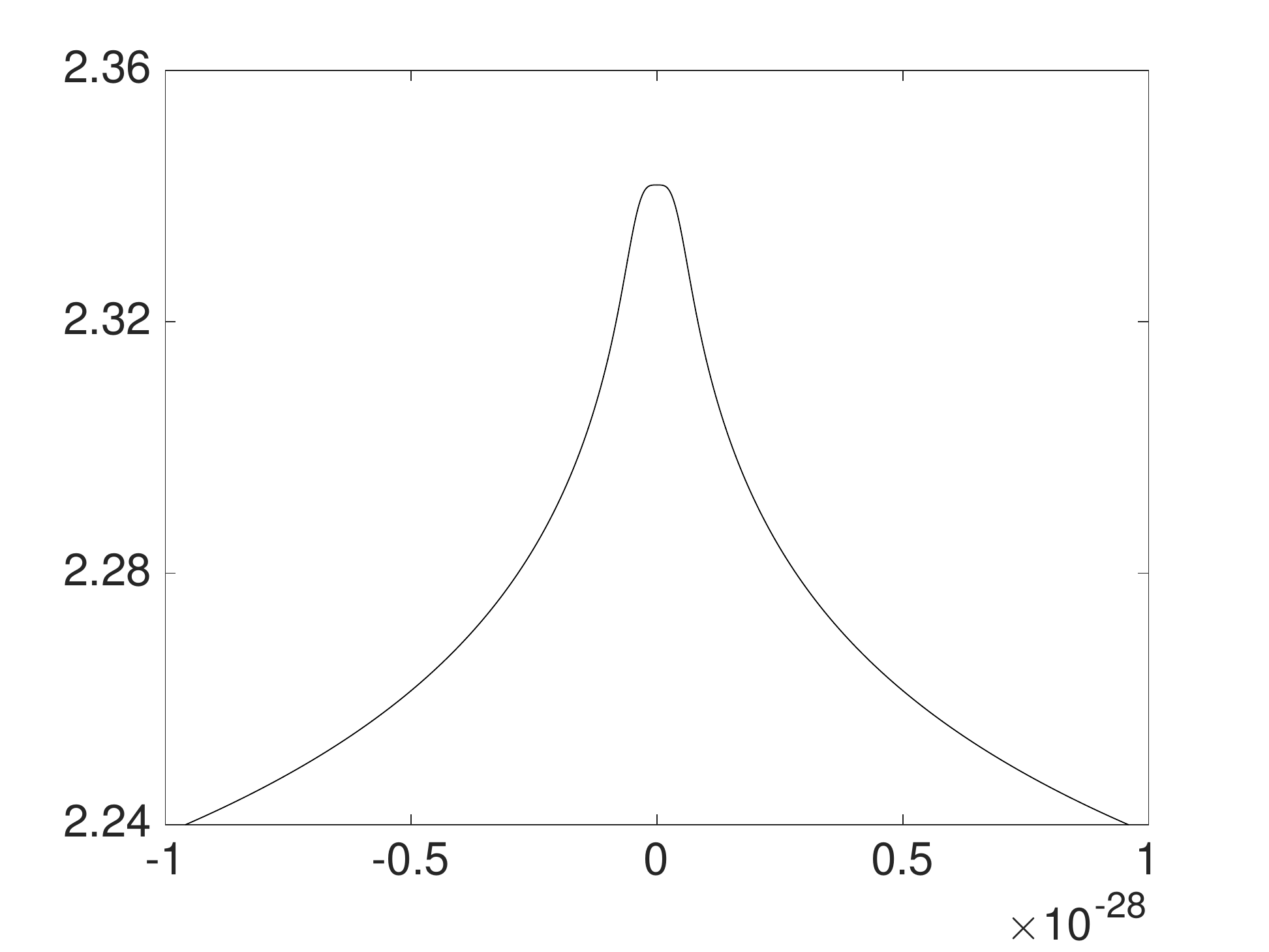}
\par\end{centering}

\centering{}\caption{\label{fig: Laplace2} Logarithm (base 10) of GMRA coefficients of
the PDF $p_{Z}$ in Example~\ref{sub:ExampleLaplace} and (b) logarithmic
singularity of $p_{Z}$ resolved to the interval $\left[-10^{-28},10^{-28}\right]$. }
\end{figure}

\subsubsection{\label{sub:ExampleGumbel}Normal distribution and Gumbel distribution}

We also consider the Gumbel distribution which is a special case of
the generalized extreme value distribution. Its PDF is not symmetric
and is given by 

\begin{equation}
g\left(x;\mu,\sigma\right)=\frac{1}{\sigma}e^{\frac{-\left(x-\mu\right)}{\sigma}-e^{\frac{-\left(x-\mu\right)}{\sigma}}},\label{eq:Gumbel Distribution with parameters}
\end{equation}
where $\mu$ is the location parameter and $\sigma$ is the scale
parameter. Here, we compute the PDF $p_{Z}$ of the product random
variable $Z=XY$, where $X\sim g(x;2,1.5)$ and $Y\sim N(1,3)$ and
display these PDFs in Figure~\ref{fig: Gumbel1}. In Figure~\ref{fig: Gumbel2}
we display the GMRA coefficients of $p_{Z}$ and resolve, within the
interval $\left[-10^{-28},10^{-28}\right]$, the singularity of $p_{Z}$
at zero . 

\begin{figure}[H]
\begin{centering}
\includegraphics[scale=0.3]{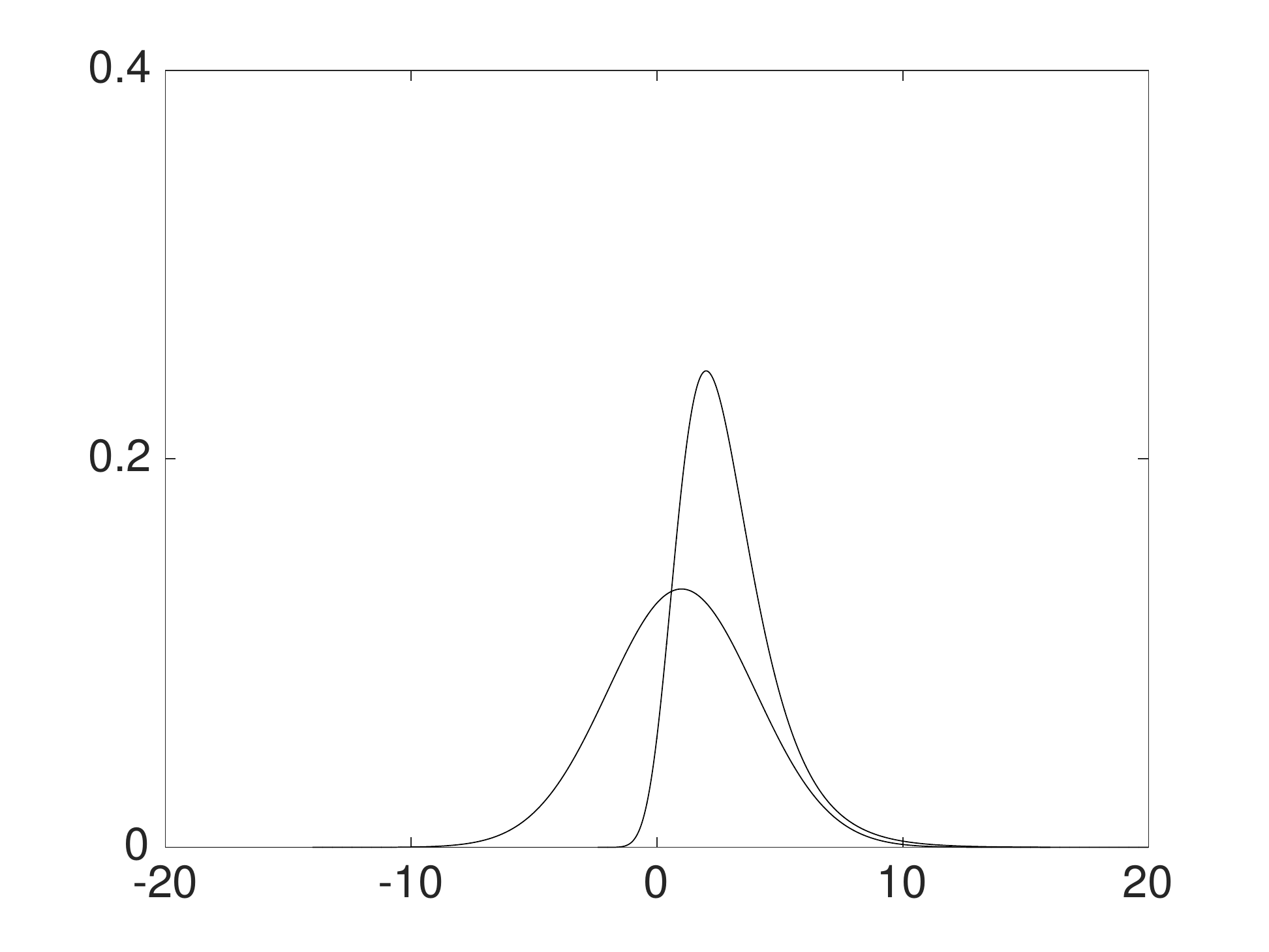}\includegraphics[scale=0.3]{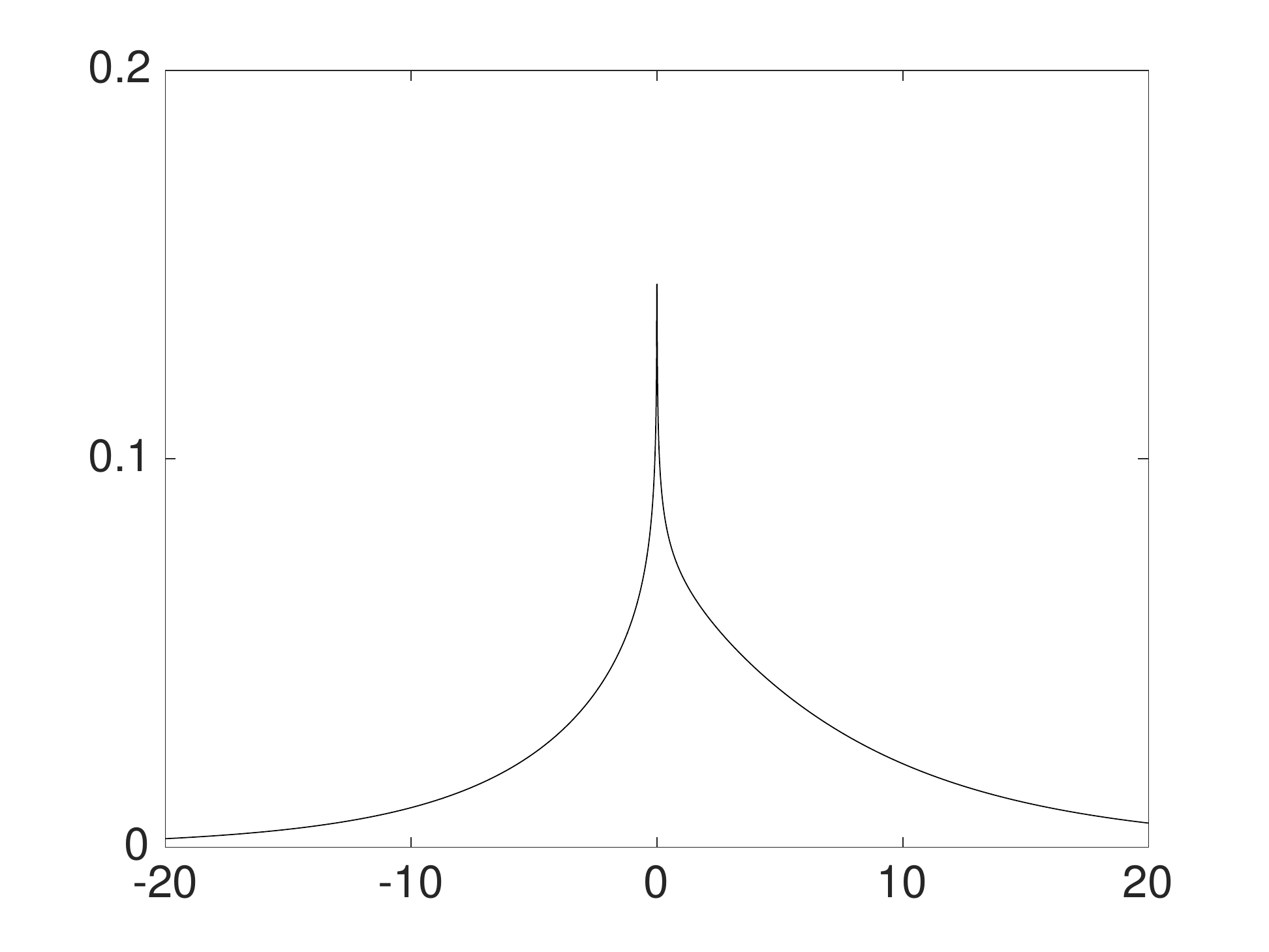}
\par\end{centering}

\centering{}\caption{\label{fig: Gumbel1} PDFs of random variables $X$ and $Y$ in Example~\ref{sub:ExampleGumbel}
(left) and computed product PDF $p_{Z}$ (right).}
\end{figure}

\begin{figure}[H]
\begin{centering}
\includegraphics[scale=0.3]{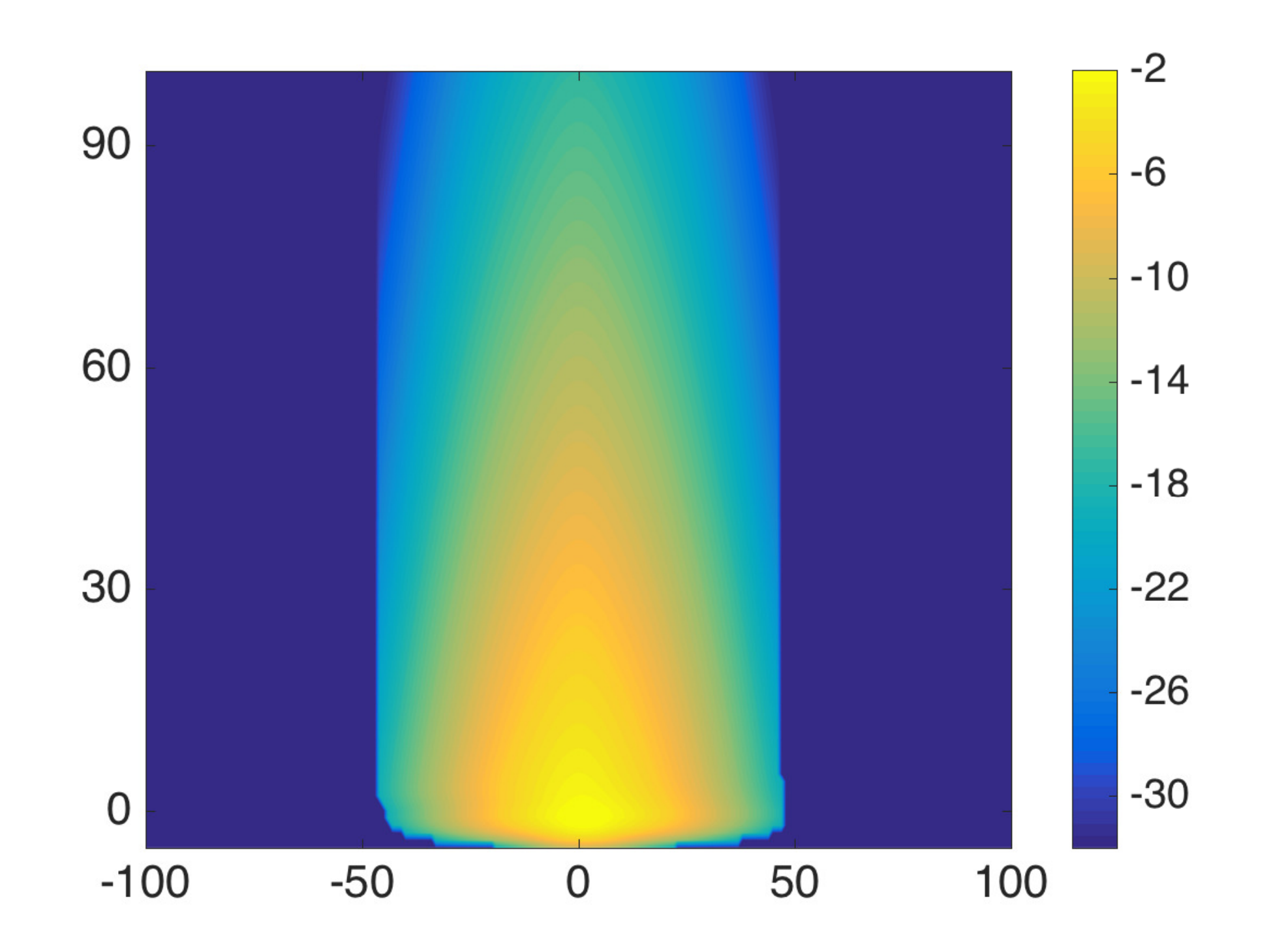}\includegraphics[scale=0.3]{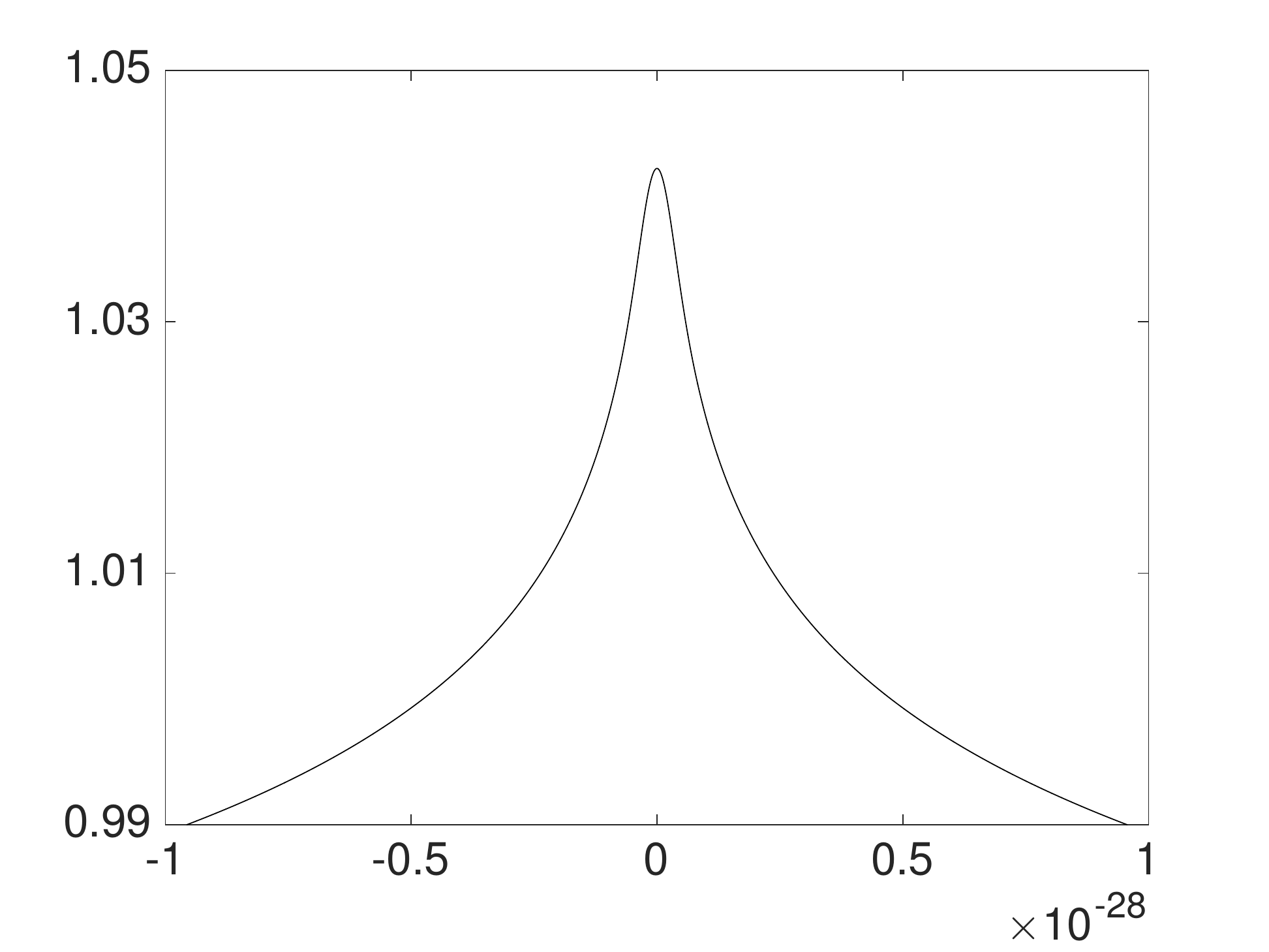}
\par\end{centering}

\centering{}\caption{\label{fig: Gumbel2} Logarithm (base 10) of GMRA coefficients of
$p_{Z}$ in Example~\ref{sub:ExampleGumbel} (left) and logarithmic
singularity of $p_{Z}$ resolved to the interval $\left[-10^{-28},10^{-28}\right]$
(right). }
\end{figure}

\subsubsection{\label{sub:The-Laplace-and-Gumbel}Laplace distribution and Gumbel
distribution}

In this example we compute the PDF $p_{Z}$ of the product random
variable $Z=XY$, where $X\sim f(x;3,1)$ is the Laplace distribution
\eqref{eq:Laplace distribution with parameters} and $Y\sim g(x;2,3)$
is the Gumbel distribution \eqref{eq:Gumbel Distribution with parameters}.
Both distributions are approximated as Gaussian mixtures as described
in Section~\ref{sub:Representing-PDFs-in-GMRA}. In computing $p_{Z}$
we ordered the parameters of Gaussians to reduce the total number
of coefficients as explained in Remark~\ref{Ordering remark}. The
Laplace distribution was approximated by a Gaussian mixture with $120$
terms and the Gumbel distribution with 300 terms, see Section~\ref{sub:Representing-PDFs-in-GMRA}
for details. We display the results in Figure~\ref{fig:The-Laplace-and-the-Gumbel-figure-1}
and Figure~\ref{fig:The-Laplace-and-the-Gumbel-figure-2}. The accuracy
of the result is $\approx10^{-6}$ as indicated by computing moments
of $p_{Z}$.

\begin{figure}[H]
\begin{centering}
\includegraphics[scale=0.3]{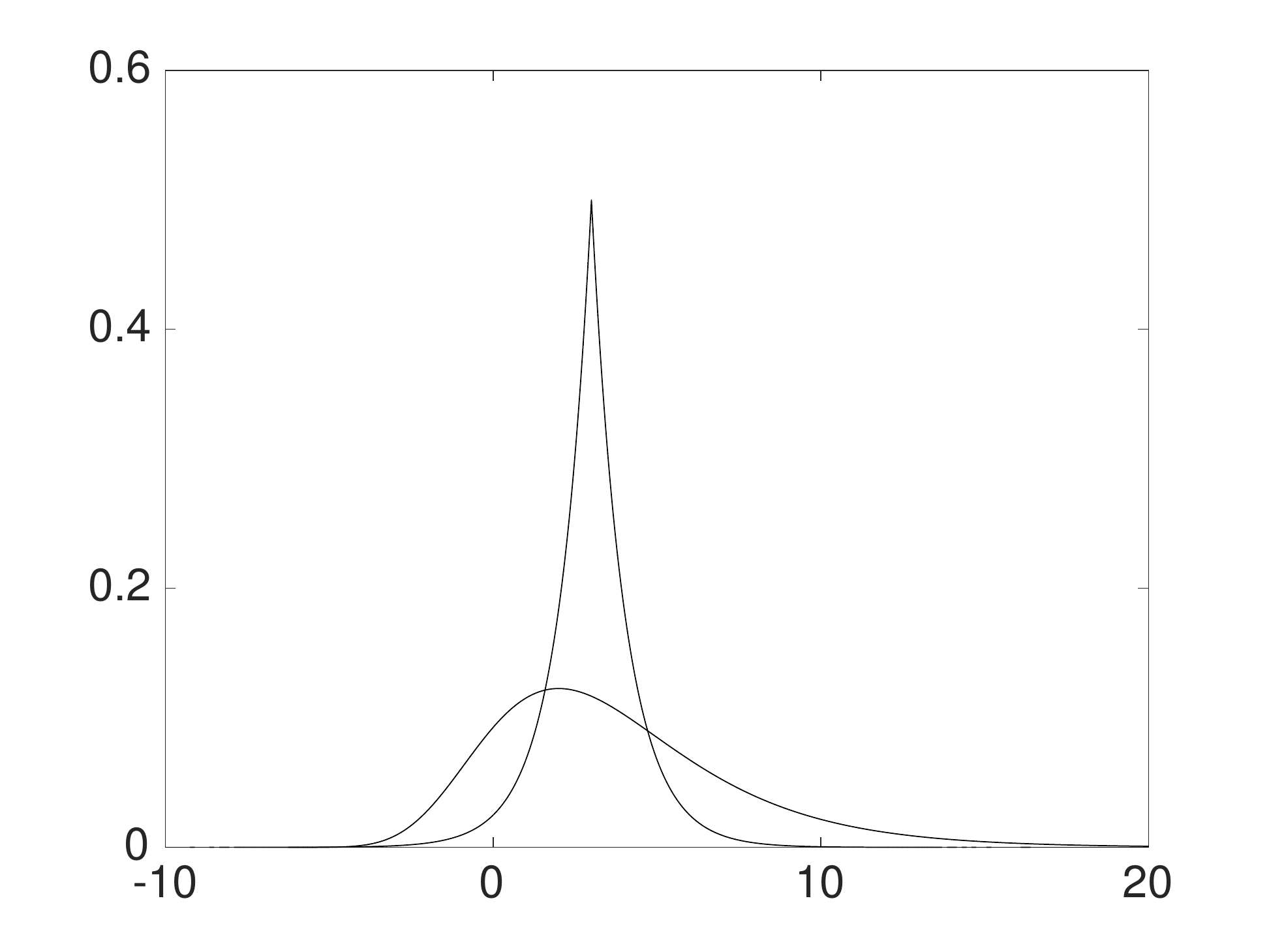}\includegraphics[scale=0.3]{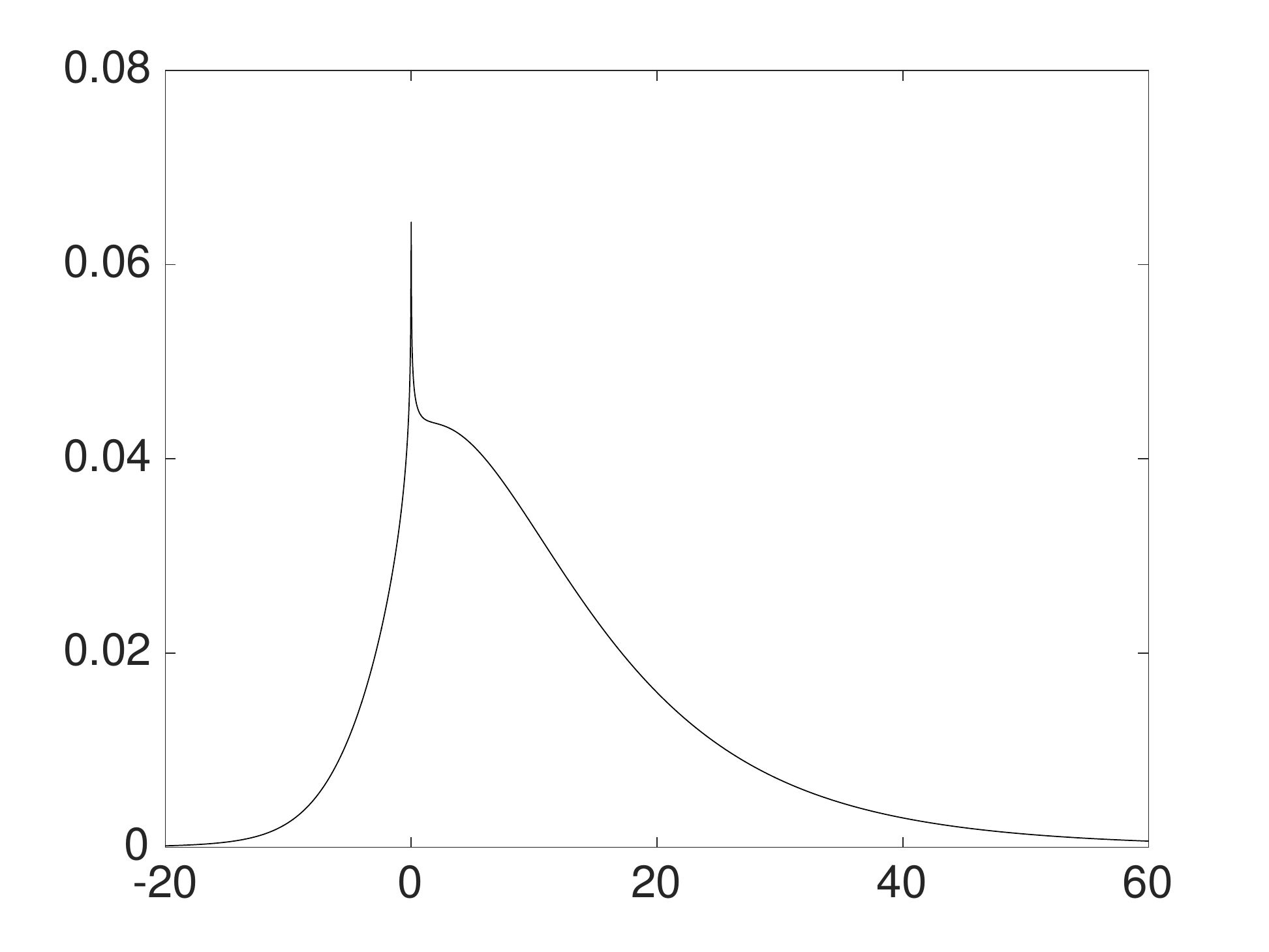}
\par\end{centering}

\centering{}\caption{\label{fig:The-Laplace-and-the-Gumbel-figure-1}The PDFs of the random
variables $X$ and $Y$ in Example~\ref{sub:The-Laplace-and-Gumbel}
(left) and computed product PDF $p_{Z}$ (right).}
\end{figure}

\begin{figure}[h]
\begin{centering}
\includegraphics[scale=0.5]{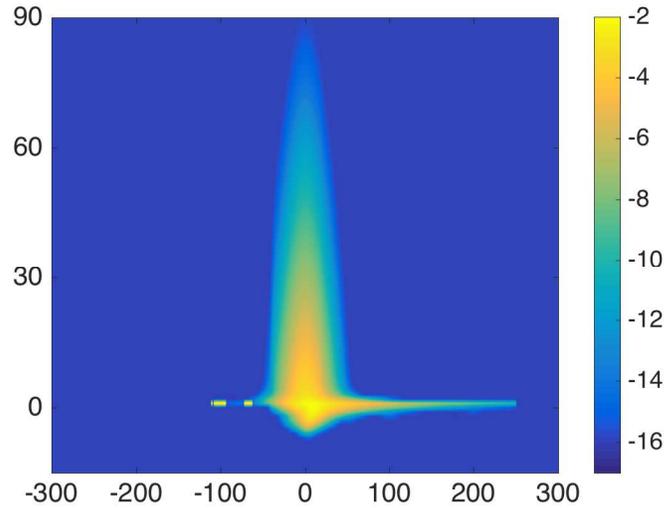}
\par\end{centering}

\centering{}\caption{\label{fig:The-Laplace-and-the-Gumbel-figure-2}Logarithm (base 10)
of GMRA coefficients of the PDF $p_{Z}$ in Example~\ref{sub:The-Laplace-and-Gumbel}.}
\end{figure}

\selectlanguage{american}%

\section{\label{sec:Conclusions-and-further}Conclusions and further work }

GMRA allows us to develop a numerical calculus of PDFs since we can
now compute sums and products of independent random variables (as
well as quotients,\textit{ }since the integral for the quotient of
two independent random variables can be evaluated in a manner similar
to that of the product). Indeed, if the initial PDFs are expressed
in the GMRA, the resulting PDF is also expressed in the GMRA, i.e.
as a multiresolution Gaussian mixture. Similar results hold for CDFs
of random variables. We expect that generalizations of the GMRA to
two and three dimensions can be accomplished in a straightforward
manner, perhaps supplemented by employing the singular value decomposition
(SVD) to simplify the representation of coefficients on each scale.
We plan to address such generalizations within relevant applications.
A practical generalization to dimensions higher than three is more
challenging and we plan to work towards it.

\section{Appendix~A\label{sec:AppendixA}}

The Algorithm~\ref{alg:Adaptive-integrator} describes an adaptive
integrator for computing (with absolute accuracy $\epsilon$) the
integral 
\[
I=\int_{a}^{b}f\left(x\right)dx,
\]
where $f$ can have integrable singularities within $\left[a,b\right]$.
Let $\left\{ x_{m}\right\} _{m=1}^{M}$ and$\left\{ w_{m}\right\} _{m=1}^{M}$
be nodes and weights of the Gauss-Legendre quadrature defined on the
interval $\left[-1,1\right]$. We use this quadrature to evaluate
the sum
\[
s=\sum_{m=1}^{M}\tilde{w}_{m}f(\tilde{x}_{m}),
\]
where nodes and weights are adjusted to an interval $\left[a_{0},b_{0}\right]$
as follows, 
\[
\tilde{x}_{m}=\frac{b_{0}-a_{0}}{2}x_{m}+\frac{b_{0}+a_{0}}{2},\,\,\,\,\tilde{w}_{m}=w_{m}\frac{b_{0}-a_{0}}{2}.
\]
We denote this sum as $s=\mbox{\textbf{quadrature\_sum}}\left(a_{0},b_{0},f\right)$.
The computed value of the integral, $I$, is returned by Algorithm~\ref{alg:Adaptive-integrator}.
\begin{algorithm}[h]
Input: interval $\left[a,b\right]$, function $f$, the target accuracy
$\epsilon$, and a fixed number $M$ of quadrature nodes and weights.

Initialization:
\begin{eqnarray*}
 &  & \mbox{\textbf{left\_p}}\left(1\right)=a\\
 &  & \mbox{\textbf{right\_p}}\left(1\right)=b\\
 &  & s\left(1\right)=\mbox{\textbf{quadrature\_sum}}\left(a,b,f\right)\\
 &  & j=1\\
 &  & I=0
\end{eqnarray*}
Main loop: ($N_{iter}^{\max}$ is large enough to accommodate any
reasonable recursion depth). 
\begin{eqnarray*}
 & \textbf{do}\,\,i=1,N_{iter}^{\max}\\
 &  & c=\frac{1}{2}\left(\mbox{\textbf{left\_p}}\left(j\right)+\mbox{\textbf{right\_p}}\left(j\right)\right)\\
 &  & s_{1}=\mbox{\textbf{quadrature\_sum}}\left(\mbox{\textbf{left\_p}}\left(j\right),c,f\right)\\
 &  & s_{2}=\mbox{\textbf{quadrature\_sum}}\left(\mbox{c,\textbf{right\_p}}\left(j\right),f\right)\\
 &  & \mbox{\textbf{if}}\left|s_{1}+s_{2}-s\left(j\right)\right|>\epsilon\,\,\,\mbox{\textbf{then}}\\
 &  & \,\,\,\,\,\mbox{\textbf{left\_p}}\left(j+1\right)=\mbox{\textbf{left\_p}}\left(j\right)\\
 &  & \,\,\,\,\,\mbox{\textbf{right\_p}}\left(j+1\right)=\frac{1}{2}\left(\mbox{\textbf{left\_p}}\left(j\right)+\mbox{\textbf{right\_p}}\left(j\right)\right)\\
 &  & \,\,\,\,\,s\left(j+1\right)=s_{1}\\
 &  & \,\,\,\,\,\mbox{\textbf{left\_p}}\left(j\right)=\frac{1}{2}\left(\mbox{\textbf{left\_p}}\left(j\right)+\mbox{\textbf{right\_p}}\left(j\right)\right)\\
 &  & \,\,\,\,\,s\left(j\right)=s_{2}\\
 &  & \,\,\,\,\,j=j+1\\
 &  & \mbox{\textbf{else}}\\
 &  & \,\,\,\,\,I=I+s_{1}+s_{2}\\
 &  & \,\,\,\,\,j=j-1\\
 & \mbox{} & \,\,\,\,\,\textbf{If}\,\,\,j=0\,\,\,\mbox{\textbf{return} (the integral has been evaluated)}\\
 &  & \mbox{\textbf{endif}}\\
 & \mbox{\textbf{enddo}}
\end{eqnarray*}

\caption{\label{alg:Adaptive-integrator}Adaptive integrator }
\end{algorithm}

\section{\label{sec:Appendix B}Appendix B}

Computing the PDF of the product of independent random variables does
not require all of the features of MRA since we only compute projections
on the scaling functions at appropriate scales. In this appendix we
briefly present some of the usual MRA constructs. We assume that the
reader is familiar with the basic concepts of orthogonal and bi-orthogonal
wavelet bases, see e.g.~\cite{DAUBEC:1992}. We note that the approximate
GMRA is in many respects similar to MRA based on high order splines,
see e.g.~\cite{CHUI:1992}. However, there are several special properties
of GMRA that are not shared by spline based MRA. For example, as shown
in Theorem~\ref{thm:The-Gaussian-on arbitrary scale}, a Gaussian
with an arbitrary exponent and an arbitrary shift can be represented
at an appropriate subspace which does not depend on the shift. Another
advantage of using GMRA is that it becomes possible to apply operators
to functions in a semi-analytic fashion which is important in a number
of applications we plan to address separately. Finally, we would like
to emphasize that the fact that the GMRA is approximate is not an
obstacle in practice since we can adjust the accuracy of the GMRA
as needed.

\subsection{Quadrature Mirror filters for GMRA}

We start by considering the ratio $\widehat{\phi}\left(2p\right)/\widehat{\phi}\left(p\right)$,
where

\begin{equation}
\widehat{\phi}\left(p\right)=\int_{\mathbb{R}}\phi(x)e^{-2\pi ixp}dx=e^{-\frac{\pi^{2}}{\alpha}p^{2}},\label{eq:Fourier_scaling_functions}
\end{equation}
is the Fourier transform of the approximate scaling function in \eqref{eq:Gaussian multiresolution basis}.
We have 
\begin{equation}
\frac{\widehat{\phi}\left(2p\right)}{\widehat{\phi}\left(p\right)}=e^{-\frac{3\pi^{2}}{\alpha}p^{2}}\label{eq:ratio gaussian}
\end{equation}
and note that for an exact MRA this ratio should be a periodic function.
This leads us to define the filter $m_{0}$ for GMRA as a periodization
of \eqref{eq:ratio gaussian},
\begin{equation}
m_{0}\left(p\right)=\sum_{n\in\mathbb{Z}}e^{-\frac{3\pi^{2}}{\alpha}(p-n)^{2}}=\sqrt{\frac{\alpha}{3\pi}}\vartheta_{3}\left(\pi p,e^{-\frac{\alpha}{3}}\right).\label{def_m0}
\end{equation}
Using the Poisson summation formula applied to $\vartheta_{3}\left(\pi p,e^{-\frac{\alpha}{3}}\right)$,
we have 
\begin{eqnarray*}
\left|e^{-\frac{3\pi^{2}}{\alpha}p^{2}}-\sqrt{\frac{\alpha}{3\pi}}\vartheta_{3}\left(\pi p,e^{-\frac{\alpha}{3}}\right)\right| & = & \left|e^{-\frac{3\pi^{2}}{\alpha}p^{2}}-\sum_{n\in\mathbb{Z}}e^{-\frac{3\pi^{2}}{\alpha}(p-n)^{2}}\right|\\
 & = & \sum_{n\ne0}e^{-\frac{3\pi^{2}}{\alpha}(p-n)^{2}}.
\end{eqnarray*}
For $\left|p\right|\leq1/2$, we obtain the error estimate
\begin{equation}
\left|e^{-\frac{3\pi^{2}}{\alpha}p^{2}}-\sqrt{\frac{\alpha}{3\pi}}\vartheta_{3}\left(\pi p,e^{-\frac{\alpha}{3}}\right)\right|\le2\sum_{n\ge1}e^{-\frac{3\pi^{2}}{\alpha}\left(n-\frac{1}{2}\right)^{2}}=\vartheta_{2}\left(0,e^{-\frac{3\pi^{2}}{\alpha}}\right),\label{error_estimate_m0}
\end{equation}
where the last equality can be found in \cite[eq. 8.180.3]{GRA-RYZ:2007}.
Since 
\begin{eqnarray*}
\sum_{n\ge1}e^{-\frac{3\pi^{2}}{\alpha}\left(n-\frac{1}{2}\right)^{2}}=\sum_{n\geq0}e^{-\frac{3\pi^{2}}{\alpha}\left(n+\frac{1}{2}\right)^{2}} & \leq & e^{-\frac{3\pi^{2}}{4\alpha}}\sum_{n\geq0}e^{-\frac{3\pi^{2}}{\alpha}n^{2}}\\
 & = & \frac{1}{2}e^{-\frac{3\pi^{2}}{4\alpha}}\left(1+\vartheta_{3}\left(0,e^{-\frac{3\pi^{2}}{\alpha}}\right)\right),
\end{eqnarray*}
Lemma~\ref{lem:Estimate_Theta_at_zero} shows that the error \eqref{error_estimate_m0}
decays exponentially fast as a function of the parameter $\alpha$.
For example, if $\alpha=1/5$, the estimate \eqref{error_estimate_m0}
provides an error bound which is approximately $1.69\times10^{-16}$.
and observe that (for any $\alpha\leq3\pi$) \eqref{ineq_v0-1} implies
\begin{align}
\left|m_{0}\left(0\right)-1\right| & =m_{0}\left(0\right)-1=\sqrt{\frac{\alpha}{3\pi}}\vartheta_{3}\left(0,e^{-\frac{\alpha}{3}}\right)-1\nonumber \\
 & =\vartheta_{3}\left(0,e^{-\frac{3\pi^{2}}{\alpha}}\right)-1\leq c_{\pi}\thinspace e^{-\frac{3\pi^{2}}{\alpha}},\label{estimate_m0_at_0}
\end{align}
where $c_{\pi}\approx2.002$, see \eqref{constant_c_pi}. Therefore,
for the values of $\alpha$ of interest, $m_{0}\left(0\right)$ is
extremely close to one. For example, for $\alpha=1/5$, \eqref{estimate_m0_at_0}
yields
\[
\left|m_{0}\left(0\right)-1\right|\leq1.015\times10^{-64}.
\]

\subsubsection{Quadrature Mirror filters for orthogonal approximate GMRA}

Next, we compute

\[
\sum_{n\in\mathbb{Z}}\left|\widehat{\phi}\left(p+n\right)\right|^{2}=\sum_{n\in\mathbb{Z}}e^{-\frac{2\pi^{2}}{\alpha}\left(p+n\right)^{2}}=\sqrt{\frac{\alpha}{2\pi}}\sum_{n\in\mathbb{Z}}e^{-\frac{\alpha}{2}n^{2}}e^{2\pi inp}=\sqrt{\frac{\alpha}{2\pi}}\vartheta_{3}\left(\pi p,e^{-\frac{\alpha}{2}}\right),
\]
to define the Fourier transform of the approximate, orthogonal Gaussian-based
scaling function as 
\begin{equation}
\widehat{\varphi}\left(p\right)=\frac{\widehat{\phi}\left(p\right)}{\sqrt{\sqrt{\frac{\alpha}{2\pi}}\vartheta_{3}\left(\pi p,e^{-\frac{\alpha}{2}}\right)}}.\label{eq:varphi GMRA}
\end{equation}
The corresponding approximate quadrature mirror filter (QMF), determined
by the identity $\widehat{\varphi}\left(2p\right)=M_{a}\left(p\right)\widehat{\varphi}\left(p\right)$,
is then given by 
\begin{equation}
M_{a}\left(p\right)=m_{0}\left(p\right)\left(\frac{\vartheta_{3}\left(\pi p,e^{-\frac{\alpha}{2}}\right)}{\vartheta_{3}\left(2\pi p,e^{-\frac{\alpha}{2}}\right)}\right)^{1/2}.\label{def_M_a}
\end{equation}
We note that $M_{a}\left(0\right)=m_{0}\left(0\right)$ so that, for
$\alpha=1/5$, $\left|M_{a}\left(0\right)-1\right|\leq1.015\times10^{-64}$
as well. The real valued filter $M_{a}$ satisfies an approximate
QMF equation, 
\begin{equation}
M_{a}^{2}\left(p\right)+M_{a}^{2}\left(p+\frac{1}{2}\right)\approx1.\label{eq:QMF_Ma}
\end{equation}
For $\alpha=1/5$, we numerically verify that
\[
\left|M_{a}^{2}\left(p\right)+M_{a}^{2}\left(p+\frac{1}{2}\right)-1\right|<0.11\times10^{-20}.
\]

\subsubsection{Quadrature Mirror filters for biorthogonal approximate GMRA}

In order to obtain the filter that is dual to $m_{0}$, we rewrite
\eqref{eq:QMF_Ma} as 
\[
M_{a}^{2}\left(p\right)+M_{a}^{2}\left(p+\frac{1}{2}\right)=m_{0}\left(p\right)M_{00}\left(p\right)+m_{0}\left(p+\frac{1}{2}\right)M_{00}\left(p+\frac{1}{2}\right)\approx1,
\]
so that the dual filter $M_{00}$ is given by 
\begin{equation}
M_{00}\left(p\right)=\frac{m_{0}\left(p\right)\vartheta_{3}\left(\pi p,e^{-\frac{\alpha}{2}}\right)}{\vartheta_{3}\left(2\pi p,e^{-\frac{\alpha}{2}}\right)}=\sqrt{\frac{\alpha}{3\pi}}\frac{\vartheta_{3}\left(\pi p,e^{-\frac{\alpha}{3}}\right)\vartheta_{3}\left(\pi p,e^{-\frac{\alpha}{2}}\right)}{\vartheta_{3}\left(2\pi p,e^{-\frac{\alpha}{2}}\right)}.\label{eq:filter_M_00}
\end{equation}
We use $M_{00}$ to obtain a formula for the projection of coefficients
onto a coarser scale. It is sufficient to consider a projection from
zero scale into the next coarser scale. Therefore, starting with $f$
of the form 
\[
f\left(x\right)=\sum_{k\in\mathbb{Z}}f_{k}\phi\left(x-k\right),
\]
so that
\[
\widehat{f}\left(p\right)=\left(\sum_{k}f_{k}e^{-2\pi ikp}\right)\widehat{\phi}\left(p\right),
\]
the coefficients of $f$ in the next coarser scale are computed via
(see e.g. \cite[Section 8.3.1]{DAUBEC:1992}), 
\begin{eqnarray}
g_{n} & = & \frac{1}{\sqrt{2}}\int_{-1/2}^{1/2}\left[\left(\sum_{k}f_{k}e^{-2\pi ikp/2}\right)M_{00}\left(\frac{p}{2}\right)+\right.\label{eq:coeff_g_n}\\
 &  & \left.\left(\sum_{k}f_{k}e^{-2\pi ik(\frac{p}{2}+\frac{1}{2})}\right)M_{00}\left(\frac{p}{2}+\frac{1}{2}\right)\right]e^{2\pi inp}dp.\nonumber 
\end{eqnarray}
We would like to point out a numerical problem associated with using
the filter $M_{00}$. The plot of $M_{00}$ for $\alpha=0.2$ in Figure~\ref{fig:Filter--for M00}
shows that one can expect a loss of significant digits due to its
large dynamic range. 
\begin{figure}
\includegraphics[scale=0.3]{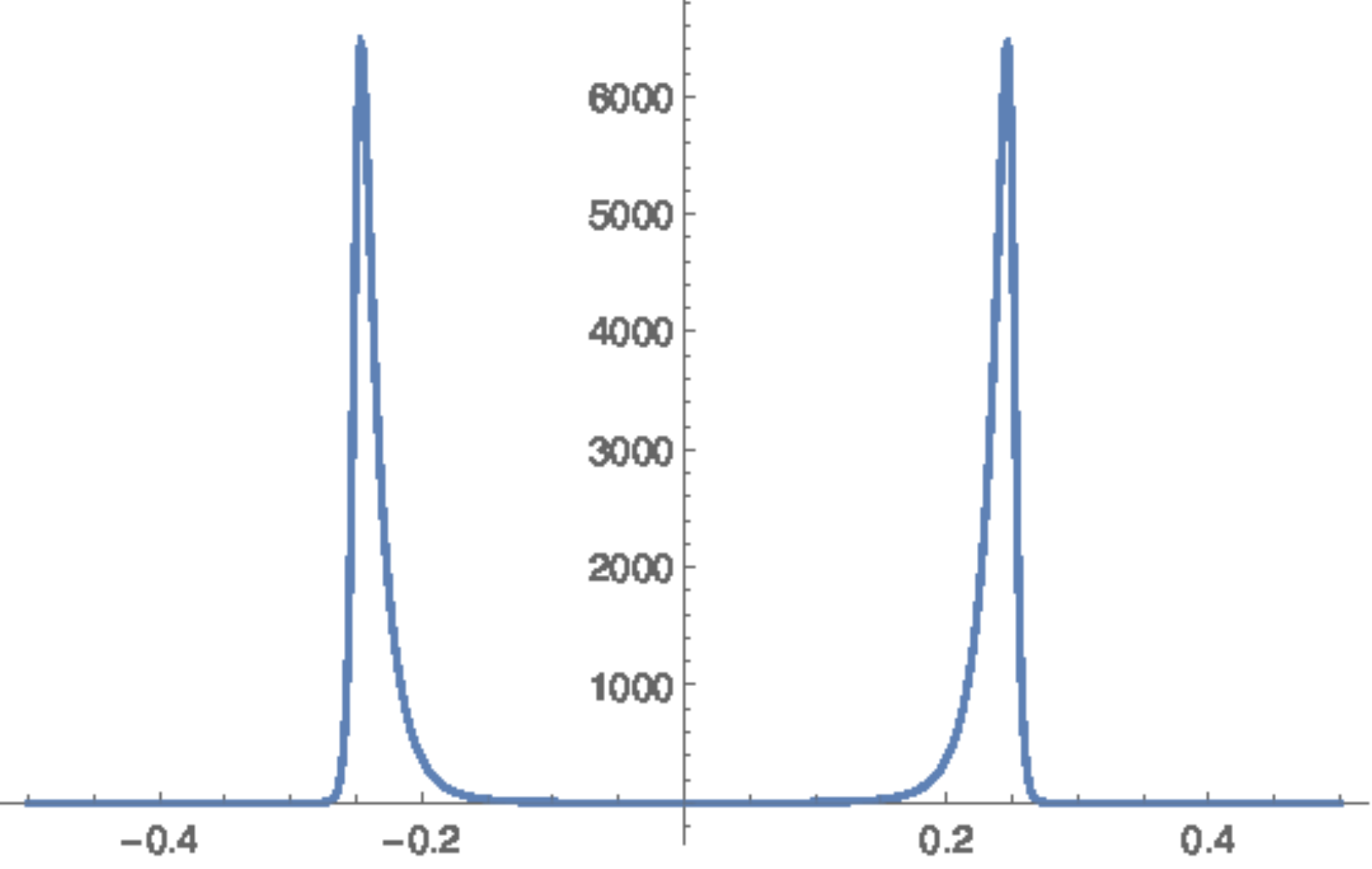}\includegraphics[scale=0.3]{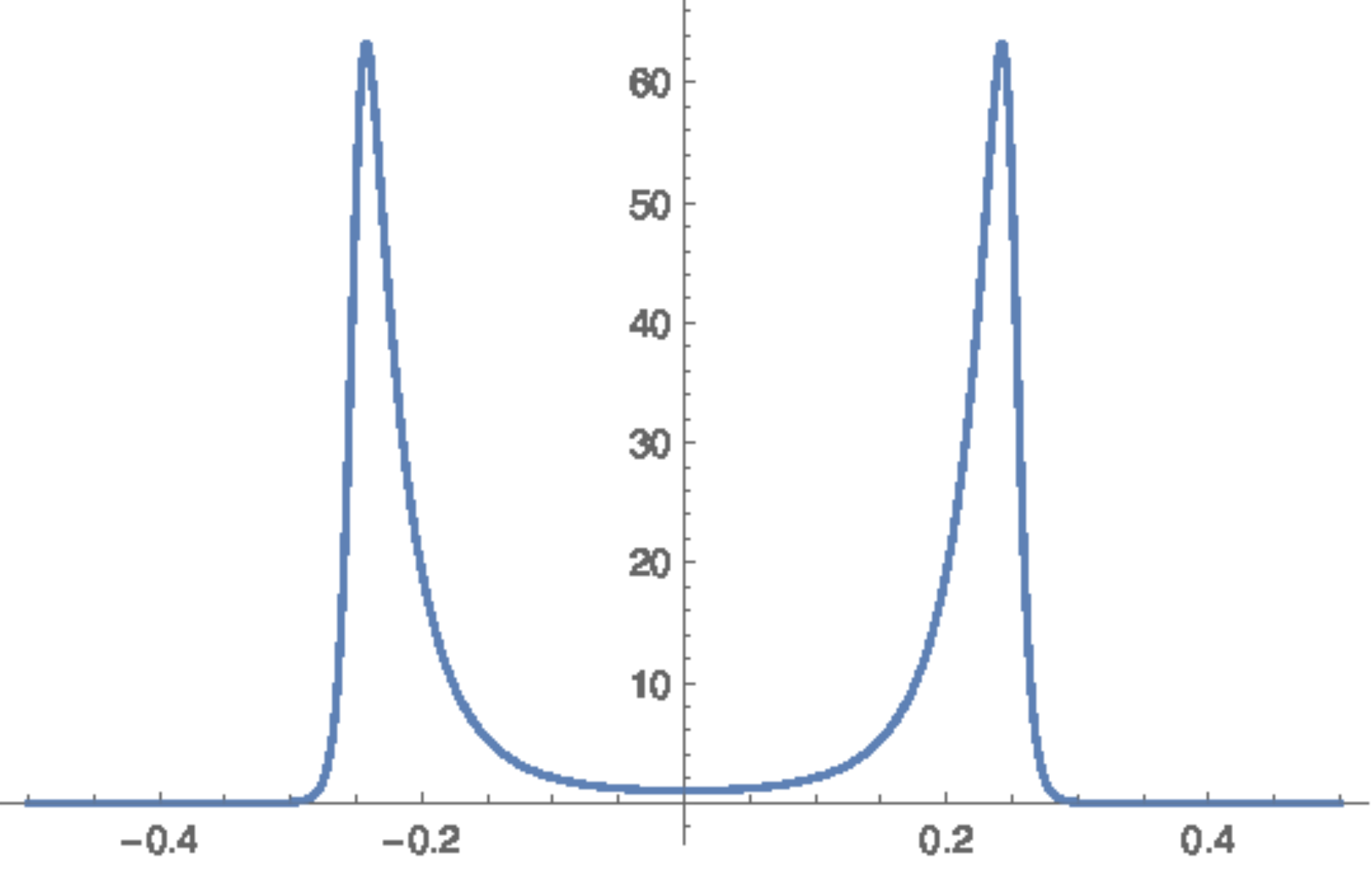}\caption{\label{fig:Filter--for M00}Filter $M_{00}\left(p\right)$ in \eqref{eq:filter_M_00}
for $\alpha=0.2$ (left) and $\alpha=0.4$ (right). }
\end{figure}
 To avoid the possible loss of significant digits we have several
options. First, as is done in this paper, we can avoid using $M_{00}$
altogether. The second option is to accept a lower (e.g. single precision)
accuracy in computations with GMRA. Finally, one can look for an alternative
(nonlinear) algorithm for computing coefficients $g_{n}$ in \eqref{eq:coeff_g_n}.
To illustrate the second option, we choose $\alpha=0.4$ and observe
the relatively small dynamic range of $M_{00}$ for this value of
$\alpha$ in Figure~\ref{fig:Filter--for M00}. Using $\alpha=0.4$
in computing the PDFs results in roughly single precision accuracy
as illustrated in Figures~\ref{fig:example0101-2} and\ref{fig:example010101-2}.

\subsection{Constructing an exact MRA}

An alternative approach motivated by the construction of GMRA is to
develop an exact MRA such that the resulting scaling function is closely
approximated by a single Gaussian. Although we do not use this approach
in this paper, we briefly indicate how to attain such construction.
We can demonstrate that the nonnegative filter

\begin{equation}
M_{0}\left(p\right)=\left(\frac{\vartheta_{3}\left(\pi p,e^{-\frac{\alpha}{8}}\right)}{2\vartheta_{3}\left(2\pi p,e^{-\frac{\alpha}{2}}\right)}+\left(1-\frac{\vartheta_{3}\left(0,e^{-\frac{\alpha}{8}}\right)}{2\vartheta_{3}\left(0,e^{-\frac{\alpha}{2}}\right)}\right)\cos2\pi p\right)^{1/2},\label{eq:M_0 definition}
\end{equation}
is an excellent approximation of the filter $M_{a}$ in \eqref{def_M_a}.
In fact, for $\alpha=0.2$, we verify numerically that
\begin{equation}
\left|M_{a}\left(p\right)-M_{0}\left(p\right)\right|<1.05\times10^{-21},\,\,\,p\in\left[-\frac{1}{2},\frac{1}{2}\right].\label{diff_M_a_and_M_0}
\end{equation}
In the next theorem we show that the filter $M_{0}$ has the necessary
properties to generate an exact orthogonal MRA.
\begin{thm}
The filter $M_{0}$ in \eqref{eq:M_0 definition} satisfies the QMF
condition,
\[
M_{0}^{2}\left(p\right)+M_{0}^{2}\left(p+\frac{1}{2}\right)=1,
\]
and it is an even, positive function in $\left(-\frac{1}{2},\frac{1}{2}\right)$,
monotonically decreasing in $\left[0,\frac{1}{2}\right]$ with $M_{0}\left(0\right)=1$. \end{thm}
\begin{proof}
Since $\vartheta_{3}$ satisfies the well-known property 
\begin{align}
\vartheta_{3}\left(\pi p,e^{-\gamma}\right)+\vartheta_{3}\left(\pi p+\frac{\pi}{2},e^{-\gamma}\right) & =\sum_{n\in\mathbb{Z}}e^{-\gamma n^{2}}e^{2\pi ipn}+\sum_{n\in\mathbb{Z}}\left(-1\right)^{n}e^{-\gamma n^{2}}e^{2\pi ipn},\label{half_period_identity}\\
 & =2\sum_{n\in\mathbb{Z}}e^{-\gamma\left(2n\right)^{2}}e^{2\pi ip\left(2n\right)}\nonumber \\
 & =2\vartheta_{3}\left(2\pi p,e^{-4\gamma}\right),\nonumber 
\end{align}
and $\cos2\pi\left(p+\frac{1}{2}\right)=-\cos2\pi p$, we obtain 
\begin{align}
M_{0}^{2}\left(p\right)+M_{0}^{2}\left(p+\frac{1}{2}\right) & =\frac{\vartheta_{3}\left(\pi p,e^{-\frac{\alpha}{8}}\right)+\vartheta_{3}\left(\pi p+\frac{\pi}{2},e^{-\frac{\alpha}{8}}\right)}{2\vartheta_{3}\left(2\pi p,e^{-\frac{\alpha}{2}}\right)}\label{QMF_equation_for_M0}\\
 & =1.\nonumber 
\end{align}
By construction, we have $M_{0}\left(0\right)=1$ and, since $\vartheta_{3}\left(\pi p,e^{-\gamma}\right)$
is an even function, $M_{0}$ is even. Using \eqref{theta_at_0_using_reciprocal}
and observing that $\vartheta_{3}\left(0,e^{-\gamma}\right)$ is strictly
decreasing for $\gamma>0$, we have 
\[
\frac{\vartheta_{3}\left(0,e^{-\frac{\alpha}{8}}\right)}{2\vartheta_{3}\left(0,e^{-\frac{\alpha}{2}}\right)}=\frac{\vartheta_{3}\left(0,e^{-\frac{8\pi^{2}}{\alpha}}\right)}{\vartheta_{3}\left(0,e^{-\frac{2\pi^{2}}{\alpha}}\right)}<1.
\]
Hence, denoting by 
\[
\eta=1-\frac{\vartheta_{3}\left(0,e^{-\frac{\alpha}{8}}\right)}{2\vartheta_{3}\left(0,e^{-\frac{\alpha}{2}}\right)},
\]
the constant in \eqref{eq:M_0 definition}, we know that $\eta\in\left(0,1\right)$.
For $p\in\left[0,\frac{1}{2}\right)$, we denote by $f$ the function
\[
f\left(p\right)=M_{0}^{2}\left(p\right)=\frac{\vartheta_{3}\left(\pi p,e^{-\frac{\alpha}{8}}\right)}{2\vartheta_{3}\left(2\pi p,e^{-\frac{\alpha}{2}}\right)}+\eta\cos\left(2\pi p\right)=g\left(p\right)+\eta\cos\left(2\pi p\right),
\]
and notice that, due to \eqref{QMF_equation_for_M0}, we have $f\left(0\right)=1$
and $f\left(\frac{1}{2}\right)=0$. Since $\eta>0$ and $\cos\left(2\pi p\right)$
is strictly decreasing on $\left(0,1/2\right)$, to show that $f$
is strictly decreasing in this interval, it is enough to show that
$g$ has the same property. By \eqref{half_period_identity}, we have
\[
g\left(p\right)=\frac{\vartheta_{3}\left(\pi p,e^{-\frac{\alpha}{8}}\right)}{\vartheta_{3}\left(\pi p,e^{-\frac{\alpha}{8}}\right)+\vartheta_{3}\left(\pi\left(p+\frac{1}{2}\right),e^{-\frac{\alpha}{8}}\right)}
\]
and, taking derivatives, $g'\left(p\right)<0$ is equivalent to the
inequality
\begin{equation}
\frac{d}{dp}\ln\vartheta_{3}\left(\pi p,e^{-\frac{\alpha}{8}}\right)<\frac{d}{dp}\ln\vartheta_{3}\left(\pi\left(p+\frac{1}{2}\right),e^{-\frac{\alpha}{8}}\right).\label{ineq_log_deriv_theta}
\end{equation}
From \cite[20.5.12]{OLVER:2010}, we have, for any $\gamma>0$,
\[
\frac{d}{dp}\ln\vartheta_{3}\left(\pi p,e^{-\gamma}\right)=-\frac{4}{\pi}\sin2\pi p\prod_{n\geq1}\frac{1}{\left|1+e^{-\gamma\left(2n-1\right)}e^{2\pi pi}\right|}
\]
 and, hence,
\[
\frac{d}{dp}\ln\vartheta_{3}\left(\pi\left(p+\frac{1}{2}\right),e^{-\frac{\alpha}{8}}\right)=\frac{4}{\pi}\sin2\pi p\prod_{n\geq1}\frac{1}{\left|1-e^{-\frac{\alpha}{8}\left(2n-1\right)}e^{2\pi pi}\right|},
\]
which implies \eqref{ineq_log_deriv_theta} and, thus, that $f$ is
strictly decreasing on $\left(0,1/2\right)$. Therefore, $f$ is also
positive in that interval. It follows that $\sqrt{f}=M_{0}$ is strictly
decreasing on $\left(0,1/2\right)$.
\end{proof}
Since $M_{0}$ has no zeros in $\left[-1/3,1/3\right]$, using Corollary
6.3.2 in \cite[p. 186]{DAUBEC:1992}, we know that the QMF $M_{0}$
generates an orthogonal MRA. The Fourier transform of the resulting
scaling function $\Phi$ can now be obtained as 
\[
\widehat{\Phi}\left(p\right)=\prod_{n\geq1}M_{0}\left(\frac{p}{2^{n}}\right).
\]
Due to \eqref{diff_M_a_and_M_0}, if we use a finite number of scales,
the exact MRA and approximate GMRA generated by the function $\widehat{\varphi}\left(p\right)$
in \eqref{eq:varphi GMRA} are exchangeable for an appropriately selected
parameter $\alpha$ tuned to the user-selected accuracy.

\subsubsection{An alternative approach to constructing exact MRA}

We have strong numerical evidence but no proofs that the following
construction also leads to an exact MRA. We modify $m_{0}$ in \eqref{def_m0}
to define
\[
M_{exact}\left(p\right)=\frac{m_{0}\left(p\right)-m_{0}\left(1/2\right)}{m_{0}\left(0\right)-m_{0}\left(1/2\right)}=\frac{\vartheta_{3}\left(\pi p,e^{-\frac{\alpha}{3}}\right)-\vartheta_{3}\left(\frac{\pi}{2},e^{-\frac{\alpha}{3}}\right)}{\vartheta_{3}\left(0,e^{-\frac{\alpha}{3}}\right)-\vartheta_{3}\left(\frac{\pi}{2},e^{-\frac{\alpha}{3}}\right)},
\]
so that $M_{exact}\left(0\right)=1$ and $M_{exact}\left(1/2\right)=0$.
Then the scaling function for the exact (non-orthogonal) MRA would
be given by the infinite product,
\begin{equation}
\widehat{\phi}_{exact}\left(p\right)=\prod_{j=1}^{\infty}M_{exact}\left(\frac{p}{2^{j}}\right),\label{eq:exact scaling function}
\end{equation}
in which case it holds that $\widehat{\phi}_{exact}\left(0\right)=1$
and $\widehat{\phi}_{exact}\left(n\right)=0$ , $n\in\mathbb{Z}/\left\{ 0\right\} $.
We verify numerically that 
\[
\left|\widehat{\phi}_{exact}\left(p\right)-\widehat{\phi}\left(p\right)\right|\le\epsilon,
\]
where $\epsilon$ is of the same order as in \eqref{eq:parameter epsilon}
and $\widehat{\phi}\left(p\right)$ is given in \eqref{eq:Fourier_scaling_functions}.
In Figure~\ref{fig:The difference exact scaling and Gaussian} we
plot the error $\log_{10}\left|\widehat{\phi}_{exact}\left(p\right)-\widehat{\phi}\left(p\right)\right|$
for two choices of the parameter $\alpha=0.25$ and $\alpha=0.4$.
Effectively we force contributions outside the main band to be below
the desired accuracy $\epsilon$, which allows us to use a single
Gaussian as an approximate scaling function. 

\begin{figure}
\begin{centering}
\includegraphics[scale=0.3]{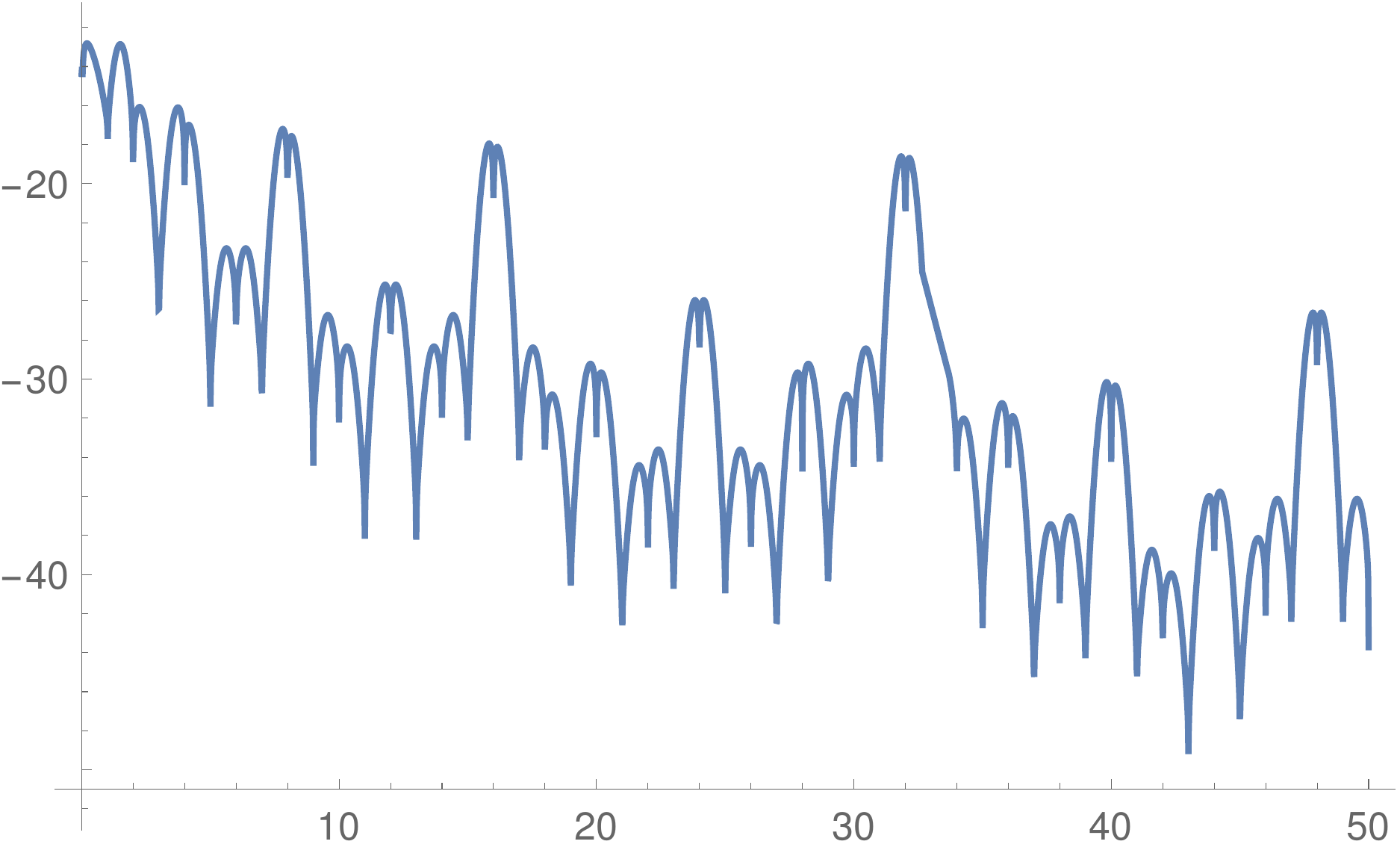}\includegraphics[scale=0.3]{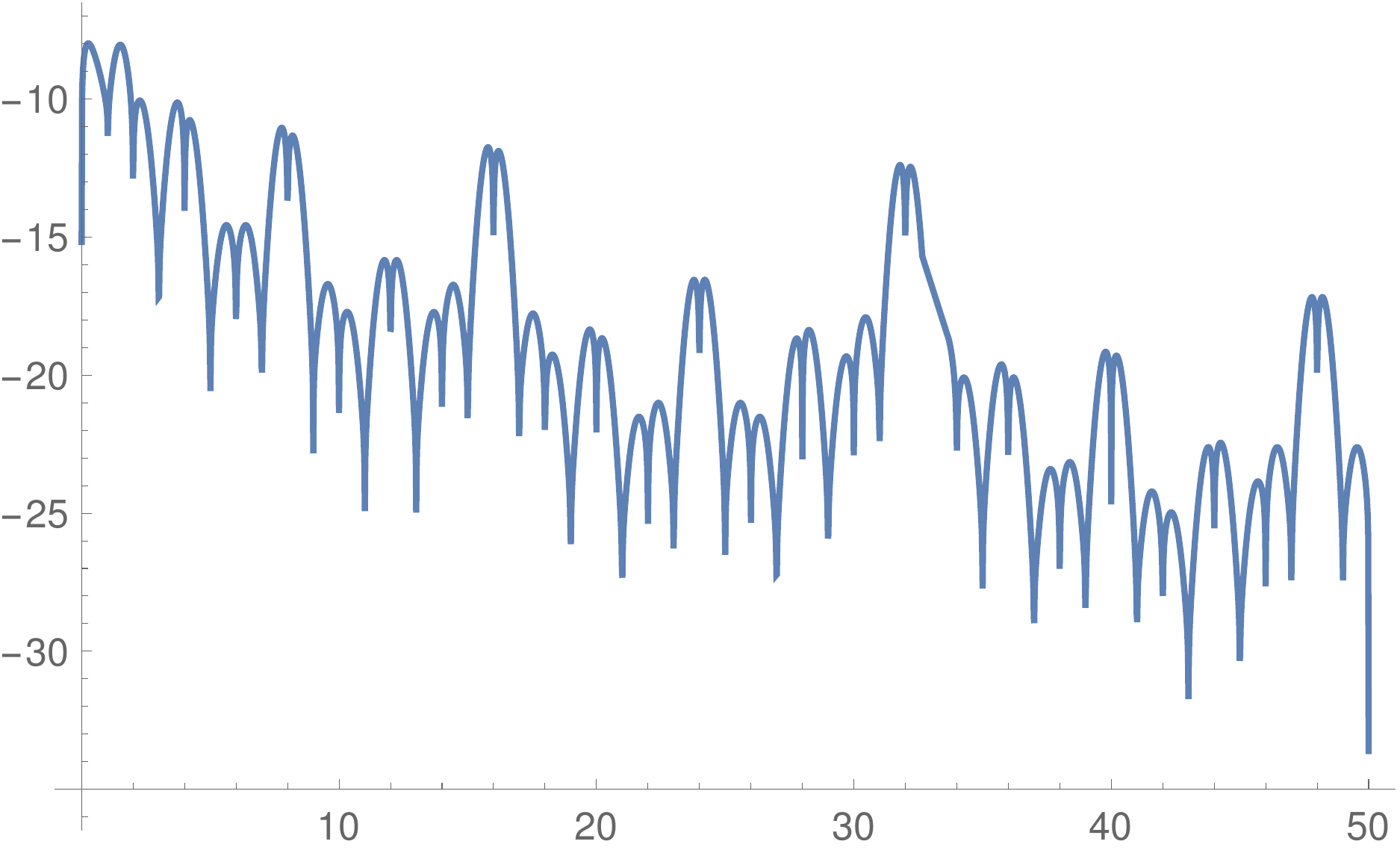}
\par\end{centering}

\caption{The logarithm (base 10) of the difference between the scaling function
defined by the product in \eqref{eq:exact scaling function} and the
Gaussian \eqref{eq:Fourier_scaling_functions} for $\alpha=0.25$
(left) and $\alpha=0.4$ (right) on the interval $\left[0,50\right]$.\label{fig:The difference exact scaling and Gaussian}}
\end{figure}

\begin{rem}
Scaling functions of all MRAs have ``bumps'' in the Fourier domain,
i.e. contributions outside the main band, due to their construction
as an infinite product of periodic functions. As an example, see discussion
in \cite[p. 245]{DAUBEC:1992}. It is particularly easy to see this
phenomenon for the B-splines. The Fourier transform of the non-orthogonal
scaling function forB-splines (of odd degree $m$ ) is given by 
\[
\widehat{\varphi}_{spline}\left(p\right)=\left(\frac{\sin\pi p}{\pi p}\right)^{m+1},
\]
so that the periodic contributions of the numerator outside the main
band are suppressed (but not eliminated) by the denominator. We observe
the same phenomenon in the behavior of $\widehat{\phi}_{exact}$.\selectlanguage{american}%
\end{rem}

\bibliographystyle{plain}

\end{document}